\documentclass[amscd,amssymb,verbatim,12pt]{amsart}

\usepackage[all]{xy}

\usepackage{pdfsync}
\usepackage{xcolor}
\usepackage{graphicx}
\usepackage{subfigure}
\newif\ifAMS
\IfFileExists{amssymb.sty}
  {\AMStrue\usepackage{amssymb}}
  {\usepackage{latexsym}}
  \usepackage{enumerate}
\theoremstyle{plain}

\newtheorem*{Guirardel}{Guirardel's Theorem}

\newtheorem{Thm}{Theorem}[section]
\newtheorem*{Main-nest}{Theorem \ref{T:nesting}}

\newtheorem*{Result}{Theorem \ref{endgame}}
\newtheorem*{angles}{Lemma \ref{L:angles}}
\newtheorem*{overl}{Definition \ref{overlap}}

\newtheorem{Cor}[Thm]{Corollary}
\newtheorem{Lem}[Thm]{Lemma}

\newtheorem{Con}[Thm]{Conjecture}

\theoremstyle{definition}
\newtheorem{Def}[Thm]{Definition}

\theoremstyle{remark}

\newtheorem{Rem}{Remark}

\newtheorem{Ex}[Thm]{Example}

\newcommand{\interior}{^{ \kern-5pt ^\circ}}
\newcommand {\bd}{\partial}

\newcommand {\iy}{\infty}
\newcommand {\cl}{\overline}

\newcommand {\N}{{\mathbb N}}
\newcommand {\R}{{\mathbb R}}
\newcommand {\Z}{{\mathbb Z}}

\newcommand {\E}{{\mathbb E}}

\newcommand{\Int}{\text{Int}\,}
\newcommand {\rad}{\text {Rad}\,}
\newcommand {\nb}{\text{Nbh}\,}
\newcommand {\st}{\text{Stab}\,}
\newcommand {\diam}{\text{diam}\,}
\newcommand {\fp}{\text {Fix}\,}

\newcommand {\hu}{\text {Hull}\,}

\newcommand {\cA}{ {\mathcal  A}}

\newcommand {\cM}{{\mathcal  M}}

\newcommand {\cP}{{\mathcal  P}}
\newcommand {\cR}{{\mathcal  R}}

\begin{document}

\title{Finite cuts and CAT(0) boundaries}

\author{ Panos Papasoglu}

\author
{Eric Swenson }


\address  [Panos Papasoglu]
{Mathematical Institute, University of Oxford, Andrew Wiles Building, Woodstock Rd, Oxford OX2 6GG, U.K.  }
\email {} \email [Panos Papasoglu]{papazoglou@maths.ox.ac.uk}

\address
[Eric Swenson] {Mathematics Department, Brigham Young University,
Provo UT 84602}
\email [Eric Swenson]{eric@math.byu.edu}

%

\begin{abstract} 
We show that if a 1-ended group $G$ acts geometrically on a CAT(0) space $X$ and $\bd X$ is separated by $m$ points then
either $G$ is virtually a surface group or $G$ splits over a 2-ended group. 
In the course of the proof we study nesting actions on $\R $-trees and we show that nesting actions with
non overlapping translation intervals give rise to isometric actions.

\end{abstract}
\thanks{This work was supported by the Engineering and
Physical Sciences Research Council [grant numbers EP/I01893X/1, EP/K032208/1], and by a grant from the Simons Foundation (209403). The second  author would like to thank the Isaac Newton Institute for Mathematical Sciences, Cambridge, for support and hospitality during the programme npc2017 where work on this paper was undertaken.}

\maketitle

\section{ Introduction}

After Gromov's seminal work on hyperbolic groups \cite{GRO} the study of groups
using their boundary and the dynamics of the action on the boundary has
proved very fruitful. Bowditch \cite{BOW} notably showed that JSJ decompositions
of hyperbolic groups are reflected on the boundary and gave some topological
conditions on the boundary that characterize splittings over 2-ended groups.
These results were partially generalized in the case of CAT(0) groups in \cite{PS1}.

A result of Bowditch that does not generalize in the CAT(0) case is the following:

A one-ended hyperbolic group has local cut points if and only
if either the group splits over a 2-ended group or it is a
hyperbolic triangle group. 

Indeed in the CAT(0) case, boundaries are not necessarily
locally connected (e.g. if $G=F_2\times \Bbb Z$, where $F_2$ is
the free group of rank 2, the boundary is a suspension of a Cantor
set, so it is not locally connected). Hence it does not make sense to
talk about local cut points except in some special cases. For example
Haulmark \cite{HAU} showed recently that Bowditch's result holds in the case
of CAT(0) spaces with isolated flats. 

However the somewhat stronger condition of having a finite
cut makes sense for CAT(0) boundaries too. Clearly a point belonging to a finite cut
is a local cut point. 

To state our main result we recall some terminology. An action of a
group $G$ on a space $X$ is called {\em proper} if for every
compact $K \subset X$, the set $\{g \in G : g(K) \cap K \neq
\emptyset \}$ is finite. An action of a group $G$ on a space $X$
is called {\em co-compact} if the quotient space of $X$ by the action of
$G$,  $X /G$ is compact.   An action of a group $G$ on a metric
space $X$ is called {\em geometric} if $G$ acts properly, co-compactly
by isometries on $X$.  It follows that $G$ is quasi-isometric to
$X$, that is they have the same coarse geometry. If $Z$ is a
compact connected metric space we say that $\{a_1,...,a_n\}$ is a {\em finite cut} of
$Z$ if $Z\setminus \{a_1,...,a_n\}$ is not connected.

\begin{Result} Let $G$ be a one ended group acting geometrically on a CAT(0) space $X$. Assume that $\partial X$ has a finite cut.
Then either $G$ is virtually a surface group or $G$ splits over a 2-ended group.
\end{Result} 

We note that CAT(0) groups have a richer JSJ-theory than hyperbolic groups as they may split along $\mathbb Z^n$ for $n\geq 2$.
One would expect that such splittings can be detected from the boundary too. We hope that our present work
that builds splittings from cuts is a step in this direction. More specifically we make the following conjectures:

\begin{Con} Let $G$ be a one ended group acting geometrically on a CAT(0) space $X$. Assume that $\partial X$
has no finite cut. Then no arc of Tits length $<2\pi $ separates $\bd X$.
\end{Con}

\begin{Con} \label{1.2} Let $G$ be a one ended group acting geometrically on a CAT(0) space $X$. Assume that $\partial X$
has no finite cut and that a circle of Tits length $2\pi $ separates $\bd X$. Then either $G$ is virtually $\pi _1S\times \Z$,
where $S$ is a closed surface, or $G$ splits over virtually $\Z ^2$.
\end{Con}

We remark in the context of conjecture \ref{1.2} that it is easy to see that if $G$ splits over $\Z^2$ then a circle of Tits length 
$2\pi $ separates $\bd X$. 

One could formulate similar conjectures for splittings over virtually $\Z^n$ for $n>2$ and also stronger versions of the conjectures above stating them for arcs of finite Tits length and for circles of finite Tits length rather than of length $<2\pi $ and $2\pi$ respectively.

We note that there are already some results along the lines of the above conjectures. V. Schroeder \cite{SCH} (see also \cite{BUY}) has shown that if a $C^{\infty}$ smooth non-positively curved manifold has a codimension 1 flat in its universal cover then there is an immersed
$n$-torus in $M$. Kleiner-Kapovich \cite{KK} have shown similar (and stronger) results in the 3-dimensional case using the boundary. It was shown in \cite{PS1} that a 1-ended group $G$ which is not commensurable to a surface group 
and acts geometrically on a CAT(0) space $X$, splits
over a 2-ended group if and only if there is a pair of points that
separates $\bd X$.

In the course of the proof of our main theorem we show several other results that are interesting in their
own right. Specifically we show that in CAT(0) spaces one can define projections to lines for any angle $\theta $ (rather than just orthogonal projections), we investigate convex subsets and subgroups of CAT(0) groups and we study nesting actions on $\R$-trees. We believe the tools we
develop would be useful for approaching the above conjectures and might prove useful in other contexts as well.

We give now an outline of the following sections:\smallskip

In section 2 we recall briefly the results of \cite{PS2} where we show that if a continuum $Z$ has a finite cut then there is an $\mathbb R$-tree $T$ encoding all minimal finite cuts (\textit{min cuts}) of $Z$. A crucial observation that allows to construct $T$ is that crossing min
cuts lie in `wheels', i.e. compact subsets which have a nice cyclic structure. We call $T$ the \textit{cactus tree} of $Z$. We recall basic
definitions about CAT(0) boundaries and we state our main technical tool for studying the dynamics of the action of $G$ on $\bd X$ (theorem
\ref{piconverge}, `$\pi$-convergence').

We further show (theorem \ref{bigtits}) that if $X$ is a CAT(0) space and its boundary $\bd X$ has a `large' cactus tree $T$ then there is a subcontinuum of $\bd X$ lying between two finite cuts that has Tits radius at least equal to $\pi $. Here by large we mean that there is a linearly ordered subset of min cuts of $T$ that has at least 8 elements. The proof is somewhat technical. To illustrate the difficulty we explain in example \ref{not4} that our argument does not work if $T$ contains just 4 linearly ordered min cuts.
We note that subsets of large Tits radius play an important role in the sequel as they allow us to control the dynamics of the action on the boundary. \smallskip

In section 3 we assume that a group $G$ acts geometrically on a $CAT(0)$ space $X$ and we study convex subsets of $X$ and convex subgroups of $G$. 
We use these results later on to control the stabilizers
of the action of $G$ on the cactus tree that we obtain from $\partial X$. Some of these results are interesting in their own right and fill gaps
in the existing literature. In particular we show that a convex subgroup acts geometrically on a boundary minimal convex subset(Corollary \ref{C:min for subgroup}). From the machinery
we develop we obtain also as a corollary that a subgroup containing a finite index subgroup that is convex is convex itself (Corollary \ref{C:convex}).
We also show that for any subset $A$ of $\partial X$ if $\fp A$ is finitely generated then $\fp A$ is convex (Lemma \ref{L:fixfg}). \smallskip

After these sections we assume that a 1-ended group $G$ acts geometrically on a CAT(0) space $X$ and $\bd X$ has a finite cut.
It follows that $G$ acts on the cactus tree $T$ of $\bd X$. Now the proof splits in 2 cases: the case where the action is non-nesting
(sections 4,5) and the case where the action is nesting (sections 6,7). 

A major difficulty in both cases is that the dynamics of the action of $G$ on $\bd X$ is not well understood. Ballmann-Buyalo conjecture in
\cite{BAL-BUY} that if the Tits-diameter of $\bd X$ is greater than $\pi $ then $G$ is rank 1. In our context we show by a case by case
analysis that $G$ is rank 1 using $\pi$-convergence. \smallskip

In section 4 we show that if $G$ acts on $T$ non-nestingly then $G$ is rank 1. As a first step we show that $G$ acts on $T$
without fixed points. To do this we distinguish cases. If $G$ fixes a point $a$ of $T$ that corresponds to a min cut then
$G$ is virtually of the form $H\times \Z$ and in this case our main theorem follows easily. In the case of a point $a$ that is
a wheel we use the results of section 2 to show that $a$ contains subsets of large Tits-radius. Using $\pi$-convergence it follows that
$a$ can not be fixed. The case of median points is somewhat more technical but the proof is similar to arguments in \cite{PS1}. From these it
follows that $G$ acts on $T$ without fixed points. To conclude that $G$ is rank 1 we need a slightly stronger property that implies that the Tits diameter of $\bd X$ is infinite. Namely we show that that
there is an axis of translation of an element $g\in G$ on $T$ containing two min cuts $a,ga$ such that $a\cap ga=\emptyset $. We deduce this using
a result of P.M. Neumann on infinite permutation groups (lemma \ref{permutation} and theorem \ref{rank1}).  \smallskip

The objective of section 5 is to show that, under our assumptions, either $G$ splits over a 2-ended group or $G$ acts non-nestingly
on an $\R$-tree $T$ with virtually cyclic arc stabilizers. This splits naturally in two cases, the cactus tree is either discrete or
non-discrete. In the discrete case vertices of the tree correspond either to min cuts or to wheels. We analyse the stabilizers of both.
To deal with min cuts we generalize orthogonal projections. Specifically we prove the following:
\begin{angles}  Let $L:\R \to X$ be a geodesic line and $\theta \in (0, \pi)$.   We will abuse notation and refer to $L(w)$ as $w$, so we refer to $L(-\iy)$ as $-\iy$.  
If $\phi:X-L \to L$ is defined by $\phi(z) = \sup \{w \in L : \angle_w(z, -\iy) > \theta\}$ then $\phi$ is continuous, and $\angle_{\phi(z)}(z, -\iy) \ge \theta$.
\end{angles}

Since $G$ is 1-ended the stabilizer of a min cut is not finite. Using results in section 3 we show that if the stabilizer contains a hyperbolic element $g$ then $G$ splits over a 2-ended group (Theorem \ref{T:cut pair}). It turns out that in this case one needs to understand the geometry of $X$. We show using projections that a finite neighborhood of an axis of $g$ separates $X$, so $<g>$ is a codimension 1 subgroup (Theorem \ref{T:cut pair}). The result then follows by the Algebraic Torus Theorem of Dunwoody-Swenson \cite{D-S}. 
Otherwise, using results of section 3 and $\pi$-convergence, we show
that almost all edges adjacent to a an element of $\mathcal R$ are stabilized by virtually cyclic subgorups. In the case $T$ is discrete this implies that $G$ splits over a 2-ended group.
In the case the tree $T$ is non-discrete we show that by collapsing some arcs that are not stabilized by a virtually cyclic group we obtain an $\R$-tree $\bar T$ and $G$ acts on
$\bar T$ non-nestingly with virtually cyclic arc stabilizers.

In section 6 we turn our attention to nesting actions. The theory of isometric actions on $\mathbb R$-trees due to Rips and Bestvina-Feighn \cite{B-F} has been a fundamental tool in recent
developments in Geometric Group Theory. This theory shows that -under some natural conditions- if a group acts on an $\mathbb R$-tree it
also acts on a simplicial tree, so that Bass-Serre theory can be applied to understand the group structure.
In several contexts one obtains an action of a group on an $\mathbb R$-tree  by homeomorphisms rather than isometries
and Levitt \cite{Le} has shown that actions by non-nesting homeomorphisms can be promoted to isometric actions,
so that the general theory applies to this setting as well.

In this section we show that sometimes one can obtain an isometric action on an $\R$-tree
tree even when the group acts on the $\mathbb R$-tree by nesting homeomorphisms. We describe our result now in some detail.

Recall that if $T$ is an $\mathbb R$-tree and
$g:T\to T$ is a homeomorphism we say that $g$ is {\em nesting} if there is some interval $[a,b]$ such that $g([a,b])\subseteq (a,b]$.
If $g$ is a nesting homeomorphism then there is a set of disjoint open intervals of $T$ such that $g$ acts on each such interval
by `translations' (see lemma \ref{L:int}). These intervals are called {\em intervals of translation of } $g$. We have the following:

\begin{overl} Let $G$ be a finitely generated group acting nestingly on an $\R$-tree $T$. If for any two 
 intervals of translation $I,J$ with $I\cap J\ne \emptyset $  no endpoint of $I$ is contained in $J$ we say that $G$ acts on $T$ with \textit{non-overlapping translation intervals}.
\end{overl}
 Assuming that the intervals of translation of the nesting elements of $G$ do not overlap one obtains a non-nesting action action on an $\R$-tree. 
We state now our result-we refer the reader to section 4 for the definition of cross-components.
 \begin{Main-nest} Suppose that a finitely generated group $G$ acts minimally and nestingly on the $\mathbb R$-tree $T$ with non-overlapping translation intervals.
 Then $G$ acts non-nestingly on an $\R$-tree $S$ (without proper $G$-invariant subtree) with every arc stabilizer of $S$  stabilizing the end of a cross-component of $T$.

\end{Main-nest}

In section 7 we assume that the action of $G$ on its cactus tree $T$ is nesting. Our strategy here is similar to the one in section 4:
We show that
there is an axis of translation of an element $g\in G$ on $T$ containing two min cuts $a,ga$ such that $a\cap ga=\emptyset $. It is considerably
more complicated to find this axis in the non-nesting case but we deduce this again using
Neumann's lemma (lemma \ref{permutation}) and $\pi $-convergence (theorem \ref{rank2}). This implies that $diam _T(\bd X)$ is infinite which in turn implies that $G$ is rank 1.
Using this and $\pi $-convergence again we show that the intervals of translations do not overlap and that stabilizers of ends of cross-components are virtually cyclic.
This allows us to obtain a non-nesting action on an $\R$-tree with virtually cyclic arc stabilizers (Theorem \ref{thick}).  \smallskip

In section 8 our starting point is the theorems \ref{nonnest}, \ref{thick} of sections 5 and 7. In both cases we obtained from the action of $G$ on its cactus tree a non-trivial
non-nesting action of $G$ an an $\R$-tree $\bar T$ with 2-ended arc stabilizers. We promote this to an action by isometries using a result of Levitt \cite{Le}. Bestvina-Feighn-Rips theory does not imply directly that $G$ splits over a 2-ended group.
However one may apply a refinement of this theory due to Sela and Guirardel. According to Guirardel's theorem the action of $G$ can be decomposed as a graph of actions of
a relatively simple type: Seifert type, axial type and simplicial type. Clearly if we have non-trivial simplicial components $G$ splits over a 2-ended group. One may rule out
easily axial type action as well so we are left with the case of simplicial pieces that reduce to points and Seifert type pieces (Theorem \ref{T:Rtree}). We note that the stabilizers
of Seifert pieces are of the form $H\times \mathbb Z$ with $H$ virtually free. So their limit set is a suspension. Using this geometric observation, $\pi $-convergence and the way the tree
$\bar T$ was constructed, we show that there are no Seifert type pieces, hence $G$ splits over a 2-ended group (Theorem \ref{endgame}).
 
We would like to thank G. Levitt and V. Guirardel for explaining their works, \cite{Le} and \cite{Gui} respectively, to us.

\section{ Cactus trees and CAT(0) groups}

In this section we collect background and useful facts about cactus trees and boundaries of CAT(0) groups.
We recall first some basic facts about pretrees which will be used throughout this paper. 

\subsection{Pretrees}

Pretrees are tree-like structures used to construct trees or $\R$-trees (see \cite{BOW5}). 
In general every `reasonable'
pretree gives rise to an $\mathbb R$-tree. Unfortunately the definition of `reasonable' is somewhat technical, we make it
precise in this section.
\begin{Def}
Let $\cP$ be a set and let $R\subset \cP\times \cP \times \cP$. We
say then that $R$ is a \emph{betweeness} relation. If $(x,y,z)\in
R$, then we write $xyz$ and we say that $y$ is between $x,z$. $\cP$
equipped with this betweeness relation is called a \emph{pretree}
if the following hold:

1. there is no $y$ such that $xyx$ for any $x\in \cP$.

2. $xzy \Leftrightarrow yzx$

3. For all $x,y,z$ if $y$ is between $x,z$ then $z$ is not between
$x,y$.

4. If $xzy$ and $z\ne w$ then either $xzw$ or $yzw$.

\end{Def}
The obvious example of a pretree is a tree (simplicial or
$\R$-tree). Note that any subset of a pretree is a pretree.

\begin{Ex} Not all pretrees are subsets of $\R$-trees. Indeed
any linearly ordered set $(P,<)$ can be seen as a pretree: we
define betweeness by: $xyz$ if $x<y<z$ or $z<y<x$. Consider now
the first uncountable ordinal $\aleph _1$. Clearly $\aleph _1$ can
not be embedded in an order preserving way to an $\R$-tree.
\end{Ex}

\begin{Def} Let $\cP$ be a pretree and let $x,y\in \cP$.
We define the \emph{open interval} $(x,y)$ to be the set of all
$z\in \cP$ between $x,y$. Similarly we define the \emph{closed
interval} $[x,y]$ and the half open intervals $[x,y),\,(x,y]$.
\end{Def}

\begin{Def} A point $x$ of a pretree $\cP$ is called a terminal point if $x \not \in (a,b)$ for all
$a,b \in \cP$.
\end{Def}

\begin{Def} A subset $I$ of a pretree $\cP$ is called \emph{linearly ordered}, if for each distinct triple of points in $I$,
one of them is contained in the open interval between the other
two.
\end{Def}
It is shown in \cite{BOW5} that every linearly ordered subset
comes with a linear order (and its opposite) defined  using  the
betweenness relation. 

\begin{Rem}\label{R:order} It is useful to note that if $\cP$ is a pretree, one
can recover the betweeness relation on $\cP$ from the set of
maximal linearly ordered subsets of $\cP$. To be precise, $xyz$
holds in $\cP$ if and only if $x,y,z$ lie in a maximal linearly
ordered subset of $\cP$ and $xyz$ holds in this subset. So one way
to define betweeness in a pretree $\cP$ is by specifying all
maximal linearly ordered subsets of $\cP$.
\end{Rem}
\begin{Def} If every maximal linearly ordered subset of a pretree $\cP$ is order isomorphic to an interval of $\R$, 
then we say that $\cP$ is an $\R$-tree.  Notice that this doesn't define a topology on $\cP$.
\end{Def}

\begin{Def} For $I$ a linearly ordered subset of a pretree $\cP$, we say that $I$ is \emph{complete} if $I$ has the supremum property,
that is if every nonempty subset of $I$ with an upper bound in $I$
has a supremum in $I$. If every maximal linearly ordered subset of
$\cP$ is complete, then we say $\cP$ is \emph{complete}.
\end{Def}
\begin {Def} For $\cP$ a pretree and $A \subset \cP$, the convex hull of $A$ is defined by $$\hu(A) = \bigcup\limits_{a,b \in A} [a,b]$$
We say that a set $A\subset \cP$ is {\em convex} if $A = \hu (A)$.  Maximal linearly ordered sets are of course convex.
\end{Def}

\begin{Def} A subset $A$ of a pretree  $X$ is \emph{predense} in $X$ if for every distinct $a,b\in X$,
$[a,b]\cap A \neq \emptyset$.

Let $\cP$ be a pretree; A maximal linearly ordered  $I\subset \cP$
is called {\em preseparable} if $I$ has a countable predense
subset.   A pretree is {\em preseparable} if every maximal
linearly ordered  subset in it is preseparable.
\end{Def}

\begin{Def}
Let $\cP$ be a pretree. We say that $\cP$ is a \textit{median}
pretree if for any three points $x,y,z\in \cP$ the intersection
$[x,y]\cap [y,z]\cap [z,x]$ is non-empty. Note that if this
intersection is non-empty then it consists of a single point
called the \textit{median} of $x,y,z$.
\end{Def}

\begin{Ex} Consider the $\R$-tree $T$ given by the union of
$x,y$-axes in the plane. The subset $\cP$ of $T$ consisting of the
intervals $(-\infty, -1]\cup (0,\infty )$ of the $x$-axis together
with the interval $(0,\infty )$ of the $y$-axis is not median. It
becomes median if we add $0$.
\end{Ex}

Clearly if a pretree $\cP$ embeds in an $\R$-tree, then $\cP$
is preseparable. So being preseparable is a necessary condition
for an embedding to an $\R$-tree. However this condition is
not sufficient (see \cite{PS2}, example 5.3).

%
%

\begin{Def}
Let $(L,<)$ be a linearly ordered set. If $A,B\subset L$, we write
$A<B$ if $a<b$ for all $a\in A,b\in B$.

$(A,B)$ is a \emph{Dedekind cut} of $L$ if $A,B$ are non-empty,
$L=A\cup B$ and $A<B$.
\end{Def}

\begin{Def}
Let $\cP$ be a pretree and let $(L,<)$ be a maximal linearly
ordered subset of $\cP $. We say that $L$ has a \emph{gap} at
$x\in L$ if one of the following two holds:

i) $(A,B)$ is a Dedekind cut of $L$, $x=\sup A$ lies in $A$ and
there is a linearly ordered subset $C$ of $\cP$ such that $A\cup
C$ and $B\cup C$ are maximal linearly ordered subsets of $\cP$.

ii) $(A,B)$ is a Dedekind cut of $L$, $x=\inf B$ lies in $B$ and
there is a linearly ordered subset $C$ of $\cP$ such that $A\cup
C$ and $B\cup C$ are maximal linearly ordered subsets of $\cP$.

If every maximal linearly ordered subset of $\cP$ has at most
countably many gaps, then we say that $\cP$ has \emph{few gaps}.
\end{Def}

We can now state a theorem of \cite{PS2} characterizing pretrees that can be `completed' to give rise
to $\mathbb R$-trees:

\begin{Thm}\label{T:PRtree}
Let $\cP$ be a pretree. Then there is an embedding of $\cP$  into
an $\R$-tree if and only if $\cP$ is preseparable and has few
gaps.
\end{Thm}

\subsection{Cactus trees}
\begin{Def} Let $Z$ be a continuum. We say that $Z$ is $m$-{\em thick} if for any $A\subseteq Z$ with $|A|<m$, $Z\setminus A$ is connected
while for some $A$ with $|A|=m$, $Z\setminus A$ is not connected. If $Z$ is $m$-thick and $A\subseteq Z$ is such that $Z\setminus A$ is not connected
and $|A|=m$ then we say that $A$ is a {\em min cut} of $Z$.
\end{Def}
We have used in \cite{PS2} also the terms minimal separators and $m$-cuts for min cuts. However we use the term min cut here because of the 
affinity of the subject with graph theory (see \cite{DKL}, \cite{FF}).

The results of \cite{PS2} assert roughly that the min cuts of an $m$-thick continuum have an $\mathbb R$-tree structure. This is rather easy to see
for $m=1$ (modulo technicalities) however it is more delicate for $m\geq 2$. The difficulty comes from min cuts `crossing' each other and this difficulty
is surmounted by observing that crossing min cuts lie in `wheels'. To obtain an $\mathbb R$-tree one starts from a pretree consisting of min cuts that
do not cross any other min cut, and wheels. We explain now this construction in some detail.

\begin{figure}[h]
\includegraphics[width=4.5in ]{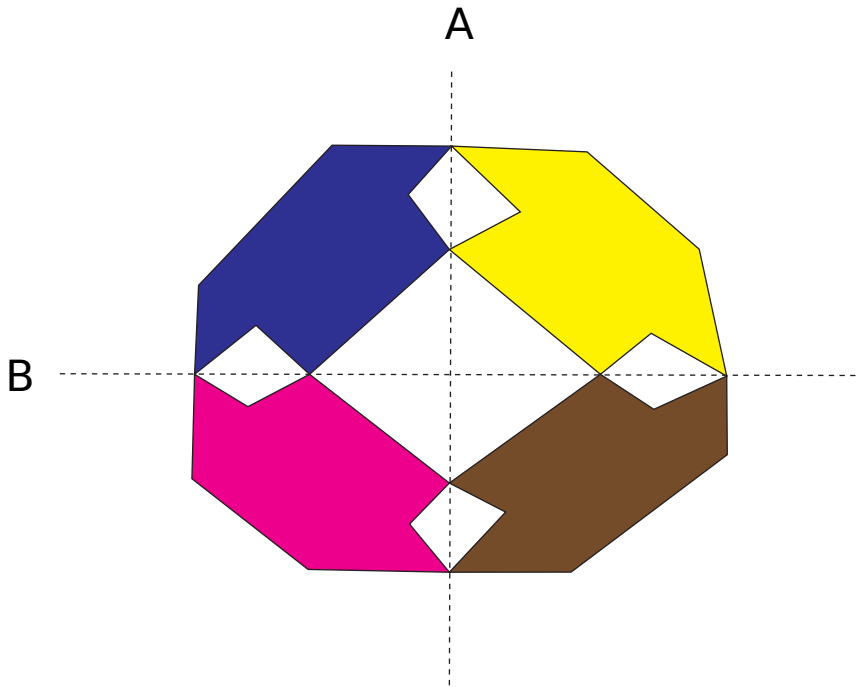}
\vspace{0.1in}
 \caption{Crossing min cuts}
 \end{figure}
 
 \begin{figure}[h]
\includegraphics[width=4.5in ]{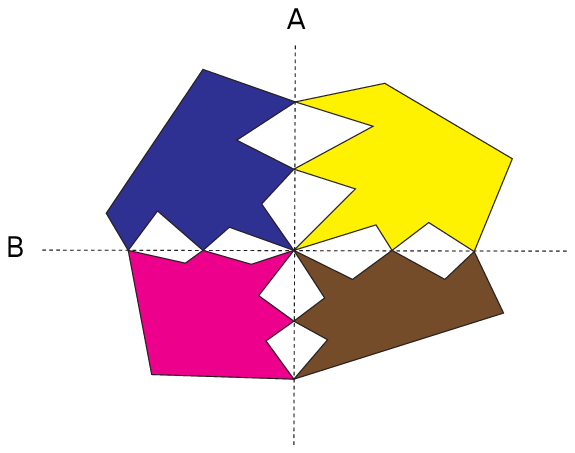}
\vspace{0.1in}
\caption{Crossing min cuts with non-empty intersection}
\end{figure}

\begin{Def} Let $Z$ be a continuum, let $A$ be a min-cut and let $B\subseteq Z$.
We say that  $A$ \textit{separates} $B$
if $B$ intersects at least two components of $Z-A$.
\end{Def}
It turns out that if $A,B$ are min cuts and $A$ separates $B$ then
$B$ separates $A$ too (see lemma 3.4 of \cite{PS2}). We say then that $A$ \textit{crosses} $B$.

\begin{Def}  Let $Z$ be a continuum and $A \subseteq Z$ finite.  If there are non-singleton continua $Y_1, \dots Y_r$ with
\begin{itemize}
\item $\bigcup\limits_i Y_i = Z$. \item $\bigcup\limits_{i\neq
j}Y_i \cap Y_j = A$.
\end{itemize}
Then we say that that $A$ {\em decomposes} $Z$ (into $Y_1, \dots
Y_r$).  Notice that $\bd Y_i= Y_i \cap A$ for all $i$.
\end{Def}

\begin{Def} \label{D:wheel} Let $Z$ be an $m$-thick continuum, where $m\geq 2$.
A finite set $W \subseteq Z$ is called a {\em wheel }
if  $W$ decomposes $Z$ into continua $M_0, \dots M_{n-1}$  with $n >3$ satisfying the following:
\begin{itemize}
\item There is a (possibly empty) $I=\cap M_i$   called the \textit{center} of the wheel with $|I|<m$ and $M_i \cap M_j=I$ for all $i-j \neq \pm 1 \mod n$.
\item For each $i$, $|\bd M_i|= m$.
\end{itemize}

The collection $M_0, \dots M_{n-1}$ is called the  {\em wheel decomposition}
of $Z$ by $W$, and $M_i, M_{i+1}$ are called an {\em adjacent pair} of the wheel decomposition. 
The intersection of an adjacent pair $M_i \cap M_{i+}$ is called a {\em half-cut} of $W$.   This decomposition is unique.

\end{Def}

\begin{figure}[h]
\includegraphics[width=4.5in ]{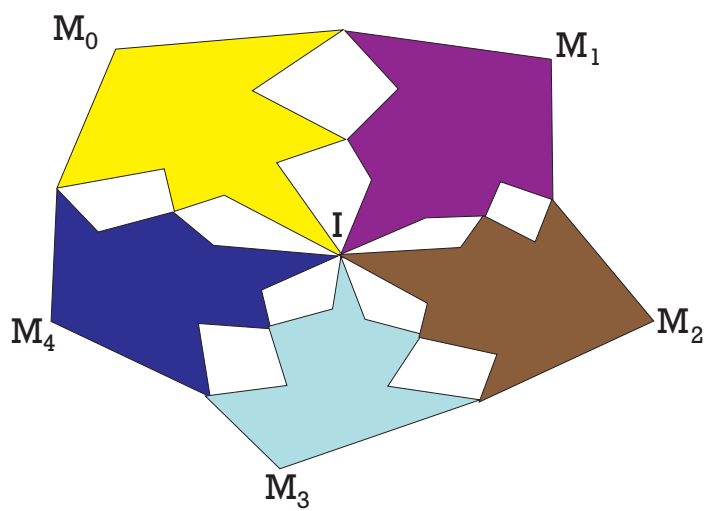}
\vspace{0.1in}
 \caption{A wheel of a 5-thick continuum}
\end{figure}

By
\cite{PS2}
\begin{itemize}
\item  The union of two distinct half-cuts of $W$ is a min cut.
\item Every half cut has the same number of points $q$
 \item $m = 2q-|I|$.
 \end{itemize}

For wheels $W$ and $V$, we say that $W$ is a \emph{sub-wheel} of
$V$ if $W\subseteq V$. By \cite{PS2}
\begin{itemize}
\item Every continuum of the wheel decomposition of $V$ will be contained in one of the continua in the  wheel decomposition of $W$.
\item The center of $W$ is the same as the center of $V$.
\item The half-cuts of $W$ are half-cuts of $V$.
\end{itemize}
If $W$ is an infinite subset of $Z$ and every finite subset of $W$
 is contained in  some  finite  wheel contained in $W$,
then we say that $W$ is {\em a wheel}. The half-cuts of $W$ are all of the half-cuts of all finite sub-wheels of $W$.

Recall that a pretree is a set with a
betweeness relation. The basic example of a pretree is the set of
vertices of a tree where betweeness is defined in the obvious way.
It turns out that given a pretree (satisfying certain conditions)
one can construct a canonical  envelopping tree. 

If a min cut $K$ does not cross any other min cut we say that it is \textit{isolated} or \textit{inseparable},
otherwise it is called \textit{separable}. It is shown in \cite{PS2} that every min cut of $Z$ is either contained in
a maximal wheel or it does not cross any other min cut, so separable min cuts are always contained in wheels.


\begin{figure}[h]
\includegraphics[width=4.5in ]{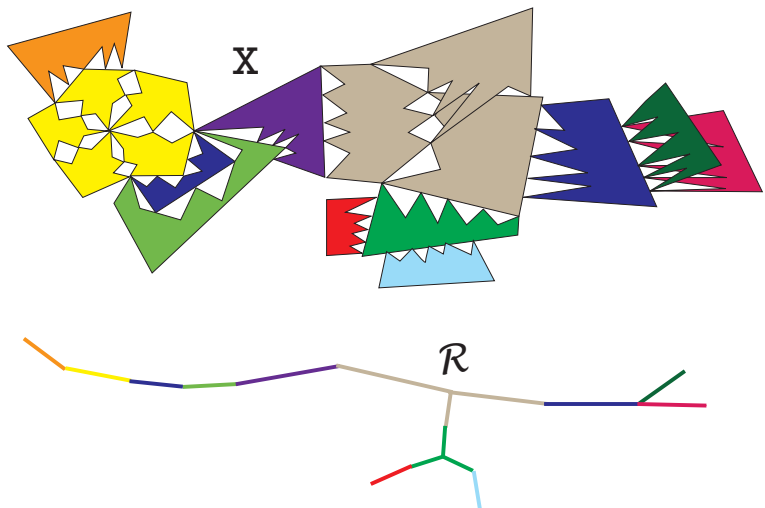}
\vspace{0.1in}
 \caption{A 5-thick continuum containing a wheel, and the corresponding Tree}
\end{figure}

We define a pretree $\cR$ with elements the maximal wheels of
$Z$ and the isolated min cuts of $Z$. We define a
betweeness relation in $\cR$:

Let $x,y,z\in \cR$ distinct. For $y$  a min cut, we say that $y$
is between $x,z$ if there are continua $A,B$ such that
$$x\subseteq A,\, z\subseteq B,\, A\cup B=Z,\,A\cap B=y$$
If $y$ is a maximal wheel we say that $y$ is between $x,z$ if for
some min cut $w\in y$, $w$ is between $x,z$. We can now
state the main theorem of \cite{PS2} 

\begin{Thm} \label{tree} Let $Z$ be a
separable, $m$-thick continuum, where $m>1$. Then the pretree $\cR
$ canonically embeds into an $\R$-tree $T$.  An action on $Z$ by homeomorphisms induces an action on $T$ by homeomorphisms.

\end{Thm}

\subsection{Boundaries of CAT(0) groups}
Recall that a geodesic is an isometric embedding of an interval of $\R$.  
 For $X$ a geodesic metric space and $\Delta(a,b,c)$ a geodesic triangle in $X$ with vertices $a,b,c \in X$ there is a Euclidean  {\em comparison }triangle $\bar \Delta=\Delta(\bar a,\bar b, \bar c) \subseteq \E^2$ with 
$d(a,b) =d(\bar a, \bar b)$, $d(a,c) = d(\bar a, \bar c)$ and $d(b,c)=d(\bar b, \bar c)$.  
We define the Euclidean comparison angle $\bar \angle_a(b,c) =\angle_{\bar a}(\bar b,\bar c)$.  

Each point $z \in \Delta(a,b,c)$ has a unique comparison point, $\bar z \in \bar \Delta$.  We say that the triangle $\Delta(a,b,c)$ is CAT(0) if for any $y, z \in \Delta(a,b,c)$ with comparison points $\bar y, \bar z \in \bar \Delta$, $d(y,z) \le d(\bar y,\bar z)$. The space $X$ is said to be CAT(0) if every geodesic triangle in $X$ is CAT(0).   

For $X$ a proper CAT(0) space, the visual boundary of $X$, $\bd X$ is the set of equivalence classes of geodesic rays,  where two rays $\alpha: [0,\iy) \to X$ and $\beta :[0,\iy) \to X$ are equivalent if $d(\alpha(t), \beta(t))$ is bounded.  Fix a base point $x_0 \in X$.    It is the only easy application of Arzela-Ascoli to show that any  geodesic ray is equivalent to unique geodesic ray $\alpha:[0,\iy) \to X$ from  $x_0$ (that is  $\alpha(0)=x_0$.)
We define the visual topology on $\bar X = X\cup \bd X$ by taking as neighborhoods, the open balls of points in $X$, and for $\alpha$ geodesic ray from $x_0$  and $\epsilon, n >0$,
$U(\alpha, n,\epsilon) =\{ p \in \bar X:\, d(\alpha(n), \beta(n)) < \epsilon ,\text{ where } \beta \text{ is the geodesic from $x_0$ to } p \}.$  It follows that $\bar X$ is a Hausdorff compactification of $X$ and $\bd X$ is compact metrizable.  

There is also a different metric on $\bd X$. For $\alpha, \beta$ geodesic rays from $x_0$, we define
$$\angle(\alpha,\beta) = \lim\limits_{t \to \iy} \bar \angle_{x_0}(\alpha(t),\beta(t))$$ which is independent of $x_0$.  This gives a metric (generally with a finer topology) on $\bd X$, and from this angle metric we define the corresponding path metric $d_T$ which we call the Tits metric.  When $\angle(\alpha,\beta) < \pi, \,\,\angle(\alpha,\beta)= d_T(\alpha,\beta)$. However  the angle is bounded above by $\pi$, while generically the image of $d_T$ is $[0,\iy]$.

An isometry $g$ of a proper CAT(0) space $X$ is called {\em hyperbolic} if $g$ acts by translation on a geodesic line $L$, which we call an \textit{axis} of $g$.  The ends of the axis $L$ represent two points in $\bd X$ which we denote $g^+$ and $g^-$, where  $g^+$ is the end of $L$ in the direction of $g$ translation.
Ballmann's dichotomy \cite{BAL} tells us that $d_T(g^+,g^-) = \pi$ or $g^\pm$ are isolated (at infinite distance from every other point) in the Tits metric.  In the latter case we refer to $g$ as a rank 1 hyperbolic element.  If a group of isometries $G$ of $X$ contains a rank 1 hyperbolic element then we say $G$ is a rank 1 group of isometries.  
For $G$ acting geometrically on the proper CAT(0) space $X$, then either $G$ is rank 1 and the Tits diameter $\diam_T \bd X = \iy$ or  $\diam_T \bd X < \frac {3\pi}2$  (see \cite{PS1}, \cite{Gu-S}).

\begin{Def} For $H$ acting properly by isometries on the proper CAT(0) space $X$, we define the limit set of $H$, $\Lambda H$ to be limit points of an orbit of $H$ in $\bd X$.  That is for any $x \in X$,  $\Lambda H = \bd X \cap \cl{Hx}$
\end{Def}

Let $X$ be a  proper CAT(0) space and let $G$ be a group acting properly on $X$. Then
$G$ acts on the compactum $\bd X$ by homeomorphisms. Even though the action of $G$ on $\bd X$ could be trivial
e.g. when $G$ is abelian, it has a convergence type property similar (but weaker) to an action
of a hyperbolic group on its boundary.  This convergence property is called $\pi$-convergence (see \cite{PS1}, \cite{Gu-S}). We will use in this paper only a special case of this that we state below:

\begin{Thm}\label{piconverge}  Let $X$ be a proper CAT(0) space and $G$ a group acting
properly on $X$.  For any sequence of distinct group elements of
$G$, there exists a subsequence $(g_i)$ and points $n,p \in \bd X$
such that for any  $x \in \bd X$ with $d_T(x,n)\geq \pi$
, $g_i(x) \to p$.\end{Thm}
In fact in the theorem above  replacing $g_i$ by $g_i^{-1}$ reverses the roles
of $n,p$.

{\bf For the remainder of this exposition $X$ will be a proper one ended CAT(0) space, $G$ will be a group acting geometrically on $X$. We set}  $Z = \bd X$ with the visual topology (not the Tits topology).  Since $X$ is one-ended, $Z$ is a metrizable continuum.

\begin{Def}Let $U \subseteq Z$. We denote by $d_T^U$ the path metric on $U$ associated to $d_T$ restricted to $U$.  For $a\in U$, $B_T^U(a,\epsilon)$ is the open  $d_T^U$ ball in $U$ about $a$ of radius $\epsilon$, and $\bar B_T^U(a,\epsilon)$ is the closed $d_T^U$ ball in $U$ about $a$ of radius $\epsilon$. 
\end{Def}

For $A \subseteq Z$,  will denote by Int $A$ the interior  of $A$ (as a subset of $Z$) and by $\partial A$ its frontier, so $\partial A=\bar A \setminus $Int $A$.

\begin{Lem}\label{L:limit path} Let $A$ be a closed subset of $\bd X$ and $\gamma_n:[0, \epsilon_n] \to A$ a sequence of paths from $a = \gamma_n(0)$ to $b_n =\gamma_n(\epsilon_n)$, with $b_n \to b $ and $ \epsilon_n \to \epsilon$.  Then there is a 1-Lipschitz path 
$\gamma:[0,\epsilon] \to A$ from $a$ to $b$.  
\end{Lem}
\begin{proof}
Choose a non-principal ultra filter $\omega$ on $\N$.  Now $\forall t \in [0,\epsilon)$,
we can define $$\gamma(t) =\lim\limits_{n \to \omega} \gamma_n(t)\in \bd X.$$ This defines $\gamma:[0,\epsilon) \to \bd X$.  Since for each $n$, $\gamma_n$ was 1-Lipschitz, so is $\gamma$.    Since $A$ is closed  $\gamma:[0,\delta) \to A$, and since $\gamma$ is 1-Lipschitz we can extend  to 1-Lipschitz $\gamma:[0,\epsilon] \to A$ .   Clearly $\gamma(0)=a$.   It follows from 1-Lipschitz, that $\gamma(\epsilon) = b$.
\end{proof}

\begin{Lem}\label{L:finite} Let $Y$ be a sub continuum of $Z$ with $\bd Y$ finite.  If $G$ is not rank 1, then $d_T^Y < 3\pi(|\bd Y| +1)$. 
\end{Lem}
\begin{proof}We may assume $Y \neq Z$.  Since $G$ is not rank 1, $\diam_T(Z) < \frac {3\pi} 2$.  Every Tits geodesic from the interior of $Y$ to the exterior of $Y$ must pass through $\bd Y$. Since for $z \in Z-Y$, $Y \subseteq B_T(z, \frac {3\pi} 2)$, it follows that $$Y \subseteq \bigcup\limits_{a \in \bd Y} \bar B_T^Y(a, \frac {3\pi} 2).$$  Since $Y$ is connected and $\bar B_T^Y(a, \frac {3\pi} 2)$ is compact for all $a \in \bd Y$, the nerve of the cover $$\{\bar B_T^Y(a, \frac {3\pi} 2)|\, a \in \bd Y \}$$ must be connected and it follows that 
$d_T^Y < 3\pi(|\bd Y| +1)$.  
\end{proof}

\subsection{Big sets and cactus trees}
 If $Z$ has a finite cut then we can associate a cactus tree $T$ to $Z$.
The main tool that we use to understand the action of $G$ on $T$ is $\pi$-convergence.  In order to apply $\pi$-convergence we need
sets with big Tits radius. The most obvious example of such a set is a non-contractible simple closed curve. 
A similar example
is a subcontinuum $N$ where some finite cut of $N$, which is minimal and contains at least 2 points, is contained in a min-cut of $Z$.  By minimal here we mean that no
subset of the cut is a cut.

We are able to show that
such a subcontinuum exists between `sufficiently many' linearly ordered min cuts in any $m$-thick continuum for $m\geq 3$. We apply this somewhat technical result about continua
to obtain sets of big Tits radius in $CAT(0)$ boundaries.

\begin{Def} For $A \subseteq Z$,  the (extrinsic)  {\em Tits radius }of $A$ is,   $$\rad A = \inf\{\epsilon : A \subseteq B_T(z,\epsilon)\,\text{for some}\, z\in \bd X\}$$
\end{Def}
\begin{Lem}\label{C:Stab} Let  $A$ is be a closed subset of $Z$ with $H<G$ stabilizing $A$.  
If $\rad A \ge \pi$ then $\Lambda H \subseteq A$. 
\end{Lem}
\begin{proof} Let $p\in \Lambda H$ and  $(h_k) \subseteq H$ with $h_kx \to p $ for some $x \in X$, and with $h_k^{-1}(x) \to n \in \bd X$.  Since $\rad A \ge \pi$, there exists $a \in A$ with $d_T(a,n) \ge \pi$.
By $\pi$-convergence $h_k a \to p$.  Since $H$ stabilizes the closed set $A$,  $p \in A$.   Thus by definition $\Lambda H \subseteq A$. 
\end{proof}


\begin{Lem}\label{L: bigpi} Let $Q$ be a subcontinuum  in $Z $. If $M_1\cup M_2$ is a decomposition of $Z$ by the finite cut $J$ which contains no cut point of $Q$, and  with $Q \not \subseteq M_i$  for $i =1,2$, 
then $\rad Q \geq \pi $.
\end{Lem}

\begin{figure}[h]
\includegraphics[width=6.0 in ]{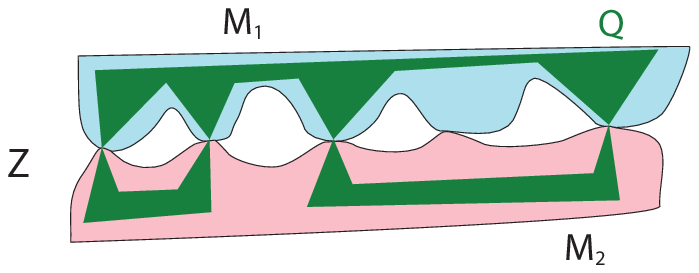}
\vspace{0.1in}
 \caption{}
\end{figure}
 
\begin{proof} Fix  $a\in Z$. 
 It suffices to show that there is $b \in Q$ with $d_T(a,b) \ge \pi$.    Let $c\in Q$ and  $p$ be  a Tits geodesic  from  $a$ to $c$.  We can write $p$ as a concatenation of Tits geodesics
$p=p_1\star p_2...\star p_n$ where each $p_i$ is contained in $M_1$ or $M_2$ and if $p_i\subseteq M_j$ then $p_{i+1}\subseteq M_{3-j}$ and $p_i$ is not the constant geodesic for $i <n$.
If $c_1,...,c_{n-1}$ are the last points  of $p_1,...,p_{n-1}$ respectively and $p_1\subseteq M_s$ we say that  $p$ is of type $(s,c_1,...,c_{n-1})$. (If $p$ is entirely contained in $M_s$ then $p$ is of type  $(s)$.) Note that all $c_i$ lie in the finite cut $J$.
Let $M(s,c_1,...,c_k)$ be the set of points of $Q$ that can be joined to $a$ by a Tits geodesic of type $(s,c_1,...,c_k)$ and let $$\cM=  \{M(s,c_1,...,c_k):\,  |M(s,c_1,...,c_k)|>1 \text{ for }k>0 \}.$$ By lemma \ref{L:limit path} the elements of $\cM$ are closed sets of $Z$.  Since $(s,c_1,...,c_k)$ can contain no repeated elements of $J$,  $|\cM| \le 2^{m+1}$. We may assume that  $Q \subseteq \cup \cM$.
 Since $Q \not  \subseteq M_i$ for $i=1,2$ and $Q$ has no cut point in $J$, it follows that $|\cM|>1$.     Since $Q$ is connected, every element of $\cM$ meets some other element of $\cM$.

If  there exists $M(s,c_1,...,c_k) \in \cM$ with the property it only intersects other  elements of $\cM$ only in $c_k$ ( or $c \in J$ for $k=0$), then $c_k$(or $c$) is a cut point of $Q$ in $J$ which is not allowed.  
Thus we may assume that each $M(s,c_1,...,c_k) \in \cM$ intersects some other element of $\cM$ in a point other than $c_{k}$ (or $c$ for $k=0$).  Choose $M(s,c_1,...,c_k) \in \cM$ with $k$ maximal.  Choose some other $N \in \cM$ with $$b \in N \cap  M(s,c_1,...,c_k)$$ so that  $b \neq c_{k}$ (or $b \neq c$ for $k=0$).  Since $k$ is maximal, a Tits geodesic from $a $ to $b$ of type $(s,c_1,...,c_k)$  doesn't have the type given by $N$.  Thus there are at least two different Tits geodesics from $a$ to $b$.  Tits geodesics of length less than $\pi$ are unique, so $d_T(a,b) \ge \pi$.  
\end{proof}
\begin{Cor}\label{C:cut cicrle} Let $S\subseteq \bd X$ be a simple closed Tits curve which passes through a finite cut, then for any $b \in \bd X$, there is $a \in S$ with $d_T(a,b) \ge \pi$.
\end{Cor}
\begin{proof} We apply lemma \ref{L: bigpi}  to $S$. \end{proof}

\begin{Cor}\label{L: notcontractible} Let $Z$ be a continua with $A, B $ sub continua of $Z$ with $|A \cap B|\ge 2$. If $M\cup N$ is a decomposition of $Z$ by the minimal cut $J$ with $A \subseteq M$ and $B \subseteq N$
then $\rad (A \cup B)\geq \pi$.
\end{Cor}
\begin{proof} Clearly $ A\cap B \subseteq J$. We may assume that $A$ and $B$ are minimal continua with respect to the properties:
\begin{itemize}
\item $A \subseteq M$
\item $B \subseteq N$
\item $|A\cap B| \ge 2$.
\end {itemize}

If $J$ contains a cut point $c$ of $A\cup B$ then there are continua $X,Y$ such that $$X\cup Y=A\cup B,\, X\cap Y=\{c\}.$$
At least one of $A,B$ intersects non-trivially both $X,Y$. Let's say $A$ does and let $A_1=X\cap A,\, A_2=Y\cap A$.

Then $A=A_1\cup A_2$ with $A_1,A_2$ non-empty continua and
$A_1\cap A_2=\{c\}$. By our assumption above that $A$ is minimal $A_1$ and $A_2$ intersect $B$ at exactly one point. Moreover if one of them, say $A_1$, intersects
$B$ at $c$ then $A_2$ interects $B$ also only at $c$ contradicting $|A\cap B| \ge 2$. It follows that $c\notin B$. 

Since $B$ is a continuum it is contained either in $X$ or in $Y$. It follows that $B$ does not intersect one of $A_1,A_2$ contradicting again the minimality of $A$.

%
\end{proof}
%

\begin{Def} Let $A_1,A_2$ be min-cuts of the $m$-thick continuum $W$. We say that a subset $S\subseteq W$ is \textit{between} $A_1,A_2$ if for any decomposition
$W=M\cup N$ with $M,N$ continua such that $M\cap N=A_i$, $(i\in \mathbb Z_2 )$, either $A_{i+1}\cup S \subseteq M$ or  $A_{i+1}\cup S \subseteq N$.


  \end{Def}

\begin{Lem}\label{L:choosy}  Let $A$ be a min cut of the continuum $W$.  If $B,C$ is a decomposition of $W$ by $A$, and $D,E$ is a decomposition of $W$ by $A$, and $C\cap D \not \subseteq A$ then $[C\cap D], [B\cup E]$ is a decomposition of $W$ by $A$
\end {Lem}
\begin{proof} Since $W = B \cup C= D \cup E$,  it follows $$W =[C\cap D]\cup [B\cup E].$$  Since $A = B\cap C = D\cap E$,  it follows that $$A =[C\cap D]\cap [B\cup E].$$

 Let $\alpha$ be a quasi-component of $W-A$.  By \cite[Lemma 2.8]{PS2} $A \subseteq \bar \alpha$.  Let $Y = F(W-A)$, the Freudenthal compactification   of $W-A$, and let $f: Y \to W$ be the continuous surjection which is the identity on $W-A$.  By \cite[Observation following Lemma 2.6]{PS2} $f$ maps the component of $\alpha$ in $Y$ onto $\bar \alpha \subseteq W$.  If follows that $\bar \alpha$ is a connected subset of $W$.  
 
  Notice that  since $C\cap D\not \subseteq A$ $$[C\cap D]= \bigcup \{ \bar \alpha :\, \alpha \subseteq [C\cap D]\, \text{is a quasi-component of}\, W-A\}$$  is a nonempty union and $C\cap D$ is a continuum as required. 
\end{proof}

\begin{Lem}\label{L:middle} 
Let  $A_1, A_2$ be two distinct finite min cuts of the continuum $W$ that do not lie both in the same wheel, with $B_i,C_i$ a decomposition of $W$ by $A_i$ for $i=1,2$ and with $A_{3-i} \subseteq C_i$.    Then 
\begin{itemize}
 \item $C_1 \cap C_2$ is connected 
 \item$[A_1 \cup A_2] -[A_1 \cap A_2] \subseteq  \bd [C_1 \cap C_2] \subseteq  A_1 \cup A_2$
 \item $|C_1\cap C_2\cap A_1|= |C_1\cap C_2\cap A_2| $
 \end{itemize} \end{Lem} 
\begin{proof} 
The containment $ \bd [C_1 \cap C_2] \subseteq [\bd C_1 \cup \bd C_2]=  A_1 \cup A_2$ is obvious.  
\hfill\break{\bf Case I:}  $A_1\cap A_2 = \emptyset$.  \hfill\break
 Since the set  $A_1 \cup A_2$ consists of limit points of $C_1 \cap C_2$, $$A_1 \cup A_2 = \bd[C_1 \cap C_2].$$
 
 If $ C_1\cap C_2$ is not connected then  $C_1 \cap C_2 =E\cup D$ where $E,D$ are nonempty disjoint closed subsets of $W$.  So $\bd E$ and $\bd D$ are disjoint subsets of $A_1\cup A_2$. 
 Since the set  $A_1 \cup A_2$ consists of limit points of $C_1 \cap C_2$, $E$ and $D$ are both infinite.    Recall that for any subset $P$ of a topological space,  $\bd P$ separates the space provided the interior and exterior of  $P$ are nonempty.  Thus  both $\bd E$ and $\bd D$ separate $W$, and it follows that $$|\bd D| =|\bd E| = m =|\bd A_i|$$ and that $E$ is connected (as is $D$).  
 
 Notice now that $(E\cap A_1)\cup (D\cap A_2)$ is a min cut which separates the min cut $(E\cap A_2)\cup (D\cap A_1)$. It follows from \cite[Lemma 3.4]{PS2} that $$[(E\cap A_1)\cup (D\cap A_2)]\cup [(E\cap A_2)\cup (D\cap A_1)]= A_1\cup A_2$$ is a wheel, which is a contradiction.  Thus $C= C_1\cap C_2$ is a continuum as required.   Notice that $C_1\cap C_2 \cap A_i = A_i$ for $i=1,2$, so 
 $$|C_1\cap C_2\cap A_1|= |C_1\cap C_2\cap A_2| =m.$$
 \hfill\break{\bf Case II}  $A_1\cap A_2 =I \neq \emptyset$. \hfill\break
 Let $Y= F(W -I)$, the Freudenthal compactification of $W-I$ (see \cite[2.1]{PS2}).  So $A_1-I$ and $A_2-I$ are min-cuts of $Y$ and since $A_1$ and $A_2$ are not contained in a wheel of $W$, it follows from \cite[Lemma 2.7]{PS2} that $A_1-I$ and $A_2-I$  are not contained in a wheel of $Y$.  In $Y$, for $i=1,2$, let $$\hat B_i = \overline{(B_i-I)}\,\, \text{and} \,\,  \hat C_i = \overline{(C_i-I)}.$$  Notice that for $i=1,2$ $\hat B_i, \hat C_i$ is a decomposition of $Y$ by $A_i-I$.  Thus by Case I, $\hat C_1\cap \hat C_2$ is connected, $$\bd [\hat C_1\cap \hat C_2]= [A_1\cup A_2] -I,$$ and  $$|\hat C_1\cap \hat C_2\cap [A_1-I]|= |\hat C_1\cap \hat C_2\cap [A_2-I]| .$$
  Let $f: Y \to W$ be continuous surjection as in \cite[Corollary 2.5]{PS2}. 
 Notice that $f(\hat C_i) =C_i$ and $f(\hat B_i) =B_i$.  Thus $C_1 \cap C_2$  is connected and 
 $$[A_1 \cup A_2] -[A_1 \cap A_2] = f(\bd [\hat C_1\cap \hat C_2]) \subseteq \bd [C_1 \cap C_2].$$  Also $$| C_1\cap  C_2\cap [A_1-I]|= |C_1\cap  C_2\cap [A_2-I]| ,$$ implying that  $$| C_1\cap  C_2\cap A_1|= |C_1\cap  C_2\cap A_2| .$$
 
 \end{proof}

\begin{Lem}\label{L:between} Let $A_1,A_2,...,A_k$ be linearly ordered min cuts of the $m$-thick continuum $W$ such that no two of them lie in the same wheel. 
 Then there are continua $M,O,N_1,N_2,...,N_{k-1}$  such that $$M\cup N_1...\cup N_{i-1}, N_i\cup ...\cup N_{k-1}\cup O$$ is a decomposition of $W$ by $A_i$ for all $i=1,\dots k$ (where $N_j=\emptyset$ for $j \neq 1,2,\dots k$) .
 Moreover $$[A_i\cup A_{i+1}] -[A_i \cap A_{i+1}] \subseteq  \bd N_i \subseteq  A_i\cup A_{i+1}$$ and $$|N_i \cap A_i|= |N_i \cap A_{i+1}| $$ for all $i$.
  \end{Lem} 
  \begin{proof} Since the set of cuts $A_i$ is linearly ordered there are decompositions of $X$, $M_i,O_i$ by $A_i$ such that $$A_{i-1}\subseteq M_i,\, A_{i+1}\subseteq O_i$$ for all
  $i$. Now we set $$M=M_1,\, N_i=O_i\cap M_{i+1},\, O=O_k$$ for $i=1,...,k-1$. The result now follows by lemma \ref{L:middle}.
  \end{proof}

 \begin{Ex} \label{not4} There is a $4$-thick continuum $X$ with 4 linearly ordered min cuts $A_1,A_2,A_3,A_4$ such that none of the
 $A_i$ lies in a wheel, $A_i$ is between $A_{i-1},A_{i+1}$ and there is no continuum between $A_1,A_4$ that is $k$-thick for some $k\geq 2$.
 
  Indeed take $5$ disks $D_1,D_2,D_3,D_4,D_5$ and pick distinct points 
  $$a_1,a_2,a_3,a_4\in D_1 \text { and } b_1,b_2,b_3,b_4\in D_5.$$ We also pick 4 distinct
  points in each of $D_2,D_3,D_4$ that we label as follows: $$a_1,a_2,x,b_1\in D_2,\,  a_3,x,y,b_2 \in D_3,\, a_4,y,b_3,b_4\in D_4.$$
  We finally identify all points with the same label in different disks to obtain a 4-thick space $X$.
  
If $A_1=\{a_1,a_2,a_3,a_4\},  A_2=\{a_3,a_4,x,b_1\},  A_3=\{a_4,y,b_1,b_2\},  A_4=\{b_1,b_2,b_3,b_4\}$ then the cuts $A_i$ are linearly ordered and $X$
has the required property (see figure \ref{F:small}).

\begin{figure}[h]
\includegraphics[width=2.5in ]{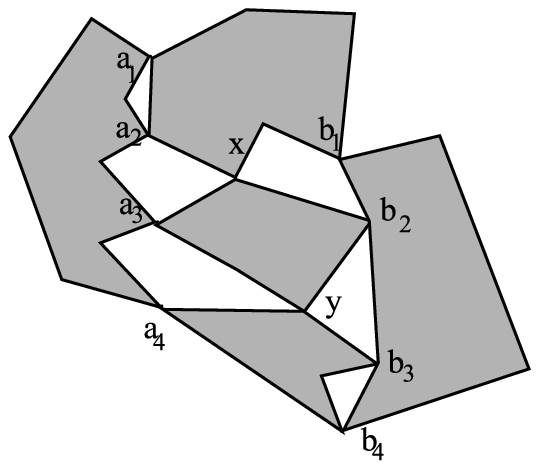}
\vspace{0.1in}
 \caption{}
 \label{F:small}
\end{figure}

 \end{Ex}

 We will need the following technical lemma:
 \begin{Lem}\label{linear1} Let $m\geq 2$ and let $A_1,A_2,...,A_6$ be linearly ordered min cuts of the $m$-thick continuum $W$ such that no two of them lie in the same wheel. 
 Let $M,O,N_1,N_2,...,N_5$ be continua such that $$M\cup N_1...\cup N_{i-1}, N_i\cup ...\cup N_5\cup O$$ is a decomposition of $W$ by $A_i$ for all $i=1,...,6$ (where $N_j=\emptyset$ for $j \neq 1,2,\dots 5$).
 Then there are $i<j$ and some $k$, $i\leq k\leq j$ such that $$|(N_i\cup N_{i+1}\cup...\cup N_{k})\cap (N_{k+1}\cup...\cup N_j)|\geq 2.$$
 
  \end{Lem} 
  \begin{proof}
  
  We distinguish two cases $m=2$ and $m\geq 3$. Suppose that $m=2$. If $A_1\cap A_3=\emptyset$ for some $i<j$ then by lemma \ref{L:between} and lemma \ref{L:middle}, 
  $$\partial (N_1\cup N_{2})=A_1\cup A_3.$$  Thus $$|(N_1\cup N_2) \cap N_3 |=| A_3|=2$$ as required. 
  
  Otherwise $A_1\cap A_2\cap A_3=\{x\}$. Say $$A_1=\{a_1,x\},\,A_2=\{a_2,x\},\,A_3=\{a_3,x\}.$$ If $\partial N_1=\{a_1,a_2\}$ and
  $\partial N_2=\{a_2,a_3\}$ then $\{a_1,a_3\}$ is a cut pair separating $\{x,a_2\}$ contradicting our assumption that $A_2$ does not lie in a wheel. It now follows from lemma \ref{L:between} that 
  $\partial (N_1\cup N_{2})\supset A_3$. Arguing in the same way with $A_3,A_4,A_5$ we have that $\partial (N_3\cup N_{4})\supset A_3$. So we may take $i=1,k=2,j=4$ and the lemma holds
  in this case.

We now settle the  case where $m\geq 3$.  By way of contradiction suppose that  $$|(N_1...\cup N_{k})\cap (N_{k+1}\cup...\cup N_5)|\le  1$$ for all $k =1,\dots 4$.
It follows that $$| N_i  \cap \bigcup\limits_{j \neq i} N_j| \le  \begin{cases}1, &\text{for } i=1,5\\ 2, &\text{for }  i=2,3,4 \end{cases}$$

   Let $$\alpha _i=(N_i\cap A_1)\setminus \bigcup\limits_{j \neq i} N_j \quad \text{and} \quad \beta_i=(N_i\cap A_6)\setminus \bigcup\limits_{j \neq i} N_j $$
 and notice that $$\alpha_i \cap \alpha_j= \emptyset =\beta_i \cap \beta_j \text{ for }i \neq j.$$ Set $a_i=|\alpha _i|$ and  $b_i=|\beta _i|$. 

Since $|A_1|=|A_5|=m$ we have
 $$ \sum a_i \le m\quad \text{and} \quad   \sum b_i \le m\ \ \ \  (*)$$
 
 Since $m \le |\partial N_i|$ for all $i$ we have:
 
 $$ m \leq \begin{cases}1+a_i +b_i, & \text{for } i=1,5\\ 2+a_i +b_i, &\text{for } i=2,3,4 \end{cases}$$

  Summing these inequalities we get
  
  $$5m \le8+ \sum a_i+ \sum b_i $$
  
  so by $(*)$ we obtain
  
  $$ 5m\le 8+ 2m  \quad \Rightarrow  \quad m\le   \frac 8 3.$$ Contradiction.  
  
  \end{proof}

 \begin{Thm} \label{bigtits}Let $m\geq 2$ and let $A_1,A_2,...,A_8$ be linearly ordered min cuts of the $m$-thick continuum $W$ such that no two of them lie in the same wheel. Then  there is a continuum $Q$ between $A_1,A_8$ such that $Q$ is separated by some $A_i$  to continua $Q_1,Q_2$ and $|Q_1\cap Q_2|\geq 2 $. In particular if $W$ is the boundary of a CAT(0) space, then
 $\rad Q\geq \pi $.
 
  \end{Thm} 
 
 \begin{proof} By lemma \ref{L:between}  there are  continua $M,O,N_1,N_2,...,N_{7}$  such that $$M\cup N_1...\cup N_i, N_{i+1}\cup ...\cup N_{7}\cup O$$ is a decomposition of $W$ by $A_i$ for all $i=0,1,...,7$. By lemma \ref{linear1} there are $1<i<j<8$ so that $$Q=N_i\cup...\cup N_j ,\, Q_1=N_i\cup N_{i+1}\cup...\cup N_{k},\, Q_2=N_{k+1}\cup...\cup N_j$$ are continua such that $$|Q_1\cap Q_2 |\geq 2,\, Q=Q_1\cup Q_2.$$ Clearly $Q$ is between $A_1,A_8$ since $$\partial Q\subseteq A_2\cup....\cup A_7.$$ If $W$ is a CAT(0) boundary, we have $\rad Q\geq \pi$ by lemma \ref{L: bigpi}  .
 \end{proof}
 
 In fact we may apply the same argument to show that there are sets of big Tits radius in wheels with `suffciently many' half-cuts. Note that if $F$ is a wheel of a continuum then
 the half-cuts (cf def. \ref{D:wheel}) are cyclically ordered. So if $A$ is a half-cut of $F$ then the half-cuts of $F\setminus A$ are linearly ordered.
 
 We have the following.
 
 \begin{Cor} \label{C:bigtits} Let $m\geq 3$ and let $F$ be a wheel with empty center of the $m$-thick continuum $W$. Let $B,C$ be two disjoint half-cuts of $F$
 and let $W_1,W_2$ be subcontinua of $W$ such that
 $$W=W_1\cup W_2 \text{ and } W_1\cap W_2=B\cup C.$$
 
Let $A_1,A_2,...,A_8$ be linearly ordered half-cuts of $F$ in $W_1$. Then  there is a continuum $Q$ between $A_1,A_8$ such that $Q$ is separated by some $A_i$  to continua $Q_1,Q_2$ and $|Q_1\cap Q_2|\geq 2 $. In particular if $W$ is the boundary of a CAT(0) space, then
 $\rad Q\geq \pi $.
\end{Cor}

\begin{proof} Since $F$ has an empty center $m=2k$ for some $k\geq 2, k\in \mathbb N$. Clearly the continuum $W_1$ is $k$-thick and $A_1,A_2,...,A_8$ are linearly ordered min-cuts of $W_1$.
Moreover no two lie in the same wheel of $W_1$-if they did $W$ would be $k$-thick rather than $m$-thick. The Corollary now follows from theorem \ref{bigtits}.

 \end{proof}
 
 \section{ Convex sets}  
Let $X$ be a CAT(0) space. We will define the parallel set of certain closed convex subsets of $X$.
\begin{Def} For $C$ a closed convex subset of $X$ we define projection $\pi_C: X \to C$ to be closest point projection which by
 \cite{BRI-HAE}   is a  Lipschitz one function.    We say that a close convex subset $C \subseteq X$ is {\em boundary minimal} if for any closed convex $B \subseteq C$ with $\bd B = \bd C$, $B=C$ (see \cite{C-M}).   If $A$ is boundary minimal and $\bd A =Z$ then we say that $A$ is {\em boundary minimal with respect to} $Z$.  
 \end{Def}
 Clearly geodesic lines are boundary minimal.  

\begin{Def}  For $A,B$ closed convex subsets of $X$, we say $A$ and $B$ are {\em parallel}  (denoted $A||B$) if $\pi_A:B \to A$ and $\pi_B:A\to B $ are inverse to one another.  (Since $\pi_A$ and $\pi_B$ are Lipschitz one,  It follows that $\pi_A:B \to A$ and $\pi_B:A\to B $ are isometries.)
For $A$ closed an convex, we define the {\em parallel set} of $A$, $P(A) = \bigcup\limits_{B|| A} B$.
\end{Def} 
Decent god fearing  convex sets have closed convex parallel sets that split as products (not all convex sets are decent or god fearing, but boundary minimal closed convex sets are ).  


\begin{Lem}\label{L:easy} Let   $A$ be closed and convex and $a \in A$.  For any $\alpha \in \bd A$, the ray $[a,\alpha)$ is contained in $A$.  
\end{Lem}
\begin{proof} Choose a sequence $(a_k) \in A$ with $a_k \to \alpha$.  Then $[a,a_k] \subseteq A$ by convexity and
$[a,a_k] \to [a,\alpha)$.  Since $A$ is closed $[a,\alpha) \subseteq A$.  
\end{proof}
For our working definition of decent god fearing,  we now state what was actually proven in Proposition 3.6 of \cite{C-M}.
\begin{Thm}\label{T:C-M}\cite{C-M} 
Let $Z$ be a boundary minimal closed convex subset of $X$. The union of all closed convex sets boundary minimal with respect to $\bd Z$  is closed and convex and splits as a product $Z \times K$ where $K$ is a convex subset of $X$. Furthermore $Z\times K = P(Z)$.  

Let $Y \subseteq X$ be a closed convex subset such that $\rad (\bd Y)> \frac \pi 2$.  Then there exists a closed convex subset $Z \subseteq X$ boundary minimal with respect to $\bd Y$.  \end{Thm}
\begin{proof}  The proof that any two closed convex sets boundary minimal with respect to the same subset of $\bd X$ are parallel is sadly lacking from \cite{C-M}.  There is a reference to \cite[page 10]{Leeb}, but the proof isn't there either, or at least not in the version that is publicly available on the archive.  Thus in the interest of completeness we provide an easy  proof of it here. 

Let $Z, \hat Z$ closed convex and boundary minimal with $\bd \hat Z =\bd Z$. We must show that the distance to $Z$, $d_Z$ is a constant function on $\hat Z$.  Suppose not, then there is an $\epsilon >0$ such that $\hat Z \cap \nb (Z, \epsilon)$ is a proper subset of $\hat Z$.  However, $\nb(Z,\epsilon)$ is closed and convex and $$\bd \nb(Z,\epsilon)= \bd Z =\bd \hat Z.$$  
It follows that $\hat Z \cap \nb (Z, \epsilon)$ is a proper closed convex subset of $\hat Z$.  Choosing a base point in $\hat Z \cap \nb (Z, \epsilon)$ and applying Lemma \ref{L:easy}, we see that  
$$\bd \left[\hat Z \cap \nb (Z, \epsilon)\right] =\bd \hat Z$$ contradicting minimality of $\hat Z$.  Similarly, $d_{\hat Z}$ is a constant function on $Z$.  It follows that for each $z \in Z$ there is a unique $\hat z \in \hat Z$ with $d(z,\hat z) = d(Z,\hat Z)$ and visa versa.  Thus $\pi_Z: \hat Z \to Z$ and $\pi_{\hat Z}:Z \to \hat Z$ are inverse functions as required.  
\end{proof}
The above neighborhood argument together with the Theorem also proves the following:
\begin{Cor}\label{C:projecting min} If $Z$ is a closed convex boundary minimal set of $X$ and $Y$ is convex with $\bd Z \subseteq \bd Y$, then there exists 
$\hat Z \subseteq Y$ parallel to $Z$.
\end{Cor}
\begin{Def}
 For $A$ a subset of a metric space $X$, the intrinsic circumradius of $A$ is $\inf \{ r:\, A \subset B(a,r)$ for some $a \in A\}$
\end{Def}
\begin{Cor} \label{C:min for subgroup} Let $H$ be a convex subgroup of $G$.  Then there exists a closed convex $Z$, boundary minimal with respect to $ \Lambda H$, on which $H$ acts geometrically. 
\end{Cor}
\begin{proof} Let $Y$ be the given closed convex subset of $X$ on which $G$ acts geometrically. By \cite{GEA-ONT}, for every $z \in \bd Y$, there is $w \in \bd Y$ with $z,w$ the endpoints of a geodesic line in $Y$.  Thus $\bd Y$ has intrinsic Tits circumradius $\ge \pi $  By Theorem \ref{T:C-M}, there are closed convex subsets $\check Z \subseteq Y$ and $K \subseteq X$ with $\check Z$ minimal among all closed convex subsets of $X$ with the property $\bd \check Z = \bd Y(=\Lambda H)$, and with the parallel set $P(\check Z)= \check Z \times K$.

For $\hat Z$ parallel to $\check Z$ suppose that $z\in \hat Z \cap Y $.  For any $\alpha \in \bd Y=\bd \hat Z$, the ray $[z,\alpha)\subseteq \hat Z \cap Y$ by Lemma \ref{L:easy}.   Since $\hat Z \cap Y$ is convex and $\bd [\hat Z \cap Y] =\bd \hat Z$, it follows by minimality of $\hat Z$ that $\hat Z \cap Y=\hat Z$.  

Thus $$Y \cap (\check Z\times K) = \check Z\times (K \cap Y).$$  
Since $$H(\bd \check Z) = \bd \check Z,\quad H(\check Z\times K) =H(P(\check Z))= P(\check Z)=\check Z\times K,$$ and thus $$H (Y \cap [\check Z\times Y]) =Y \cap (\check Z\times K)=\check Z\times (K \cap Y).$$
Since $\bd \check Z = \bd Y$, it follows  that $K  \cap Y$ is compact.  Since $K\cap Y$ is a compact CAT(0) space, it has a unique circumcenter $c$.  Thus $Z= \check Z \times \{c\}$ has the property that $H(Z) =Z$, and it follows that $H$ acts on $Z$ geometrically.  

\end{proof}

\begin{Cor} \label{C:convex} Let $H$ subgroup of $G$. If a finite index subgroup $K$ of $H$ is convex then $H$ is convex as well. 
\end{Cor}
\begin{proof} By corollary \ref{C:min for subgroup} $K$ acts on a closed convex set $Z$ boundary minimal with respect to $ \Lambda H=\Lambda K$.
All translates of $Z$ by $H$ are boundary minimal with respect to $ \Lambda H$ and there are finitely many such translates. By theorem \ref{T:C-M}
the union of all these translates is $Z\times F \subseteq Z \times A$  where $F$ is a finite subset of some convex set $A$.  $F$ has a unique centroid $p \in A$, so $H$ acts geometrically on the convex set $Z \times \{p\}$.  
\end{proof}
\begin{Def} For  for $H <G$, we define $$Z_H= \{g \in G: gh=hg\quad\forall h \in H\},$$ the centralizer of $H$ in $G$, and for $g\in G$ we define
$Z_g= Z_{\langle g \rangle}$.
  For $H< G$ we define $$\fp H = \{\beta \in \bd X : h(\beta) =\beta\quad \forall h \in H\},$$ and for $g\in G$ we define $\fp g = \fp\langle g \rangle$.
\end{Def}
\begin{Lem} \label{L:zent} For any $H < G$ if $g \in Z_H$  then $\Lambda H \subseteq \fp g$.  
\end{Lem}
\begin{proof} Let $p \in \Lambda H$, then there exists a sequence $(h_k)\subseteq H$ with $h_k(x) \to p$ for any $x \in X$.  
Since $g(x)$ is a fixed element of $X$, $h_k(g(x)) \to p$.  However $g(h_k(x)) \to g(p)$ since $G$ acts on $\bar X$, the compactification of $X$ by $\bd X$, by homeomorphisms.  Since $h_kg =gh_k$ it follows that $p = g(p)$.  Thus $\Lambda H \subseteq \fp g$.  
\end{proof}
\begin{Lem}\label{L:SWE} \cite{SWE}  Let $H < G$ be finitely generated.  Then $\fp H = \Lambda Z_H$, and $Z_H$ is a convex subgroup of $G$.
Furthermore if $\fp H\neq \emptyset$ then there is a hyperbolic element in $Z_H$. \end{Lem}
\begin{proof}
Let $h_1,\dots h_n$ be generators of $H$.  Clearly $\fp H = \bigcap\limits_{i=1}^n \fp h_i$.    By \cite{SWE}: 
\begin{itemize}
\item $ Z_{h_i}$ is convex for each $i$,
\item $\fp h_i = \Lambda Z_{h_i}$,
\item $\cap \Lambda Z_{h_i} = \Lambda \left[\cap Z_{h_i}\right]$,
\item $\cap Z_{h_i}$ is a convex subgroup of $G$.  
\end{itemize} Clearly $Z_H = \cap Z_{h_i}$ so $Z_H$ is a convex subgroup and  $$\fp H = \cap \fp h_i = \cap \Lambda Z_{h_i} = \Lambda \left[\cap Z_{h_i}\right] = \Lambda Z_H.$$
If $\fp H= \Lambda Z_H \neq \emptyset$ then $Z_H$ is an infinite CAT(0) group, and so contains a hyperbolic element by \cite{SWE}.
\end{proof}
\begin{Def} For $\Delta, \Gamma, \Theta \in \bd X$ we say that $\Theta$ is a $\bd$-join of $\Delta$ and $\Gamma$ denoted $\Theta = \Delta*_\bd \Gamma$ if $\Theta$ is the topological join of $\Delta$ and $\Gamma$ in the cone topology and $\Theta$ is the spherical join of $\Delta$ and $\Gamma$ in the Tits metric (taking $\Delta$ and $\Gamma$ as subsets of $\bd_T X$ for the latter).
\end{Def}
\begin{Lem} \label{L:limfixed join}   Let $H, K <G$ infinite,  at least one of which is finitely generated  with $K = Z_H$ and $H=Z_K$ and let $p$ be the rank of the abelian group $H \cap K$, so $\Lambda [H \cap K] = S^{p-1}$.  
Then there exists  $\Gamma,\Delta \subseteq \bd X$ with
\begin{enumerate}
\item $\Lambda \langle H, K \rangle \subseteq \Delta*_\bd\Gamma*_\bd S^{p-1}$
\item $\dim \Delta =  \dim \Lambda H -p +1$
\item $\dim \Gamma = \dim \Lambda K -p +1$
\item
 If $K$ is virtually a subgroup of $H$ or visa versa, then one of $\Delta$ or $\Gamma$ will be empty, but otherwise both $\Delta$ and $\Gamma$ are nonempty non singleton sets.
 
\item If   $\Lambda \langle H, K \rangle$  can be cut by a finite cut $C$ then
\begin{itemize}
\item either $H$ and $K$ are commensurable and both are virtually $\Z^2$, and $\Lambda H = \Lambda K$ is separated by $C$.
\item or there is a hyperbolic $g \in G$ with $H$  virtually $\langle g \rangle$ and $K$ virtually  $Z_g$ or visa versa, and $\Lambda Z_g$ is separated by $C$.
\end{itemize}
 In the latter case, every such finite cut must contain the endpoints $g$.
\end{enumerate}
 \end{Lem}


\begin{proof}   Say $H$ is finitely generated.
By Lemma \ref{L:SWE}, $Z_H =K$ is a convex subgroup and so finitely generated.  Now $H$ is convex by the same argument. 

By Corollary \ref{C:min for subgroup}  there exists $X_H \subseteq X$  a boundary minimal closed convex subset on which $H$ acts geometrically, 
 and we define $X_K$ similarly. 
 
  Now consider the CAT(0) space $X_H$ on which $H$ acts geometrically.  Since $H$ and $K$ commute,  $H\cap K$ is abelian and  $\Z^p$ is a finite index subgroup of $H\cap K$.  Since $H\cap K$ is the center of $H$, by Lemma \ref{L:zent} $H\cap K$ fixes $\Lambda H =\bd X_H$.  
 Restricting to $H$ acting on $X_H$, for  $ g \in H\cap K$, $$\bd \min g = \fp g  =\bd X_H $$ by \cite{SWE}.  Thus by the flat torus theorem \cite[7.1]{BRI-HAE} in $X_H$, and \cite{SWE} :
 \begin{itemize}
 \item  $\min \Z^p= X_H = A\times \E^p$ for some closed convex $A \subseteq X_H$,
 \item $H$   preserves this product decomposition,
 \item $H/ \Z^p$ acts geometrically on $A$. 
 \end{itemize} 
 Since $X_H$ was boundary minimal, so is $A$.
 
 Arguing similarly for $K$ we have in $X_K$:
  \begin{itemize}
 \item $ \min \Z^p=X_K = B\times \E^p$ for some closed convex $B \subseteq X_K$,
 \item $K$   preserves this product decomposition,
 \item  $K/ \Z^p$ acts geometrically on $B$. 
 \end{itemize} 
 Since $X_K$ was boundary minimal, so is $B$.
 
 Now in $X$:
 \begin{itemize}
 \item
 $\min \Z^p =  D \times \E^p$ for some closed convex $D \subseteq X$
 \item $Z_{\Z^p}> \langle H,K\rangle$ preserves this product decomposition.
 \item $Z_{\Z^p}/ \Z^p$  acts geometrically on $D$.
 \end{itemize}
  Clearly these product structures all match  up since each of the $\{\text{ point }  \} \times \E^p$ is parallel to every other.  If one of $H$ or $K$ is virtually a subgroup of the other
  then $A$ or $B$ is a point and  we are done.  We now assume that neither is virtually a subgroup of the other, and so $A$ will contain a geodesic line (being a CAT(0) space on which an infinite group acts geometrically), and similarly for $B$.
  
{ \bf Until further notice, we now restrict ourselves to the CAT(0) space} $\min \Z^p=D \times \E^p$.  
 By Theorem \ref{T:C-M}, we have for the parallel sets $$P(X_H)= X_H \times D_H=  A\times \E^p \times D_H$$
 $$P(X_K)= X_K \times D_K=  B\times \E^p \times D_K$$ where $D_H$ and $D_K$ are closed convex subsets of $\min \Z^p$. 
 Let $\Delta = \bd A$.  
 Then $$\Lambda H = \bd X_H= \bd (A \times \E^p) = (\bd A) *_\bd \bd (\E^p)=\Delta *_\bd S^{p-1}$$ and  we have (2).

 For any $k \in K$, by Lemma \ref{L:zent}  $k(\Lambda H) = \Lambda H$ and so $\bd k(X_H) =\Lambda H$.   $k(X_H)$ is a closed convex set of $\min \Z^p$ and using $k^{-1}$ we see that $k(X_H)$ is boundary minimal.  Thus $k(X_H)$ is parallel to $X_H$  and so $$k(X_H) \subseteq P(X_H)$$ for all $k \in K$.   It follows that $P(X_H)$ contains a $K$-orbit (lots of them in fact) and so $$\Lambda K \subseteq \bd P(X_H) .$$    
  
 By Corollary \ref{C:projecting min}, there is a set parallel to $X_K$ contained in $P(X_H)$.  Thus by rechoosing $X_K$, we may assume that $X_K \subseteq P(X_H)$.  Similarly by rechoosing $X_H$, $A$ and $B $ we may assume that
 \begin{itemize}
 \item $X_H \subseteq  P(X_K)$.
 \item $X_K \cap X_H \neq \emptyset$
 \item  $A \cap B \neq \emptyset$ 
 \item $D_H \cap B \neq \emptyset$.  
 \end{itemize}

Notice that since $X_K \subseteq P(X_H)$, $B \subseteq A \times D_H$. Since $\Z^p$ acts geometrically on  $X_H\cap X_K = \E^p\times [B\cap A]$, $B\cap A$ is compact by the flat torus theorem.  In fact $B \cap [A\times\{x\}]$ will  be compact by the same argument for any $x \in D_H$.
Let $$\rho : A \times D_H \to D_H$$ be the projection, and  by abuse of notation, we will also write
$$\rho: \bd A *\bd D_H- \bd A \to \bd D_H$$ for the corresponding projection along join arcs.

 For $L$  a geodesic   in $B$, then $\rho \circ L$ must be a constant speed geodesic in  $D_H$.  Fix a base point $b \in B$ and for each $\beta \in \bd B$ let $L_\beta$ be the geodesic ray from $b$ to $\beta$, and $$\theta_\beta = \angle_b(L_\beta, \rho \circ L_\beta).$$ Since $\bd B \cap \bd A = \emptyset$, there is a uniform bound $\theta< \frac \pi 2$ with $\theta_\beta \le \theta$ for all $\beta \in \bd B$.  We leave it to the reader to show that neighborhoods of $\bd B$ based at $b$ are cofinal under $\rho$ to neighborhoods based at $b$ in $\rho(\bd B)$ and so  $$\rho (\beta) = \rho \circ L_\beta(\iy)$$ defines a topological embedding $\rho: \bd B \to \bd D_H $.  
 Also $\rho(B)$ will be  a closed convex subset of $D_H$ by \cite[I 5.3(3)]{BRI-HAE}.
 Clearly $\bd \rho(B) = \rho (\bd B)$ and  $$B \subseteq A \times \rho(B) \subseteq A  \times D_H.$$  Thus $$(A \times \E^p) \cup (\E^p \times B) \subseteq A \times \E^p \times \rho(B)$$ which in both $H$ and $K$ invariant, and so $\langle H ,K \rangle $ invariant.   Let $ \Gamma = \rho(\bd B)= \bd \rho(B)$ so  $$\Lambda \langle H , K\rangle \subseteq \bd A *_\bd \bd \E^p *_\bd \bd \rho(B)=\Delta*_\bd S^{p-1} *_\bd \Gamma$$
 
, $\dim \Gamma = \dim \bd B$ we have proven (3).
  Since $\rho :\bd B \to \Gamma$ is injective (in fact a homeomorphism), $\Gamma$ contains at least two points as required by (4).

Now consider the case where $C$ separates $\Lambda \langle H, K\rangle$. In the case when $$\Delta*_\bd\Gamma*_\bd S^{d-1}= S^1,$$  $H$ and  $K$ are commensurable and virtually $\Z^2$ and $\Lambda H= \Lambda K$ are separated by $C$
Now  we may assume    $$S^1 \neq\Delta*_\bd\Gamma*_\bd S^{d-1}.$$ 
So for there to be even one finite cut, $\Delta*_\bd\Gamma*_\bd S^{d-1}$
must be a suspension.   Thus one of $\Delta$ or $\Gamma$ will be empty and $d =1$.
So say $H$ will be two ended and taking $g \in H$ hyperbolic, $K$ will be virtually $Z_g$ as required.  Clearly every finite cut cutting $\Lambda \langle H, K\rangle $ must contain the only finite cut $\{g^\pm\}$ of $\Delta*_\bd\Gamma*_\bd S^{d-1}$  and will separate $\Lambda Z_g$.

\end{proof}

\begin{Lem} \label{L:fixfg}
  Let $A\subseteq \bd X$, if $\fp A$ is finitely generated then $\fp A$  is a convex subgroup of $G$
\end{Lem}
\begin{proof} Let $H= \fp A$;  by Lemma \ref{L:SWE}, $Z_H$ is a convex subgroup and  $\fp H  = \Lambda Z_H$.   Clearly $H <Z_{Z_H}$ and
by Lemma \ref{L:SWE} ${Z_{Z_H}}$ is a convex subgroup and $$\fp Z_{Z_H} = \Lambda Z_{Z_{Z_H}}.$$  But it is always the case that 
$Z_{Z_{Z_H}} =Z_H$, so $$\fp Z_{Z_H}= \Lambda Z_H = \fp H \supseteq A.$$ So $\fp A =H\subseteq Z_{Z_H} \subseteq \fp A$ and so $H = Z_{Z_H}$.
\end{proof}
This argument also yields:
\begin{Cor}  If $H $ is a finitely generated subgroup of $\fp A$ then $H <Z_{Z_H}< \fp A$ and $Z_{Z_H}$ is convex.
\end{Cor}
\begin{Lem}\label{L: finite cutter} Let $X$ be a $m$-thick continuum $m\ge2$.  Let $H < G$ be an infinite finitely generated group with $\fp H \neq \emptyset$.   Then $$\fp H \cup \Lambda H\subseteq \Lambda \langle  Z_H, Z_{Z_H} \rangle$$ is contained in a join/spherical-join in $\bd X$/ $\bd_T X$. 
If $\Lambda \langle  Z_H , Z_{Z_H}\rangle$ is separated by a finite cut $C$, then  either
\begin{itemize}
\item $H$ is  virtually a subgroup of a $\Z^2$ with $\fp H \cup \Lambda H \subseteq \Lambda \Z^2$   separated by $C$  
\item or there is a hyperbolic $g\in G$ with  $g^\pm \in C$ and $\Lambda \Z_g$ separated by $C$.
Also $H$ is virtually a subgroup of $ \Z_g$  and $Z_H$ is commensurable with $\langle g \rangle$ or $H$ is commensurable with  $\langle g \rangle$ and $Z_H$ is commensurable with $Z_g$. 

\end{itemize}
\end{Lem}
\begin{proof}

By Lemma \ref{L:SWE}, $Z_H$ is a convex subgroup and  $\emptyset \neq \fp H  = \Lambda Z_H$.  Thus $Z_H$ is infinite.
Now since $Z_H$ is convex, it is finitely generated.  It follows by Lemma \ref{L:SWE} that  $\fp Z_H = \Lambda Z_{Z_H}$.  It is always the case that  $$H< Z_{Z_H}\text{ and }Z_H= Z_{Z_{Z_H}}.$$
Applying Lemma \ref{L:limfixed join} to $Z_H$ and $Z_{Z_H}$, $$\Lambda \langle H, Z_H\rangle \subseteq \Lambda \langle Z_{Z_H} Z_H\rangle  \subseteq \Delta*_\bd\Gamma*_\bd S^{p-1}$$  where $p$ is the rank of the abelian group $Z_H \cap Z_{Z_H}$.

If $ \Lambda \langle Z_{Z_H} Z_H\rangle$ is separated by a finite cut and either
\begin{itemize} 
\item $Z_H$ and $Z_{Z_H}$ are commensurable and virtually $\Z^2$ in which case $H$ is a virtually a subgroup of this $\Z^2$ 
\item or there is a hyperbolic $g\in G$ with  $g^\pm \in C$ and $\Lambda \Z_g$ separated by $C$.
Either $Z_{Z_H}$ is commensurable with  $Z_g$  (so $H$ is virtually a subgroup of $Z_g$) and $Z_H$ is commensurable with  $\langle g \rangle$ 
or $H$ is commensurable with  $\langle g \rangle$  and  $Z_H$ is commensurable with $Z_g$. 
\end{itemize}
\end{proof}

\begin{Thm}\label{T: torsion fixer} Let $X$ be a proper CAT(0) space and $G$ a group acting geometrically on $X$.  If $H<G$ is infinite torsion then
$H$ doesn't fix a point in $\Lambda H$. (See  \cite{SWE} )
\end{Thm} 
\begin{proof}  Suppose not.  We may assume the $\dim \bd X$ is minimal for all such counter examples, and that $H$ fixes the point $q \in \Lambda H$.  There is a bound on the size of a finite subgroup of $G$, so we may assume that $H$ is finitely generated. 
By Lemma \ref{L:SWE}   there is a hyperbolic $g \in  Z_H$.  Thus $ H < Z_g$.  
Now the minset $M_g$ decomposes as a product $Y \times \R$ where $Y$ is a convex subset of $X$ $\R$ "is" an axis of $g$, and $Z_g$ acts geometrically on $M_g$ preserving the product structure by \cite{BRI-HAE}.  Since $H$ is torsion, it doesn't translation in $\R$ and so $H(Y) =Y$.  It follows that $\Lambda H \subseteq \bd Y$, and $q \in \bd Y$.

 Since $\langle g \rangle \triangleleft Z_g$,  by \cite{SWE}  the quotient  $Z_g/\langle g \rangle $ acts geometrically on the quotient $$M_g /\langle g \rangle=
Y \times S^1.$$  Thus $Z_g/\langle g \rangle $ acts geometrically on $Y$.  Since $H$ is torsion $H \cap \langle g\rangle = \{e\}$ and so the projection $$\rho:Z_g \to Z_g/\langle g \rangle $$ restricts to monomorphism on $H$ and we may write $H<  Z_g/\langle g \rangle $. 

By \cite{CH-SW}  $$\dim \bd (M_g) =\dim\bd Y +1$$ and so $\dim \bd Y < \dim \bd X$.   Since $H$ fixes $q \in \bd Y$,  this contradicts the minimality of $\dim \bd X$.  

\end{proof}
\begin{Cor}\label{C: finite limit set} \cite[lemma 16]{PS1}   If $H < G$ and $\Lambda H$ is finite then $H$ is virtually cyclic.  
\end{Cor}
\begin{proof} We may assume $\Lambda H \neq \emptyset$.   Passing to a finite index subgroup of $H$, we may assume that $H$ fixes $\Lambda H$. 
If every finitely generated subgroup of $H$ is virtually cyclic, then $H$ is virtually cyclic, so we may assume that 
$H $ is finitely generated. By Lemma \ref{L:SWE}  there is a hyperbolic $g \in Z_H$.  Thus $ H < Z_g$.
As above,  $Z_g/\langle g \rangle $ acts geometrically on the quotient $M_g /\langle g \rangle=
Y \times S^1$ or just $Y$.  Let $$\rho: Z_g \to Z_g/\langle g \rangle $$ be the projection map.  
 If $\rho(H)$  is finite then $H$ is clearly virtually cyclic.  

If $\rho(H)$ is infinite, them every point of the limit set of $H$ in $\bd Y$, is the image of a point of $\Lambda H \subseteq \Lambda Z_g$ under the obvious projection from the  suspension  $\sum \bd Y$ minus the suspension points  to $\bd Y$.  Thus the limits set of $H$ in $\bd Y$ is finite.  We now argue by induction on $\dim \bd X$. Since $\dim \bd Y < \dim \bd X$ and the result is clearly true when $\dim \bd X = -1$ (where $G$ itself is finite) we are done.
\end{proof}
\section{Rank 1}
We will show in this section that if $\bd X$ is $m$-thick for $m>2$, and the tree is non-nesting, then $G$ (or $X$) is rank 1.

Let $T$ be the cactus tree of $\bd X$. We will show first that the action of $G$ on $T$ has no fixed points.
We recall that the cactus tree $T$ is an $\R$-tree obtained (via a canonical procedure, see
\cite {PS2}) from the pretree $\mathcal R$ consisting of the inseparable cuts (isolated min cuts) of $\bd X$ and the maximal wheels of $\bd X$.
The pretree $\mathcal R$ might not be median. To obtain $T$ one first turns $\mathcal R$ into a median pretree $\mathcal P$. In particular
if $x,y,z\in \mathcal R$ are such that none of the three separates the two others and $$[x,y]\cap [y,z]\cap [x,z]=\emptyset $$ then there is
a point $M\in \mathcal P$ which is between any two of $x,y,z$. If $M$ is a median point then $M$ corresponds to a subset of $\partial X$ which
we will also denote by $M$. We explain how one obtains this subset. If $M$ is a median point then there are disjoint linearly ordered subsets
$R_i,\, (i\in I)$ of $\mathcal R$ where $|I|\geq 3$ such that $\bigcup _{i\in I}R_i=\mathcal R$ and $M$ is between $R_i,R_j$ for any $i\neq j$ in $\mathcal P$.
Note that since $\mathcal R$ is countable $I$ is countable. For each isolated min cut $x\in R_i$ and for each decomposition
$\bd X=A\cup B$ with $A,B$ continua and $A\cap B=x$ either $A$ or $B$ contains the union $\bigcup _{j\ne i} R_j$. Let's denote
the continuum that contains this union by $S_{x,A,B}$. Then $M$ corresponds to the intersection of all $S_{x,A,B}$. As these sets are continua
and the intersection of any finitely many of them is clearly non empty the intersection of all of them is a non-empty as well. We will
denote this intersection by $M$ as well and we will call these sets the \textit{median sets} of $\bd X$. 

We show now that, more generally, if $K$ is a connected subset of $T$ then there is a subset of $\bd X$ which we will denote also by $K$ which corresponds to it.
Assume first that $K$ is closed with endpoints min cuts $x_i, i\in I$. For each min cut $x$ which either lies in $I$ or does not lie in $K$ and for each decomposition
$\bd X=A\cup B$ with $A,B$ closed and $A\cap B=x$ either $A$ or $B$ contains all min cuts of $K$. Let's denote
the set that contains this union by $S_{x,A,B}$. Then $K$ corresponds to the intersection of all $S_{x,A,B}$. 
More generally a connected set $K\subseteq T$ can be written as an ascending union of sets which have min cuts as endpoints and (possibly)
a union of maximal wheels and median sets which are endpoints of $K$. So again it corresponds to a subset of $\bd X$. 
In the sequel it will be convenient to use this, so we will pass often from subsets of $T$ to subsets of $\bd X$.

Since the action of $G$ on $T$ is non-nesting it is enough to show that $G$ does not fix any points of $\mathcal P$. The points of $\mathcal P$ are
of three types: inseparable min cuts, median sets and maximal wheels.

 We deal now with the first type.
 
\begin{Lem}\label{fixedcut}
$G$ fixes no point of $T$ corresponding to an inseparable min cut.
\end{Lem}
\begin{proof}

If $G$ leaves an
inseparable min cut invariant, then $G$ virtually fixes each element of the cut.
 By \cite[Lemma 26]{PS1} 
 if $G$ virtually fixes a point of  $\bd X $  then
$G$ is virtually $ H \times \Z$ for some finitely presented group.  Since $\bd X$ is not separated by 2 points $H$ is one ended. However in this case
$\bd X$ is not separated by a finite set of points, a conradiction.

\end{proof}

\begin{Lem} \label{diameter}If the Tits diameter of $\bd X$ is more than $\frac {3\pi} 2$,
then $G$ leaves no median set invariant.
\end{Lem}
\begin{proof} Let $M$ be a median set and let $A_1,A_2$ be min cuts such that $M$ is between $A_1,A_2$. There  are nonempty disjoint
open subsets $U_i$ and $V_i$ ($i=1,2$) such that $$U_i\cup V_i = \bd X-A_i .$$ We may assume that $M\subseteq V_1\cap V_2$. Choose a
rank 1 hyperbolic element $g \in G$ with $g^- \in U_1$ and $g^+ \in
U_2$.   By $\pi$-convergence, it follows that $g^n(p) \subseteq U_1$  for sufficiently big $n \in \mathbb N$, so $g^n(p)\ne p$.
\end{proof}

%

\begin{Thm}\label{median}
$G$ doesn't leave a median set invariant.
\end{Thm}
\begin{proof}

Let's assume that $G$ leaves invariant a median subset of
$\partial X$, say $D$. Clearly then there is a minimal invariant set $I$ for the
action of $G$ on $\bd X$ that is contained in $D$. Note that by definition $D$ is closed.
By lemma \ref{diameter} $\bd X$ has finite Tits diameter.

$D$ is a point of $T$. Let $R$ be a component of
$T-D$.  We can view $R$ as a subset of $\bd X$ in the appropriate way.  With this viewpoint,  the closure of $R$ in the Tits topology (i.e. the topology defined by $d_T$)
intersects $D$ at at most $m$ points. Indeed note that if $x_1,...,x_n\in R\cap D$ then there are disjoint Tits geodesics $p_1,...,p_n$ joining points of $R$ to $x_1,...,x_n$.
However every point of $R$ is separated from $D$ by some min cut. So there is a single min cut separating the endpoints of $p_1,...,p_n$ that lie in $R$ from $D$. It follows that
$n\leq m$. Since $\bd X$ has finite Tits diameter the closure of $R$ intersects $D$ at at least one point.

Since $G$ does not virtually fix a point the orbit $g\cdot R$ is infinite so $T\setminus D$ has infinitely many components.

%

\begin{figure}[h]
\includegraphics[width=2.5in ]{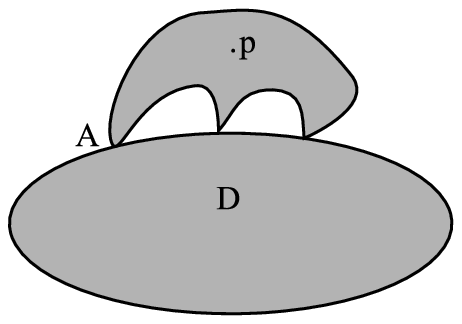}
\vspace{0.1in}
 \caption{}
\end{figure}

Let $R$ be a component of $T\setminus \{D\}$ and let $A$ be an inseparable min cut in $R$.
 Let $p$ be a point of $\bd X$ separated from $D$
by $A$,  and let $g_i\in G$ be such that for some (hence for
all) $x\in X$, $g_i(x)$ converges to $p$. By passing to a
subsequence we may assume that $g_i^{-1}x$ converges to some $n\in
\partial X$.

With this notation we have the following
\begin{Lem}\label{two} There are $x_1,x_2$ lying in distinct components of $T\setminus \{D\}$ such that $d_T(n,x_i)\geq \pi $.
\end{Lem}

\begin{proof} Let $R$ be a component of $T\setminus \{D\}$. We distinguish two cases:

\textit{Case 1}. $R$ contains infinitely many inseparable min cuts. 

Let $A_1,A_2$ be inseparable min cuts lying in $R$. Say that $A_2$ is between $A_1$ and $D$. 

There are continua
$M_i,N_i$ with $$M_i\cap N_i=A_i \text{ and } M_i\cup N_i=\bd X\, (i=1,2).$$ Let's say $D\subseteq M_i$ $(i=1,2)$. Let $Y=M_2\cap N_1$. Then $Y$ is compact
and it has a connected component $Y_1$ such that $|Y_1\cap A_2|\geq 2$.  

Let $w_1$ be a Tits path of minimal length in $Y_1$ joining two points $a,b$ of $A_2$. Let $w_2$ be a Tits path of minimal length in $N_2$ joining $a,b$.
We homotope $w_1\cup w_2$ to a (local) Tits geodesic in $\bd X$
and we denote it by $w$. By Corollary \ref{C:cut cicrle} there is some $x_1\in w$ such that $d_T(n,x_1)\geq \pi $. By $\pi $-convergence $g_ix_1\to p$ so $x_1\notin D$.

Clearly there is a translate $R'=gR$ of $R$ disjoint from $R$. We pick a point $x_2\in R'$ as we did for $R$. In this case the lemma is proven.

\begin{figure}[h]
\includegraphics[width=2.5in ]{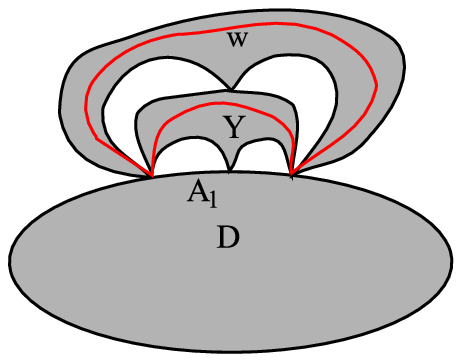}
\vspace{0.1in}
 \caption{}
\end{figure}

\textit{Case 2}. $R$ contains finitely many inseparable min cuts. Then clearly $\bar R \cap D$ contains an inseparable min cut $A$.
There are continua
$M,N$ with $M\cap N=A$ and $M\cup N=\bd X$. Let's say $D\subseteq M$. Let $w_1$ be a Tits geodesic of minimal length in $N$ joining two points of $A$.
Denote these points by $a,b$. There is a Tits path $w_2$ of finite length in $M$ joining $a,b$. 

We homotope $w_1\cup w_2$ to a (local) Tits geodesic in $\bd X$
and we denote it by $w$. By Corollary \ref{C:cut cicrle} there is some $x_1\in w$ such that $d_T(n,x_1)\geq \pi $. By $\pi $-convergence $g_ix_1\to p$ so $x_1\notin D$.
It follows that $x_1$ lies in a component of $T\setminus \{D\}$, say $R'$. Let's say that $w=v_1\cup v_2$ with $v_1=\bar R'\cap w$.

\begin{figure}[h]
\includegraphics[width=3.0in ]{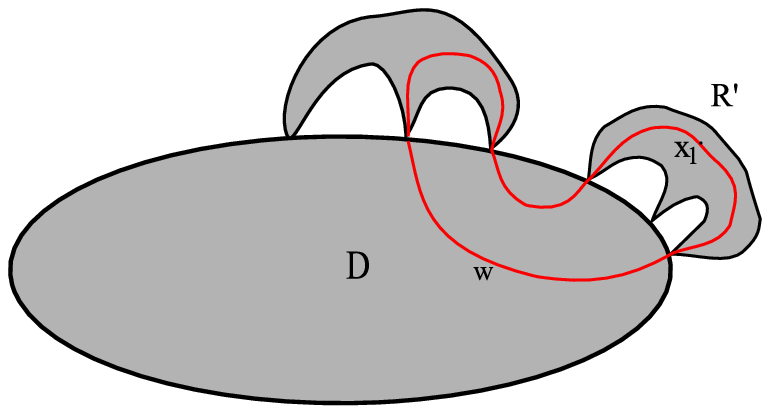}
\vspace{0.1in}
 \caption{}
\end{figure}

We remark that the set $\{gR: g\in G\}$ is infinite as otherwise $G$ virtually fixes a point. Since $w$ has finite Tits length it does not intersect infinitely many translates $gR$. Let $R_1$ be a component of $T\setminus \{D\}$  that does not intersect $w$.
Let $A_1$ be an min cut lying in the closure of $R_1$.
There are continua
$M_1,N_1$ with $M_1\cap N_1=A_1$ and $M_1\cup N_1=\bd X$. Let's say $D\subseteq M_1$. Let $u_1$ be a Tits geodesic of minimal length in $N_1$ joining two points of $A_1$.
Denote these points by $a_1,b_1$. There is a Tits path $u_2$ of finite length in $M$ joining $a_1,b_1$. We homotope $u_1\cup u_2$ to a (local) Tits geodesic in $\bd X$
and we denote it by $u$. Assume that $u$ does not intersect $R'$.
By Corollary \ref{C:cut cicrle} there is some $x_2\in u$ such that $d_T(n,x_2)\geq \pi $ and the lemma is proven.

If $u$ intersects $R'$ we write $$u=u_1'\cup u_2'\text{ with }u_1'=\bar R'\cap u.$$ Note that the endpoints of $u_1'$ lie in $A$. If the endpoints of $u_1'$ are the same as the endpoints of $v_1$
we consider $u'=u_2'\cup v_2$.We homotope it to a (local) Tits geodesic in $\bd X$ which we denote still by $u'$. We note that $u'$ is either a simple closed curve or contains
a simple closed curve that is separated by a min-cut. So by
corollary \ref{C:cut cicrle} there is some $x_2\in u'$ such that $d_T(n,x_2)\geq \pi $. As we argued before $x_2$ does not lie in $D$ and the lemma is proven.

\begin{figure}[h]
\includegraphics[width=3.0in ]{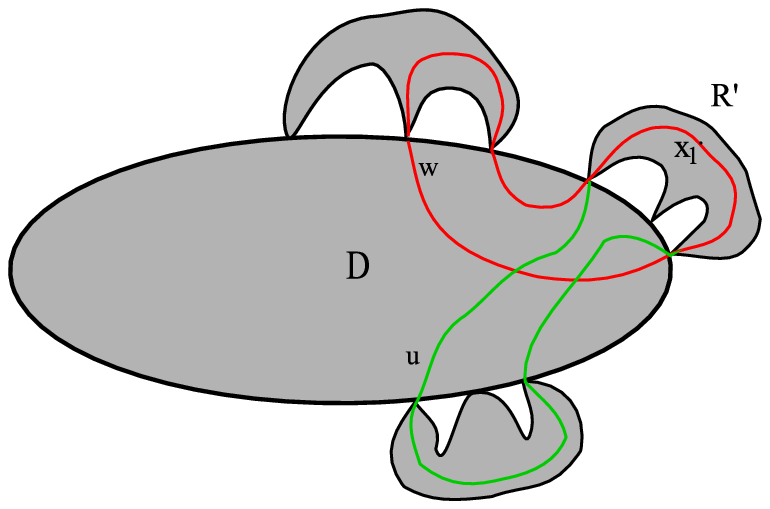}
\vspace{0.1in}
 \caption{}
\end{figure}

If the endpoints of $u_1'$ are not the same as the endpoints of $v_1$ we repeat the procedure starting with a component $R_2$ which does not intersect $w,u$.
In this way we obtain a new non-contractible closed curve as before. Again the lemma is proven unless this curve intersects $R'$ and $A$. Since $A$ is finite this procedure will
eventually produce two arcs $u_1',v_1'$ intersecting distinct components of $T\setminus \{D\}$ with the same endpoints lying in $A$. We pick then as before
$x_2\in u_1'\cup v_1'$ such that $d_T(n,x_2)\geq \pi $ and the lemma is proven.

\end{proof}

Let $x_1,x_2$ be the points given in lemma \ref{two}. Then by $\pi $-convergence $g_ix_1\to p$ and $g_ix_2\to p$. However this is impossible
as $x_1,x_2$ lie in distinct components $R,R_1$ of $T\setminus \{D\}$. If $p$ lies in a component $\tilde R$ of $T\setminus \{D\}$ since 
$g_ix_1\to p$ and $g_ix_2\to p$ for sufficiently large $i$ we have $g_iR'=\tilde R$ and $g_iR_1=\tilde R$ which is clearly impossible.
This proves the theorem.

\end{proof}

\begin{Lem} \label{fixwheel}
$G$ leaves no maximal wheel invariant.
\end{Lem}
\begin{proof} Assume that $G$ leaves invariant the maximal wheel $W$. If $W$ is finite then $G$ virtually fixes a point. But as we noted
earlier this is impossible. For the same reason the center of $W$ is empty. Since $W$ is infinite it has infinitely many
disjoint half-cuts. 
Since the Tits diameter of $\bd X$ is finite it follows that for any $\epsilon >0$ there are half cuts
of $W$, $A_1,A_2$ such that $d_T(A_1,A_2)<\epsilon $. If we fix a half cut $A$ of $W$ we may define betweeness on the set
of half cuts, we say $A_2$ is  between $A_1,A_3$ if the cut $A_1\cup A_3$ separates $A,A_2$. It follows that given the choice of $A$
the set of half cuts of $W$ can be linearly ordered. We fix now such an order $<$.

Let $A_1<A_2<...<A_{10}$ be half cuts of $W$. Then  there are continua $W_1,W_2$ such that 
$$\bd X=W_1\cup W_2, W_1\cap W_2=A_9\cup A_{10},\, A_i\subseteq W_1,\, i=1,...,8.$$

By corollary \ref{C:bigtits} there is a continuum $Q$ in $W_1$ between $A_1,A_8$, of Tits-radius $\geq \pi$.

 
Let $B_1,B_2$ be half cuts in $W$ with $d_T(B_1,B_2)<\pi/2 $. We pick $p$ in the interior of a shortest (Tits) path joining $B_1$ to $B_2$. Let $g_i\in G$
such that $g_i\to p$. By passing to a subsequence we may assume that $g^{-i}\to n$. There are points $s_1, s_2\in Q$ with $d_T(s_1,s_2)=\pi $
with $$d_T(n,s_1)\geq \pi,\, d_T(n,s_2)\geq \pi .$$ It follows that $g_is_1\to p, g_is_2\to p$. This is however impossible as $g_i$ sends half cuts to half cuts and
$d_T(s_1,s_2)\geq  \pi/2 $.
\end{proof}
We now have the following:
\begin{Cor}\label{nofixed}
$G$ acts on the Cactus tree $T$ of $\bd X$ without fixed points.
\end{Cor}
We will need a lemma from permutation groups that is due to P.M. Neumann:

\begin{Lem}\cite[Lemma 2.3]{Neu}\label{permutation} Suppose $G$ is acting on a set $\Omega $ without virtually fixing a point. Then for any finite $A\subseteq \Omega$ there is some
$g\in G$ such that $gA\cap A=\emptyset $.
\end{Lem}

\begin{Thm} \label{rank1}
$G$ is rank 1.
\end{Thm}
\begin{proof} 

\begin{figure}[htbp]
\hspace*{-3.3cm}                                                           
   \includegraphics[scale=1.00]{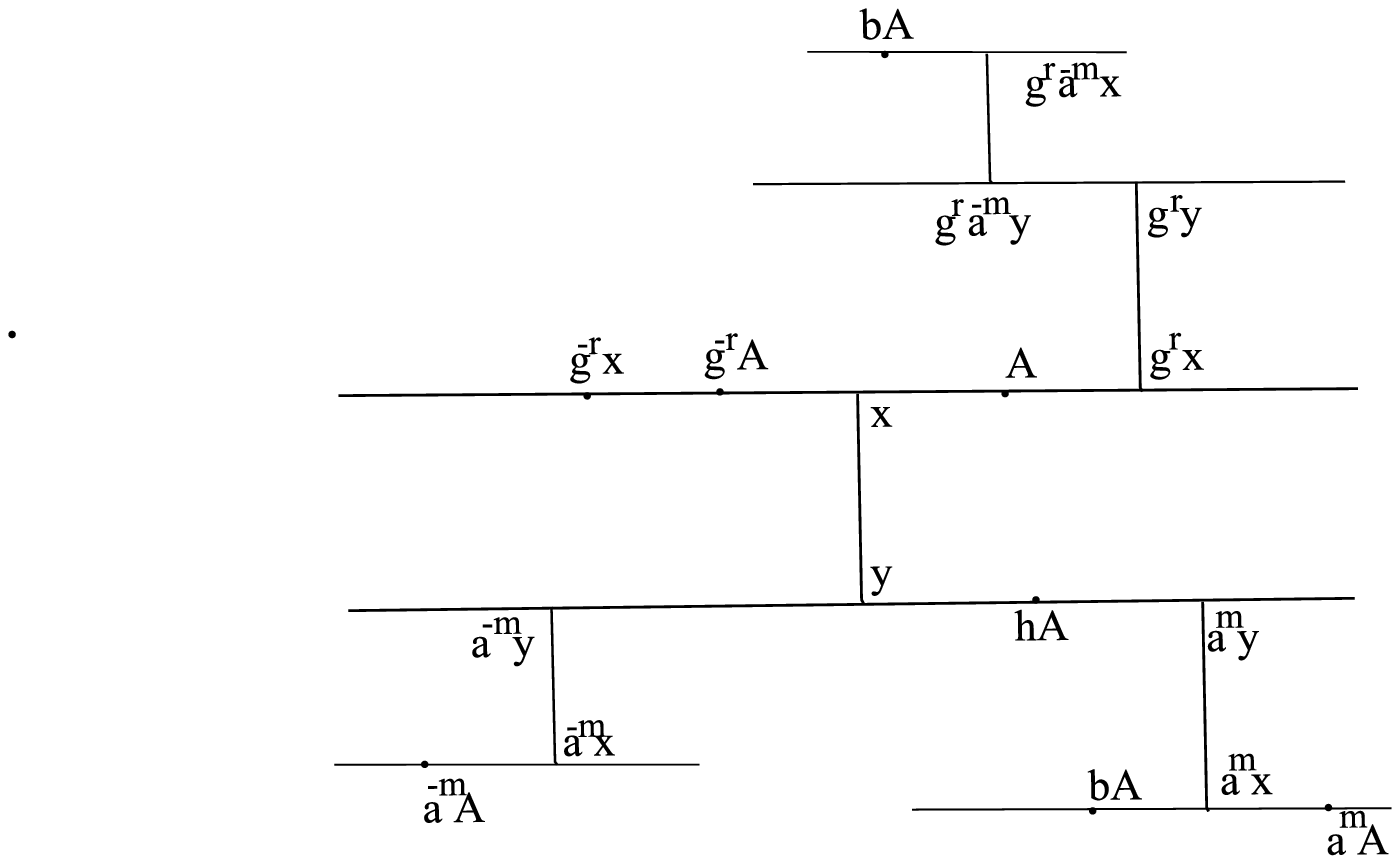}%

  \caption{}
  \label{Tree.eps}
\end{figure}


Let $T$ be the cactus tree of $\partial X$. Since $G$ acts on $T$ without fixed points there is an element $g$ acting
hyperbolically with axis of translation $l$. Let $A$ be an inseparable cut lying on $l$. If for some power $g^n$ we have
$g^nA\cap A=\emptyset $ then $$d_T(A,g^nA)=\epsilon >0$$ and $$d_T(A,g^{nk}A)=k\epsilon $$  for any $k\in \mathbb N$. It follows that
the Tits diameter of $\partial X$ is infinite and $G$ is rank 1. 
%

Otherwise $$K=\bigcap _{n\in \mathbb Z}g^nA\ne \emptyset .$$ By lemma \ref{permutation} there is some $h\in G$ such that $hA\cap A=\emptyset $
and $a=hgh^{-1}$ translates along the axis $hl$. We remark that every min cut in $l$ contains $K$ so every min cut in $hl$ contains $hK\subseteq hA$. Since $K\cap hA=\emptyset $ we have that $l\cap hl$ does not contain a min cut. It follows that $l\cap hl$ is either a single point or empty. Let $p=[x,y]$ be a minimal path on $T$ joining $l$ to $hl$ (possibly $x=y$).
Let $g^r$ be such that $A\in [x,g^rx]$  and $a^m$ be such that $$hA\in [y,a^my].$$ Let  $b=a^mg^{-r}$. Then the intervals $$[b^{-1}A,A],[A,bA]$$
intersect only at $A$ so $b$ is hyperbolic (see picture). Also $hA\in [A,bA] $. It follows that $A\cap bA=\emptyset$ as $A$ and $hA$ are disjoint.
Hence $$d_T(b^s(A),A)=\epsilon >0$$ and we conclude that $G$ is rank 1 as in the first case.

\end{proof}
\section{Stabilizers}
We now assume that $G$ is rank 1 and $G$ acts nontrivially and non-nestingly on the cactus tree $T$.
%
%
%
%
%
%

\begin{Def}
For $W$ a wheel consider the collection $$\cM(W) = \{ M\, \text{continua in a wheel decomposition of  a finite subwheel of }W\}.$$
We order $\cM(W)$ by inclusion and we define a {\em gap} of $W$ to be the intersection of a maximal chain in $\cM(W)$.
Clearly if $W$ is finite, then the gaps of $W$ are just the continua of the wheel decomposition of $W$.  A gap $P$ which contains $B,C \in \cR$ with $B \in (W, C)$ is called {\em nonterminal}.  A  pair of gaps $P \neq Q$ of $W$ are called adjacent if for every finite subwheel $\hat W$ of $W$, $P,Q$ are contained in an adjacent pair of the wheel decomposition of $\hat W$.  (Notice that $P,Q$ may be in the same continua of the wheel decomposition of $\hat W$).  
\end{Def}

\begin{figure}[htbp]
\hspace*{-3.3cm}     
\begin{center}
                                                      
   \includegraphics[scale=0.600]{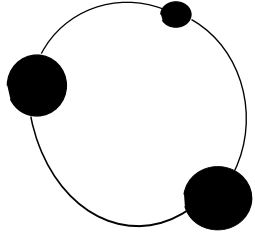}%
   \end{center}

\caption{A wheel with 3 gaps in a 2-thick continuum}
  \label{}
\end{figure}

Clearly when $P$ and $Q$ are adjacent gaps of $W$ then for some finite wheel $\hat W$,  $P \subseteq M_0$ and $Q \subseteq M_1$ for some adjacent pair $M_0,M_1$, It follows that $P \cap Q = M_0\cap M_1$ and so $$|P\cap Q| = k +|I| \ge 2$$ for $m \ge 3$.  Also notice that $\bd X $ is the union of all gaps of $W$.
\begin{Def} For $A$ a minimal  cut of $\bd X$, then $\bd X = M \cup N$ where $N$ and $M$ are continua with $M \cap N = A$.
We call $M,N$ a decomposition of $\bd X$ by $A$.  Clearly  such a decomposition is not unique.  
Let $$\cM(A) = \{ M\, \text{continua in a  decomposition of } \bd X \text{ by } A\}.$$
As before we order $\cM(A)$ by inclusion, and we define a {\em gap} of $A$ to be the intersection of a maximal chain in $\cM(A)$. 
 As before, a gap $M$ which contains $ B,C \in \cR$ with $B \in (A, C)$ is called {\em nonterminal}.  
\end{Def}
For $C$ a wheel or minimal cut, $C \in T$, the cactus tree of $\bd X$ and there is an injective function from the set of components of $T-C$ to the set of gaps of $C$, defined as follows.   If $D$ is a component of $T-C$ and $\cP$ be the cactus pretree of $\bd X$.  The elements of $D \cap \cP$ are not separated by $C$, therefore $D\cap \cP$ is contained in a gap $M$ of $C$, and our function sends $D$ to $M$.  This will be 1 to 1 by \cite{PS2}.

\begin{Lem}\label{L:gap contract} Let $X$ be a $m$-thick continuum $m\ge 2$. If $M$ is either 
\begin{enumerate}
\item the union of adjacent gaps of a wheel and $m>2$,
\item a non terminal gap (of a wheel or  inseparable cut)
\end{enumerate}
 then $\rad M \ge \pi$.
\end{Lem}
\begin{proof} 

 For case (1) say $M = P\cup Q$ where $P \subseteq M_0$ and $Q \subseteq M_1$ with $M_0\dots M_{n-1}$  a wheel decomposition of some finite wheel.  Since $\bd M$ is $m$-thick for $m \ge 3$,  $|P\cap Q| \ge 2$.    By Lemma \ref{L: notcontractible},  $\rad M \ge \pi$.

For case (2) say $W$ is the wheel or inseparable cut of which $M$ is  gap.  Since $M$ is non-terminal, there exists $\hat S,\hat T \in \cR$, and $\hat S,\hat T \subseteq M$ with $\hat S \in (W,\hat T)$.  Take $T$ a minimal cut of $\hat T$.  Now $ T$ decomposes $\bd X$ into continua $A,C$, so $A \cap C = T$ and $A \cup C = \bd X$ with
$W, \hat S \subseteq C$, so $A \subseteq M$.  Similarly $S$ decomposes $\bd X$ into $D,E$ with $W \subseteq D$ (so $E \subseteq M$) and $\hat T \subseteq E$.  Let $B= C \cap E$.  We must show that $B$ is connected. \begin{figure}[htbp]
\hspace*{-3.3cm}    
\begin{center}
                                                       
   \includegraphics[scale=1.00]{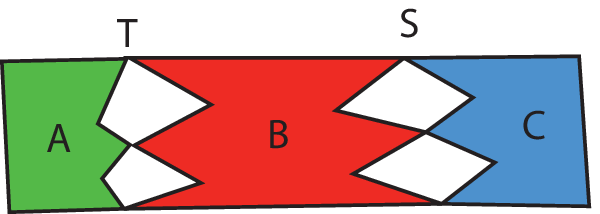}%
\end{center}

  \caption{}
\end{figure} 

 Suppose not, so $B =B_1 \cup B_2$ where $B_1$ and $B_2$ are non-empty disjoint closed subsets of $\bd X$.  Now $$\bd B_1 \cup \bd B_2 \subseteq S\cup T$$ and $$\bd B_1 \cap \bd B_2 = \emptyset .$$  It follows by minimality that $$|\bd B_1|= |\bd B_2|\geq m,$$ also $$\bd B_1 \cup \bd B_2 = S\cup T.$$  Thus $S \cap T = \emptyset$ and so $|\bd B_1|=|\bd B_2| =m$.  Each of the disjoint sets  $$(\bd B_1 \cap S) \cup (\bd B_2 \cap T)$$ and $$(\bd B_2 \cap S) \cup (\bd B_1 \cap T)$$ separates $B_1$ from $B_2$, so these are crossing minimal cuts. It follows that $S,T$ are in the same wheel which is a contradiction.  Thus $B$ is connected.   By Lemma \ref{L: notcontractible}, $\rad (A \cup B)\geq \pi$ so $\rad M \ge \pi$ as well.

\end{proof}

\begin{Lem}\label{L: rank1gap} Suppose that $G$ is rank 1 and non-nesting on the Cactus tree of $\bd X$, and that $\bd X$ is decomposed into continua $Y_1, \dot Y_i$, (so $\bd X= \cup Y_i$ and $Y_i \cap Y_j$ is finite for $i\neq j$) then:
\begin{itemize}
\item  For each $i,j$ there is a rank 1 hyperbolic element $h$ with $h^- \in\,$Int$\, Y_i$ and $h^+ \in\,$Int$\, Y_j$
\item  If $\cup \bd Y_i$ is contained in a wheel, $W$ then $Y_i$ contains a gap corresponding to  a branch of the minimal $G$ invariant subtree of the cactus tree of $\bd X$.
\end{itemize}\end{Lem} 
\begin{proof}The first claim follows from  \cite[Theorem 3.4]{BAL}.  For the second claim, there is a rank 1 hyperbolic element $h$ with $h^- \in\,$Int$\, Y_i$ and $h^+ \in\,$Int$\, Y_j$, and it follows that $h(W ) \subseteq Y_j$.  It follows that $$h^2(Y_j) \subsetneq Y_j,$$ so $h$ is hyperbolic in its action on the cactus tree and its axis runs into $Y_j$ so $Y_j$ contains a gap corresponding to a ray on this axis.
\end{proof}
\begin{Cor}\label{C: rank1gap} Suppose that $G$ is rank 1, acts non-nestingly the cactus tree of $\bd X$, and that $W$ is a wheel or inseparable cut of $\bd X$ with infinitely many gaps. Then the minimal $G$ invariant subtree $\hat T$ of the cactus tree has infinite valence at $W$.

\end{Cor}
\begin{proof} By infinitely many gaps, for any $n\in \N$, there are continua $Y_1, \dots Y_n$ such that $\bd X = \cup Y_i$ and $Y_i \cap Y_j$ is a finite subset of  $W$ for each $i,j$.
It follows from lemma \ref{L: rank1gap}, that the valence of $\hat T$ at $W$ is at least $n$.  
\end{proof}
We now generalize orthogonal projection to other angles.
\begin{Lem}\label{L:angles}  Let $L:\R \to X$ be a geodesic line and $\theta \in (0, \pi)$.   We will abuse notation and refer to $L(w)$ as $w$ and to $L(-\iy)$ as $-\iy$.  
If $\phi:X-L \to L$ is defined by $$\phi(z) = \sup \{w \in L : \angle_w(z, -\iy) > \theta\}$$ then $\phi$ is continuous, and $\angle_{\phi(z)}(z, -\iy) \ge \theta$.
\end{Lem}
\begin{proof} 
Fix $z \in X-L$.   Suppose that for $w> \hat  w\in L=\R$, we have   $$\pi >\angle_w(z,-\iy) \ge \angle_{\hat w}(z,-\iy)>0.$$  
\begin{figure}[htbp]
\hspace*{-3.3cm}  
\begin{center}                                                         
   \includegraphics[scale=1.00]{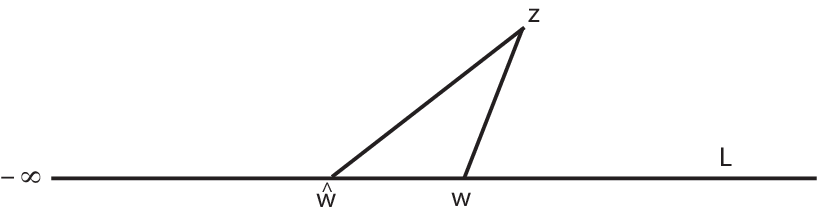}%
\end{center}

  \caption{}
  \label{}
\end{figure}
By the triangle inequality of angles,
$$\angle_{\hat w} (z,w)  \ge \pi-  \angle_{\hat w} (z,-\iy) \ge \pi-  \angle_{w} (z,-\iy) =  \pi-  \angle_{w} (z,\hat w).$$ Thus the angle sum of the triangle $\Delta(z,\hat w,w)$ is at least $\pi$.  So by \cite[II 2.9]{BRI-HAE} $\Delta(z,\hat w,w)$ is a Euclidean triangle.  Also $\angle_z(w,\hat w) =0$, so $\Delta(z,\hat w, w)$ is a line segment.   This contradicts $$\pi > \angle_w(z,\hat w)> 0.$$  We have shown that the function $f: L \to [0,\pi]$ defined by $$f(w)=\angle_w(z,-\iy)$$  is  strictly decreasing on $f^{-1}((0,\pi))$.
A sub-argument of the above shows that $f$ is non-increasing on $[0,\pi]$.  

Recall \cite[II 3.3]{BRI-HAE} that the function $(p,x,y) \to \angle_p(x,y)$ (on $X^3$) is upper semicontinuous, while for fixed $p$, the function $(x,y) \to \angle_p(x,y)$ (on $X^2$) is continuous. By this upper semicontinuity, $$\angle_{\phi(z)}(z, -\iy) \ge \theta .$$  

 Let $z_n \to z $, and let $w= \phi(z)$.  By definition $$\exists  (w_k)\subseteq (-\iy, w) \subseteq L =\R$$  with $$\angle_{w_k}(z,\iy) > \theta$$  for all $k$ and $w_k \to w$.  By continuity, for each  $k$ there exists $N_k$ such that for all $n > N_k$, $\angle_{w_k}(z_n,\iy) > \theta$.  It follows that for $n> N_k$,  $\phi(z_n) \ge w_k$.  Thus $$\varliminf \phi(z_n) \ge w.$$ 

   We now show that $$\varlimsup \phi(z_n) \le w.$$  Suppose not, then there is $\hat w > w$  and a subsequent $(z_{n_i}) \subseteq (z_n)$ with $\phi(z_{n_i}) > \hat w$.
 Thus $$\angle_{\hat w}(z_{n_i}, -\iy) > \theta $$ and by continuity $$f(\hat w)= \angle_{\hat w} (z,-\iy) \ge \theta .$$  Now by definition of $\phi$, $f(y) \le \theta$ for $y\in (w, \hat w]$.
 Since $f$ is non increasing $f(y) = \theta$ on $(w,\hat w]$ contradicting that $f$ is strictly decreasing on $f^{-1}(0,\pi)$.
 
   Thus $\varlimsup \phi(z_n) \le w \le \varliminf \phi(z_n)$ as required.
\end{proof}

We recall that a subset $Y\subseteq X$ \textit{coarsely separates} $X$ if sor some $K>0$, $X\setminus \nb _K(Y)$ has at least 2 connected components
that are not contained in any finite neighborhood of $Y$.
We say $H <G$ has \textit{codimension 1} if $H\cdot x_0$ coarsely separates $X$ for some (all) $x_0$.

\begin{Thm}\label{T:cut pair} Let $G$ be a group acting geometrically on a CAT(0) space $X$ with $G$ rank 1.  Suppose that $G$ doesn't split over a virtually cyclic subgroup.  If $A$ is a separable min cut  of $\bd X$ and $g$ is a hyperbolic element fixing $A$ with $\{g^\pm\} \subseteq A$, then $A= \{g^\pm\}$ and $\langle g\rangle$ is a codimension one subgroup of $G$.
\end{Thm}

\begin{proof} By \cite{PS2} the separable min cut $A$ will be a subset of a wheel $W$.

Claim: {\em We may assume $g(M) =M$ for every gap $M$ of $A$}.
Take $\hat W$ to be a finite 
subwheel of $W$ with $A\subseteq \hat W$,and $\hat W$ containing a min cut which crosses $A$.  By uniqueness of the wheel decomposition of $\hat W$ there are only two gaps of $A$, so $g^2(M)=M$ and we replace $g$ by $g^2$ if need be. This completes the claim.

We recall the fact (see \cite{ONT}) that geodesics in $X$ are
`almost extendable' i.e. there is an $E>0$ such that for any $d,e \in X$ 
there is a unit speed geodesic ray $\alpha:[0,\iy) \to X$ with $\alpha(0) =d$ and $d(\alpha(d(e,d)),e) < E$.  
  Since $G$ is rank one, {\bf we can and will always choose $\alpha$ so that $\alpha(\iy)$ is the  endpoint  of a rank 1 hyperbolic element of $G$.} 

Let $\bar X=X\cup \partial X$. Let $L$ be an axis for $g$, and let $x \in L$. 
Notice that since $$A \subseteq \fp g= \Lambda Z_g = \bd \,Min\, g$$ then for each $a\in A$ there is a half-flat $F$ attached along $L$ with $a \in \bd F \subseteq \bd X$.
Thus for any $a \in A$, $ \angle_x(g^-,a) =  d_T(g^-,a)$.
 Let $$\theta = \frac 1 2 \min\limits_{a \in A-\{g^-\}} \angle_x(g^-,a) = \frac 1 2 \min\limits_{a \in A-\{g^-\}} d_T(g^-,a),$$ and note that $0 < \theta \le  \frac \pi 2$.  Let $\tau $ be the translation length of $g$ and choose $D$ with 
$$D \gg \frac{4\tau+ 1 +  E}{\sin  \theta }  .$$

   We consider
the subrays $\alpha_\pm$ of $L$ from $x$ to $g^\pm$ respectively. For any $y \in \bar X$ we define 
$\alpha_y:[0, d(x,y)] \to \bar X$ to be the unit speed geodesic from $x$ to the  point $y$.
We define the following neighborhoods:
\begin{itemize}
\item $ U=\{y\in \bar X: d(\alpha_y(D),\alpha_-(D))<1 \}$
\item $V  =\{y\in \bar X: d(\alpha_y(D),\alpha_+(D))<1 \}$
\item For $a \in A-\{g^\pm\}$    we define
$$ O_a=\{y\in \bar X: d(\alpha_y(D),\alpha_a(D))<1 \}$$
\end{itemize}

Choose nonempty continua $\bar Z_1$, $\bar Z_2$ such that $\bar Z_1\cup \bar Z_2 =\bd X$ and $\bar Z_1 \cap \bar Z_2 =A$.  We have the disjoint open sets $Z_i = \bar Z_i -A$.

We claim that there is a $K>0$ such that a
$\nb_K (L) \cup U \cup V\cup \bigcup\limits_{a \in A-\{g^\pm\}} O_a$ separates $Z_1$ from $Z_2$.

\begin{figure}[htbp]
\hspace*{-3.3cm}  
\begin{center}                                                         
   \includegraphics[scale=0.85]{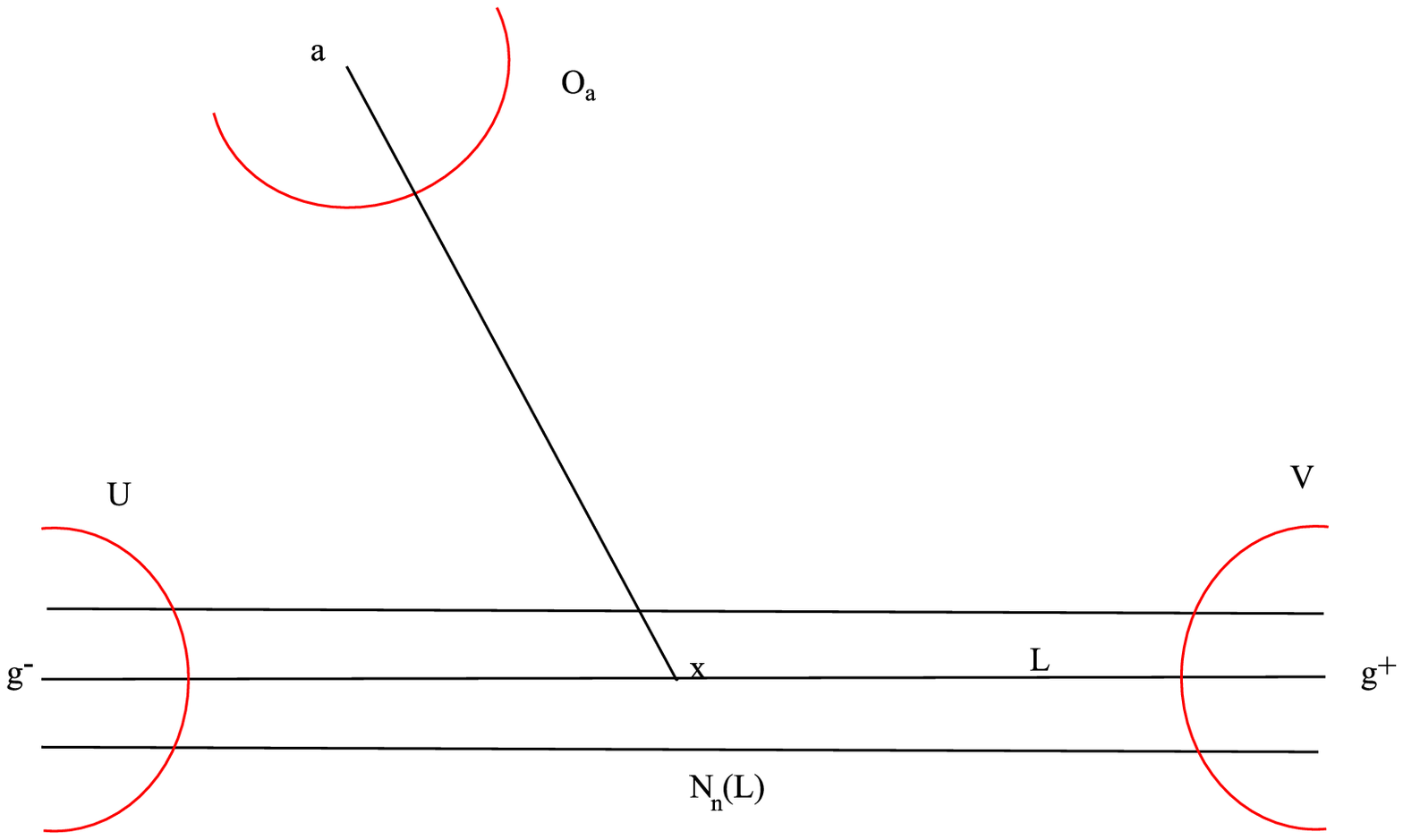}%
\end{center}

  \caption{}
  \label{}
\end{figure}
If not, for any $n$
there is a path $q_n$ in $$\bar X- \left[\nb_n(L)\cup U \cup V\cup \bigcup\limits_{a \in A-\{g^\pm\}} O_a\right]$$ from $Z_1$ to $Z_2$ (in fact we can choose $q_n$ such that 
only the endpoints of $q_n$ are in $\bd X$).  Some subsequence of $(q_n)$ will converge to a continua $Q$ in $\bar X$ from $Z_1$ to $Z_2$.  By construction $Q \subseteq \bd X$ and $A \cap Q = \emptyset$.  This is a contradiction and the claim is proven.

Let $\phi:X-L \to L$ be the function of Lemma \ref{L:angles}, so $$\phi(z) = \sup \{w \in L : \angle_w(z, -\iy) > \theta\}$$ (where we are taking $L=\R$), and  $\phi$ is continuous.

We claim that there is an $M>0$, and disjoint open sets $U_1$, and $U_2$ of $\bar X$ with
\begin{itemize}
\item $U_1 \cup U_2 =\bar X- [\nb_M(L) \cup \{g^\pm\}]$
\item $Z_i \subseteq U_i$ for $i =1,2$.
\end{itemize}
and this claim clearly implies the theorem.

Notice that there are two infinite
rays $r_1,r_2$ from $x$ to $Z_1,Z_2$ respectively with $\angle_x(g^-, r_i) = \theta$ . Indeed we consider rays from $x$ to, say,
points on the continuum $\bar Z_1$. The angle that these rays form with $L$
varies continuously and takes the values $\{0,\pi \}$ at
$\{g^\pm\}$. Therefore, for some such ray $r_1$,  $\angle_x(g^-, r_1) = \theta$.

Given $n\gg D + K $ let $$R_1=r_1(n+1),R_2=r_2(n+1).$$ If
$R_1,R_2$ are not contained in the same component of $X-N_n(L)$
then the claim is proven. Otherwise there is a path $p$ in
$X-N_n(L)$ joining $R_1$ to $R_2$. For every $y\in p- \{R_1,R_2\}$ we consider
$\phi(y)\in L$ and we pick an infinite ray $r_y$ from $\phi(y)$ such
that $$d(r_y(d(x,y)),y)< E$$ and $r_y(\iy)$ is the endpoint of a rank 1 hyperbolic. For $R_1,R_2$ we choose the corresponding
rays to be $r_1,r_2$ respectively. Clearly there are $y_1,y_2\in
p- \{R_1,R_2\}$ such that $$d(\phi(y_1),\phi(y_2)),  d(y_1,y_2) \ll 1$$   with $r_{y_1} (\iy) \in Z_1$ and $r_{y_2}(\iy) \in Z_2$.  

For $i=1,2$, let $M_i$ be the gap of $A$ with $r_{y_i}(\iy) \in M_i$.  Since $M_i -A \subseteq Z_i$,  $M_1 \neq M_2$.  By the previous claim, $g(M_i) = M_i$.

 Since $g$ acts by translation  on $L$ we may choose $k \in \Z$ with 
 \begin{itemize}
\item $d(g^k(\phi(y_1)),x), d(g^k(\phi(y_2)),x)<2\tau$,
\item $g^k(\phi(y_1), g^k(\phi(y_2) \in (-\iy,x) \subseteq L$.  
\end{itemize}
Notice that $$g^k(\phi(y_i)) = \phi(g^k(y_i))$$ since $g$ acts by isometries on $X$ and $L$.  
By the first claim, $g$ leaves $M_i$ invariant 
so $$g^k(r_{y_1}(\iy))\in M_1-A \subseteq Z_1$$ and $$g^k(r_{y_2}(\iy))\in M_2-A \subseteq Z_2.$$    

To arrive at our contradiction, it suffices to show that the ends of $g^k(r_{y_1})$ and $g^k(r_{y_2})$ are contained in the same component of 
$$X- \left[\nb_n(L)\cup U \cup V\cup \bigcup\limits_{a \in A-\{g^\pm\}} O_a\right]$$
For any ray $\alpha:[0,\iy)\to X$ with $\alpha(0) =x$, 
 if $$d(\alpha(D), \alpha_-(D))<1,\text{ then }\alpha([D, \iy]) \subseteq U$$ and if 
$$d(\alpha(D), \alpha_-(D))\ge 1\text{ then }\alpha([D, \iy]) \cap U =\emptyset .$$

Now let $\alpha_i:[0,\iy)$, $i=1,2$ be the unit speed geodesic ray from $x$ to $g^k(r_{y_i}(\iy))$.  
For $i=1,2$ let $z_i$ be the point on $$[g^k(y_i), g^k(\phi(y_i)]$$ with $$d(z_i, g^k(\phi(y_i)) =D.$$  Notice that by convexity, 
$$d(\alpha_i(D),z_i) \le  d(\alpha_i(D), g^k(r_{y_i}(D))  + d(g^k(r_{y_i}(D)),z_i) \le 2\tau + E   \ \ \ \ (*).$$  

We now show that $$d(z_i , \alpha_-(D)) \ge 1+ 2 \tau + E$$ which implies that $$d(\alpha_i(D), \alpha_-(D)) \ge 1.$$
Now by Lemma \ref{L:angles} $$\angle_{g^k(\phi(y_i))}(z_i, -\iy) \ge \theta ,$$ so  $$d(z_i,L) \ge D\sin \theta \gg1+ 4\tau + E$$ by the choice of $D$.

Now for $i=1,2$ let $\hat z_i$ be the point on $[g^k(y_i), x]$ with $d(\hat z_i, x) =D$.  By convexity $$d(z_i,\hat z_i)\leq 2\tau $$ so using $(*)$
$$d(\alpha_i(D), \hat z_i) \le E +4\tau.$$
Since $$\phi(g^k(y_i)) \in (-\iy, x)$$ it follows that $$\angle_x(\hat z_i, -\iy) < \theta .$$  Thus for $a \in A-\{g^-\}$, $$\angle_x(\hat z_i, \alpha_a) > \theta .$$  
This implies that $$d(\hat z_i, \alpha_a) \ge D\sin \theta\gg 1 +4\tau +E.  $$  Thus $d(\alpha_i(D), \alpha_a(D)) \gg1$.  
We have now shown that $$\alpha_i([D, \iy)) \cap  \left[ U \cup V\cup \bigcup\limits_{a \in A-\{g^\pm\}} O_a\right]= \emptyset$$

By convexity, the above calculations imply that $$d\left(g^k(y_i),  \left[ U \cup V\cup \bigcup\limits_{a \in A-\{g^\pm\}} O_a\right]\right) \gg 1+2\tau + E.$$  
Notice  $d(g^k(y_i), L) >n \gg K + E +1$.  Let $$m = \underset{i=1,2}{\min}\, d(x,y_i)\gg D.$$  Thus we have that

$$\alpha_i([m, \iy)) \cap  \left[  \nb_n(L) \cup U \cup V\cup \bigcup\limits_{a \in A-\{g^\pm\}} O_a\right]= \emptyset$$

  By construction $$d(\alpha_1(m),\alpha_2(m))\le 1+2\tau +E.$$  
However  $\alpha_1(D)$ is very far from $  \left[ U \cup V\cup \bigcup\limits_{a \in A-\{g^\pm\}} O_a\right]$  so, by convexity is $\alpha_1(m)$.   Thus
$\alpha_1([m, \iy))$ and $\alpha_2([m,\iy))$ are in the same component of $$X- \left[  \nb_K(L) \cup U \cup V\cup \bigcup\limits_{a \in A-\{g^{\pm}\}} O_a\right]$$ which is a contradiction.

\end{proof}
\begin{Thm}\label{T: cutcyclic} Let $A$ be a minimal inseparable cut of $\bd X$.  There is at most one nonterminal gap $M$ of $A$ with $\st M \cap \st A$ not virtually cyclic.
\end{Thm}
\begin{proof} If $\st A$ is virtually cyclic then we are done. Let $H < \st A$ be finitely generated and infinite.  Since every virtually cyclic subgroup of a CAT(0) group is contained in a maximal virtually cyclic subgroup, we may assume that $H$ is not virtually cyclic.  
  Passing to a finite index subgroup, we may assume that $H$ fixes $A$.

  First we consider the case where $\Lambda \langle Z_H, Z_{Z_H} \rangle$ is not separated by $A$.  In that case there is a gap $M$ of $A$ with 
$$\Lambda H \cup \fp H \subseteq M.$$  Let $\hat M$ be a non-terminal gap of $A$ distinct from $M$, and let $$K= \st \hat M \cap H.$$  By Lemma \ref{L:gap contract},  $\Lambda K \subseteq \hat M$.  Since $\Lambda K \subseteq M$, $\Lambda K \subseteq M\cap \hat M = A$.  Thus  by Corollary \ref{C: finite limit set}, $K$ is either virtually $\langle g \rangle$ where $g$ is hyperbolic with
$g^\pm \in A$ or $K$ is finite.  Thus $K$ is either finite or virtually $\langle g \rangle$.  Since CAT(0) groups have  maximal virtually cyclic subgroups, the proof is complete in this case.

We may now assume that $\Lambda \langle Z_H, Z_{Z_H} \rangle$ is  separated by $A$.  Thus by Lemma \ref{L: finite cutter} either  
\begin{enumerate}
\item $H$ is  virtually a subgroup of a $\Z^2$ with $\fp H \cup \Lambda H \subseteq \Lambda \Z^2$  separated by $A$ 
\item or there is a hyperbolic $g\in G$ with  $g^\pm \in A$ and $H$ virtually a subgroup of  $Z_g$,  and $Z_H$ commensurable with $\langle g \rangle$  and 
$\Lambda \Z_g$ separated by $A$.
\end{enumerate}
First consider the case (1) where $H$ is a virtually a subgroup of $\Z^2$.    Since maximal virtually abelian subgroups exists in CAT(0) groups, we may assume that 
$\st(A) =H$.    Let $M$ be a non-terminal gap of $A$ and suppose that $$K = \st(M) \cap \st(A)$$ is not virtually cyclic.  Then $K$ is finite index in  $H$.  Since $A$ separates $\Lambda K= \Lambda H$ there is $p \in \Lambda K  -M$.  However Lemma \ref{L:gap contract} say that $\Lambda K \subseteq M$, a contradiction.

Now consider the case of (2), so $\exists g\in G$ with  $g^\pm \in A$ with $H $ virtually a subgroup of $ Z_g$, $Z_H$ commensurable with $\langle g\rangle$   and 
$\Lambda \Z_g$ separated by $A$.  Let $a \in A-\{g^\pm\}$.  By Lemma \ref{L:SWE},  $a \in \fp H = \Lambda Z_H$,  but $Z_H$ is commensurable with $\langle g\rangle$ and so $\Lambda Z_H = \{g^\pm\}$, a contradiction.
\end{proof} 

\begin{Lem}\label{L: vwheel fixer} For any wheel $W$, $\fp W$ is virtually cyclic.
\end{Lem}
\begin{proof}
We may assume that  $\fp W$ is infinite.
By taking a subwheel if needed we  may assume that $W$ is a finite wheel.  
By Corollary \ref{C: finite limit set}
it suffices to show that  $\Lambda \fp W  \subseteq I$, the center of $W$.
Suppose not and  let $p \in \Lambda \fp W-I$.  
Clearly there is an adjacent pair  $N$,$M$ of  the wheel decomposition of $\bd X$ by $W$, so that
 $p \not \in M \cap N$.  
  By the uniqueness of a wheel decomposition, $\fp W$ stabilizes $M \cup N$.   By Lemma \ref{L:gap contract}  $\Lambda \fp W \subseteq M\cup N$, contradicting $p \not \in M\cup N$.   
\end{proof}
\begin{Thm} \label{T: notsimp} Let  $G$ be rank 1 and $\bd X$ $m$-thick for $m \ge 3$. Assume that $\bd X$ contains a wheel $W$ and that $G$ does not split over a $2$-ended group. Then the minimal $G$ invariant subtree  of the cactus tree of $\bd X$ is not simplicial. Moreover for any wheel there is at most one non-terminal gap of the wheel whose stabilizer  is not virtually cyclic. \end{Thm}
\begin{proof}  By Lemma \ref{L: rank1gap} $W$ is contained in the minimal $G$ invariant subtree $\hat T $ of the cactus tree $T$.   We assume by way of contradiction that $ \hat T$ is simplicial.
Since $G$ doesn't split over a virtually cyclic group, then $\hat T$ has no virtually cyclic edge stabilizers.
Thus it suffices to show that there is at most one non-terminal gap of $W$ whose stabilizer  is not virtually cyclic.  
Let $I$ be the center of the wheel $W$.  Suppose that   $H$ is a finitely generated non-virtually cyclic  subgroup of $\st W$.

First consider the case where $\Lambda H $ is not separated by $W$.  In that case there is a gap $M$ of $W$ with 
$\Lambda H \subseteq M$.  Let $\hat M$ be a non-terminal gap of $W$ distinct from $M$, and let $$K = \st \hat M \cap H.$$  By Lemmata \ref{L:gap contract} and \ref{C:Stab},  $\Lambda K \subseteq \hat M$.  Since $\Lambda K \subseteq M$, $\Lambda K \subseteq M\cap \hat M $ which is a finite subset of $W$.  Thus  by Corollary \ref{C: finite limit set}, $K$ is either virtually $\langle g \rangle$ where $g$ is hyperbolic with
$g^\pm \in W$ or $K$ is finite.  Thus $K$ is either finite or virtually $\langle g \rangle$.  Since CAT(0) groups have  maximal virtually cyclic subgroups, the proof is complete in this case.

\subsubsection*{Case: $\fp H= \emptyset$}  So we may assume that $I =\emptyset$ and $W$ is infinite.  Thus the size of a half cut of $W$ is $q = \frac m 2 >1 $, and it follows that $m \ge 4$.
Let $M$ be a non-terminal gap of $W$ with $K = \st M \cap H$ not finite.  By Lemma \ref{L:gap contract}, $\Lambda K \subseteq M$.  Using separation, $M$ determines a linear order on the set of half-cuts of  $W$, and  this also gives us a linear ordering on the gaps other than $M$.  Passing to a subgroup of index 2 if need be, we may assume that $K$ preserves this linear order.

By $\pi$-convergence and Lemma \ref{L:gap contract},  there is an element $k \in K$ which doesn't fix $W$, and so $k$ acts non-trivially on the linearly ordered set of half-cuts of $W$.   Clearly then $k$ has infinite order, and so $k$ is hyperbolic, and $\{k^\pm \}\subseteq \Lambda K  \subseteq M$.   

  Let $N',N''$ be a adjacent continua in the decomposition of some finite subwheel of $W$.   with $M \cap N = \emptyset$ where $N = N'\cup N''$.  Let $\epsilon $ be the Tits distance across $N$, that is 
$$\epsilon = d_T(\bd N \cap \bd N', \bd N \cap \bd N'')>0.$$  Lemma \ref{L:gap contract} tells us that $\rad N\geq \pi$, so there is a point $x \in N$ with $k^n(x) \to k^+$.  Similarly there is a point $y \in N$ with $k^{-n}(y) \to k^-$.   It follows that $k$ translates every half-cut of $W$ up in the linear order.  

 It follows that there are infinitely many disjoint $\langle k\rangle$ translates  of $N$ both above and below $N$ in the linear order.   Thus $d_T(N,M) = \iy$.  
 Now rechoosing $N$, we may assume that $k(N) \cap N$ is a half cut of $W$.  Now consider the nested union of continua $$Q=\bigcup\limits_{n=0}^\iy k^n(N).$$ 
The boundary of $Q$ consists of the half-cut below $N$ (namely $N \cap  k^{-1}(N)$) together with the point $g^+$ (by $\pi$-convergence).  This is a contradiction as $\frac m 2 + 1 < m$ for $m\ge 4$.   Thus $\st M \cap H$ is finite, and the proof of the case is completed.

We are left with the case where $\fp H$ is nonempty and $$\Lambda H \subseteq \Lambda \langle Z_H, Z_{Z_H} \rangle $$ is  separated by $W$.  Thus by Lemma \ref{L: finite cutter} either  
\begin{enumerate}
\item $H$ is  virtually a subgroup of a $\Z^2$ with $\fp H \cup \Lambda H \subseteq \Lambda \Z^2$  separated by $W$ 
\item or there is a hyperbolic $g\in G$ with  $g^\pm \in W$ and $H$ virtually a subgroup of  $Z_g$,  and $Z_H$ commensurable with $\langle g \rangle$  and 
$\Lambda \Z_g$ separated by $W$.  
\end{enumerate}
First consider the case (1) where $H$ is a virtually a subgroup of $\Z^2$.    Since by \cite{BRI-HAE}, maximal virtually abelian subgroups exists in CAT(0) groups, we may assume that 
$\st(W) =H$.    Let $M$ be a non-terminal gap of $W$ and suppose that $$K = \st(M) \cap \st(W)$$ is not virtually cyclic.  Then $K$ is finite index in  $H$.  Since $W$ separates $\Lambda K= \Lambda H$ there is $p \in \Lambda K  -M$.  However Lemma \ref{L:gap contract} say that $\Lambda K \subseteq M$, a contradiction.

Now consider the case of (2), so $\exists g\in G$ with  $g^\pm \in W$ with $H $ virtually a subgroup of $ Z_g$,  $\langle g\rangle$ a finite index subgroup of $Z_H$  and 
$\Lambda \Z_g$ separated by $W$.    
 Clearly $I \subseteq \fp H$, and by Lemma \ref{L:SWE}, $$ \fp H = \Lambda Z_H =\{g^\pm\},$$ so $I \subseteq \{g^\pm\}$.   Clearly $\{g^\pm\}$ is independent of our choice 
 of finitely generated subgroup $H <\st W$.  
 
  Notice that if $g \not \in \st W$,  then $g(W) = W'$ a different wheel, but
 $gHg^{-1}=H $ stabilizes $W'$.  The wheel $W'$ determines a unique gap $M$ of $W$, and since $H$ stabilizes $W$ and $W'$, then $H$ stabilizes $M$.  
 By Lemma \ref{L: notcontractible} $\rad M\geq \pi$, so by Lemma \ref{C:Stab},  $\Lambda H \subseteq M$.  Thus for any other nonterminal gap $\hat M$ of $W$,
 $$\Lambda (\st \hat M \cap H )\subseteq \hat M \cap M$$ which is a finite set.  Thus $\st \hat M \cap H$ is virtually cyclic by Corollay \ref{C: finite limit set}.  Since there are maximal virtually cyclic subgroups,  $\st \hat M \cap \st W$ is virtually cyclic as required.
 
 Thus we may assume that $g \in \st W$ and so we may assume that $g \in H$.   
 Since $\Lambda H$ is separated by $W$,  $ H/ \langle g\rangle< Z_g / \langle g\rangle$ has more than one end.  Now $Z_g / \langle g\rangle$ acts geometrically on a closed convex subset $T$  of $X$ with $\min g = T \times \R$  where $\R$ is an axis for $G$.  If  $H$ is virtually $\Z^2$ we argue as in case (1), so {\bf we may assume that $ H/ \langle g\rangle$ has more than 2 ends}.  In this case, $W$ will be infinite, and so have infinitely many gaps so
 by Corollary \ref{C: rank1gap} the vertex $W$ will have infinite valence in $\hat T$.   
 
 \begin{Lem}\label{S: septits}  In our setting, suppose that $g$ fixes $c \in W-\{g^\pm\}$, then there is a unique Tits geodesic $\alpha(c)$ from $g^+$ to $g^-$ through $c$ with $\ell(\alpha(c)) =\pi$.  Furthermore $\alpha(c)$ is not separated by $W$.  
 \end{Lem}
 \begin{proof}
Let $C$ be the unique half-cut of $W$ containing $c$.  Taking a power if need be, we may assume that $g$ fixes $C$.  Since $c \in \fp g = \Lambda Z_g$, a spherical suspension, $c$ is on a unique Tits geodesic $\alpha(c)$ of length $\pi$ from $g^+$ to $g^-$.    Consider a finite subwheel $\hat W $ of $W$ containing $C$.  Let $M_0$ and $M_1$ be the adjacent pair in the wheel decomposition for $\hat W$ with
 $M_0 \cap M_1 =C$.  Let $\alpha(c)$ be the unique Tits geodesic between $g^\pm$ through $c$.  Since  $\alpha(c) \subseteq \fp (g) =\Lambda Z_g $, a spherical suspension, $\alpha(c)$ is a limit of Tits geodesics of length $\pi$ from $g^+$ to $g^-$ in $\Lambda Z_g $.  
Distinct  Tits geodesics of length $\pi$ from $g^+$ to $g^{-1}$ intersect only in $g^+$ and $g^-$. It follows that either $\alpha(c) \subseteq M_0$ or $\alpha(c) \subseteq M_1$.
This shows that $\alpha(c)$ is not separated by $W$.
\end{proof}
Consider the case where $I \neq \{g^\pm\}$, so we may assume that $I = \{g^+\}$.  In the case where $g^- \not \in W$, then since $\Lambda Z_g$ is separated by $W$, there will be a $c \in W-\{g^\pm\}$ with $c \in \Lambda Z_g$ and $\alpha(c)$ separated by $W$.  This contradicts lemma \ref{S: septits}.  Thus $g^- \in W$.
Since $H /\langle g\rangle$ has more than two ends,  $\Lambda H$ is the suspension of a perfect set.  It follows that $\Lambda H$ is contained in the two gaps of $W$ adjacent to the half cut containing $g^-$.  Every other non-terminal gap $M$ of $W$ intersects these two gaps in only $\{g^+\}$.  It follows from Corollary \ref{C:Stab} that the limit set $\st M$ is at most 1 point, but by Corollary \ref{C: finite limit set}  $\st M$ must then be finite.  This is a contradiction, and {\bf we are left with the case where 
$I = \{g^\pm\}$.}

By Theorem \ref{T:cut pair}, $g$ cannot fix a separable minimal cut of $W$.  Thus there is at most one inseparable min cut $A$ of $W$ fixed by $g$.  This min cut (if it exists) will be the boundary of a single gap $M$ of $W$.  There is a linear order on the half cuts of $W$ which are not contained in $A$ and we may assume that that $g$ acts preserving this order.  Thus for any other gap $\hat M$ of $W$, $g(\hat M) \neq \hat M$.  For $\hat M$ corresponding to a branch of the minimal invariant subtree of the cactus tree of $\bd X$, by Lemmata \ref{C:Stab} and \ref{L:gap contract}, $\Lambda(H\cap \st \hat M)\subseteq \hat M$.  By Lemma \ref{L:zent},  $\Lambda H =\fp g \subseteq M$.  Thus $\Lambda(H\cap \st \hat M)\subseteq \hat M \cap M$ which is a finite set.  Thus by Corollary \ref{C: finite limit set}  $H\cap \st \hat M$ is virtually cyclic as required.
\end{proof}

\begin{Cor}\label{stabI} Suppose that the one-ended group $G$ acts on the CAT(0) space $X$ geometrically. If $\bd X$ is $m$-thick, $G$ acts nonnestingly on the Cactus tree $T$ of $\bd X$ and $I$ is an interval of $T$ containing at least three elements of the pretree $\mathcal R$ of the cactus tree $T$, then $\st (I)$ is virtually cyclic.
\end{Cor}
\begin{proof} 
Let $A,B,C \in I \cap \mathcal R$ with $B \in (A,C)$. By Theorems \ref{T: notsimp} and \ref{T: cutcyclic}  at most 1 gap of $B$ has non virtually cyclic stabilizer.  So either the gap of $B$ determined by $A$ or the gap of $B$ determined by $C$ has virtually cyclic stabilizer. Clearly $\st (I)$ will be a subgroup of this virtually cyclic group.  
\end{proof}

\begin{Thm}\label{nonnest} Suppose that the one-ended group $G$ acts on the CAT(0) space $X$ geometrically. If $\bd X$ is $m$-thick and $G$ acts nonnestingly on the Cactus tree $T$ of $\bd X$ then either $G$ splits over a 2-ended group and $m=2$, or $G$ acts non-nestingly, without fixed points and with virtually cyclic arc stabilizers on an $\R$-tree $\bar T$.
\end{Thm}
\begin{proof}  By \cite{PS1}, $\bd X$ is a metric continua without cut points, so $m\ge 2$.

Notice that if $G$ splits over a  two-ended subgroup, then an axis $L$ for  a hyperbolic element of said subgroup will coarsely separate $X$.  It follows that the endpoints of $L$ separate $\bd X$, and so  $m\le2$.   
By Corollary \ref{nofixed}, $G$ doesn't fix a point of the Cactus tree $T$ of $\bd X$.  Let $\hat T$ be the minimal $G$ invariant subtree of $T$.\hfill\break
{\bf Case I}: {$\hat T$ is simplicial}. \hfill\break By construction of $T$, for every pair of adjacent vertices of $\hat T$, one vertex of the pair is a minimal inseparable cut or wheel. It follows from Theorems \ref{T: notsimp} and \ref{T: cutcyclic}, that there is an edge of $\hat T$ with virtually cyclic stabilizer, and $G$ splits over a virtually cyclic subgroup as required. \hfill\break
{\bf Case II}:  {$\hat T$ is not simplicial}.\hfill\break 
By Corollary \ref{stabI} if the stabilizer of an arc $I$ is not virtually cyclic then it contains at most two elements of the pretree $\mathcal R$. In particular every element of $\mathcal R$ in $I$ has an element adjacent to it.
If $a\in \mathcal R$ is adjacent to $b,c$ in $I$, where both $b,c\in \mathcal R$ then by theorems \ref{T: cutcyclic}, \ref{T: notsimp} the stabilizer
of at least one of $[a,b],[a,c]$ is virtually cyclic. We define now a new pretree $\mathcal R '$ as follows: we define an equivalence relation on $\mathcal R$ by $a\sim b$ if and only
if the stabilizer of $[a,b]$ is not virtually cyclic. By lemma \ref{stabI} and theorems \ref{T: cutcyclic}, \ref{T: notsimp} $a,b$ are then adjacent in $\mathcal R$. The elements
of $ \mathcal R '$ are then the equivalence classes of this equivalence relation. We define betweeness in the obvious way: $[b]$ is between $[a],[c]$ in $ \mathcal R'$
if $b$ is between $a,c$ in $\mathcal R$. It is clear that $\mathcal R$ is preseparable with few gaps since $\mathcal R$ is, so we obtain a corresponding $\mathbb R$-tree $\bar T$ by \cite{PS2}.
Clearly $G$ acts on $\bar T$ non-nestingly and for every arc $I$ of $\bar T$, $\st I$ is virtually cyclic. 

We note that $\bar T$ is obtained from $T$ by collapsing some intervals of $T$ to points.

\end{proof}

\section{ Nesting homeomorphisms and the non-nesting quotient}

As shown by Levitt \cite{Le} the Besvina-Feighn-Rips theory of isometric actions on $\mathbb R$-trees extends to the case of non-nesting actions by homeomorphisms.
Even though omitting the non-nesting hypothesis is impossible (think of $Homeo(\R)$) there are cases where one may obtain a non-nesting action from a nesting action.
The objective of this section is to give some general conditions that ensure this (and apply in our case of actions of CAT(0) groups on the boundary). We recall:

\begin{Def} Let  $T$ be  an $\R$-tree without terminal points.
We say a homeomorphism $g$ of $T$ is \textit{nesting}  if there is an arc  $I \subseteq T$ with 
$g(I) \subsetneq I$. The arc $I$ is called a \textit{nesting interval} of $g$.  If $g$ inverts the linear order on $I$ then we say $g$ \textit{inverts a nesting interval}.  For $G$ a group acting by homeomorphisms on $T$, we
say $G$ is \textit{nesting} if it contains a nesting element. 
\end{Def}

We give below some (artificially constructed) examples of nesting actions on $\mathbb R$-trees to illustrate how nesting actions can give rise to non-nesting actions.
Consider the action of the free group $F_2=\langle a,b\rangle $  on the Bass-Serre tree $T$
corresponding to the free product decomposition $F_2=\langle a\rangle *\langle b\rangle $. For each vertex $v$ fixed by a conjugate
$gag^{-1}$ of $a$,  glue a copy of the interval $[0,1]$ to $T$ identifying $1$ to $v$. Now extend the action of $gag^{-1}$ on
the interval so that it fixes $0,1$ and acts by `translations' on $(0,1)$ (note that $(0,1)$ is homeomorphic to $\mathbb R$).
Clearly the action of $a$ on the new tree that we constructed is nesting. Note that this action is not minimal.

We give now a more interesting example similar to the actions that we consider in this paper. Consider the amalgamated product
$F_2*_{\langle a \rangle }G$ where $F_2$ is the free group of rank 2 and $G$ is any group. Let $T$ be the Bass-Serre tree corresponding to this
decomposition. Let $X$ be the Cayley graph of $F_2$. We compactify $X$ by adding its ends and we still denote it $X$ for simplicity.
Clearly $a$ fixes two ends of $X$, $\epsilon, \gamma $. If $v$ is a vertex of $T$ fixed by $a$ we replace $v$ by $X$. We explain how
we glue $X$ to the edges incident to $v$. If $e$ is the edge fixed by $a$ and $\partial _0 e$ is identified with $v$ then we identify
$\partial _0 e$ with $\epsilon $. If $ge$ ($g\in F_2$) is another edge incident to $v$ we identify its endpoint to the end $g\epsilon $
of $X$. We extend this equivariantly to the whole tree $T$, namely we replace
each vertex $gv$ ,$(g\in F_2*_{\langle a \rangle }G$) by a copy of $X$ and we identify with the incident edges extended equivariantly
the gluing for the vertex $v$. Clearly this construction gives rise to a nesting action on an $\mathbb R$-tree. \smallskip

Note that if a nesting element $g$ inverts a nesting interval then $g^2$ is also nesting and does not invert a nesting interval.
\begin{Lem}\label{L:int}Let $g$ be a  nesting homeomorphism of the $\R$-tree $T$ which doesn't invert a nesting interval.   There is a collection $\{I_\alpha\}$ of pairwise disjoint open intervals and rays of $T$ such that:
\begin{enumerate}
\item For each $\alpha$, $g$ acts on $I_\alpha$ by translation.
\item If $I$ is a convex linearly ordered subset of $T$ and $g$ acts by translation on $I$,
then $I=I_\alpha$ for some $\alpha$.
\item  $g$ fixes each point of the closure of $\hu \left[ \cup I_\alpha \right] -\cup I_\alpha $. 
\item For all $\alpha, \beta$, the convex hull, $\hu \left[ I_\alpha \cup I_\beta\right]$ is a linearly ordered subset of $T$.  
\end{enumerate}
\end{Lem}
\begin{proof}
There exists $[a,b] \in T$ with 
$g([a,b]) \subsetneq [a,b]$. Since $g$ doesn't invert $[a,b]$, by the fixed point theorem we may assume that $g(b) =b$ and that $g$ has no other fixed points in $[a,b]$,so $g(a) \in (a,b)$. 
Notice that $g$ cannot act by translation on a line of $T$.  
  It follows that the set $\{g^n(a) :\, n \in \Z\}$ is linearly ordered.  Let $I_a$ be the convex hull $\{g^n(a) : n \in \Z\}$.  Notice that $I_a$ is an open interval or ray of $T$ having $b$ as one of its endpoints, and that $g$ acts on $I_a$ by translation.  
 Repeating this process give $\{I_\alpha\}$. This proves (1) and (2).
 
 Suppose that for some $\alpha, \beta$,  $I_\alpha \cap I_\beta\neq \emptyset$. 
 Let $c \in I_\alpha \cap I_\beta$.  Since $g$ act by translation on both $I_\alpha$ and $I_\beta$,  $I_\alpha$ is the convex hull of the set $\{g^n(c):\,n \in \Z\}$, as is 
 $I_\beta$.  Thus $I_\alpha = I_\beta$, proving pairwise disjoint.
 
 Suppose that for some $\alpha, \beta$, $$\hu \left[ I_\alpha \cup I_\beta\right]$$ is not linearly ordered.  Then interchanging $\alpha $ and $\beta$ if need be, there is $a \in I_\alpha$ such that $$\hu\left[ \{a\} \cup I_\beta\right] \cap I_\alpha = \{a\}.$$  Notice that $$\hu\left[ \{g(a)\} \cup g(I_\beta)\right] \cap I_\alpha = \{g(a)\}$$ and so $g(I_\beta) \cap I_\beta =\emptyset$ which contradicts the fact that $g$ acts on $I_\beta $.   Thus (4) is true.
 
 Lastly let $$c \in \hu \left[ \cup I_\alpha \right] -\cup I_\alpha .$$ Then there exists $\alpha, \beta$ with $c \in \hu\left[ I_\alpha \cup I_\beta \right] $.  Let $a$ be an endpoint of $I_\alpha$ in $\hu\left[ I_\alpha \cup I_\beta \right]$ and $b$ be an endpoint of $I_\beta$ in $\hu\left[ I_\alpha \cup I_\beta \right]$. Clearly $g$ fixes $a$ and $b$.  If $c$ is not fixed by $g$, then either $g(c) \in (a,c)$ or $g(c) \in (c,b)$ either way, $c$ is in $I_\gamma$ for some $\gamma$ by construction and we have proven (3).
\end{proof}

\begin{Def}For $g$ a nesting homeomorphism of the $\R$-tree $T$ which doesn't invert a nesting interval, the intervals $\{I_\alpha\}$ of Lemma \ref{L:int} are called the {\em intervals of translation} of $g$. The closure of the convex hull of $ \cup I_\alpha$ is called the \textit{nesting subtree} of $g$, denoted $T_g$. 
More generally, if $g$ is a homeomorphism of the $\R$-tree $T$, $I$ is a maximal interval such that $g$ does not fix any point of $I$ and $g(I)=I$ then we say that
$I$ is an \textit{axis} of $g$. Clearly if $g$ is non-nesting then $g$ has a unique axis.

\end{Def}

\begin{Lem}\label{L:commute}Let $g,h$ be two commuting nesting homeomorphisms of the $\R$-tree $T$ which don't invert a nesting interval.  Let $I=(a,b)$ be a nesting interval of $g$
and let $J=(c,d)$ be a nesting interval for $h$. If $b\in (c,d)$ then the intervals $h^n (I),\, (n\in \Z)$ are all disjoint nesting intervals for $g$ and are contained in $J$. Moreover
for every $n\ne 0$ $h^ng^k$ is a nesting homeomorphism with nesting interval $J$.
\end{Lem}
\begin{proof} Without loss of generality we may assume that $(b,d)$ does not intersect $I$ and $h(b)\in (b,d)$.
Let $s=h^{-1}b$. Suppose $a\notin (s,b)$ and let $(s,b)\cap (a,b)=(t,b)$. Then $gt\ne t$, so $gs\ne s$ and $hgs\ne b$. However $g,h$ commute and $gb=b$ so 
$$hgs=ghs=gb=b \,.$$
Therefore $(a,b)\subset (h^{-1}b,b)$. It follows that the intervals  $h^n (I),\, (n\in \Z)$ are all disjoint. Clearly $$gh^na=h^nga=h^na$$ so $g$ fixes $h^na$ and similarly fixes $h^nb$.
For  fixed $x\in I$ $$I=\bigcup _{k\in \Z}[g^kx,g^{k+1}x]$$ so
$$h^n(I)=\bigcup _{k\in \Z}[g^kh^nx,g^{k+1}h^nx]$$ so $h^nI$ is an interval of translation for $g$. Note finally that $h^ng^kb=h^nb\in (b,d)$ and $g(d)=d$ since
$g$ fixes all $h^nb$ which limit to $d$ and $g$ is a homeomorphism. So $h^ng^k$ is nesting. Clearly $J$ is a nesting interval for this homeomorphism.

\end{proof}


 \begin{Def}\label{overlap} Let $G$ be a finitely generated group acting nestingly on an $\R$-tree $T$. If for any two 
 intervals of translation $I,J$ with $I\cap J\ne \emptyset $  no endpoint of $I$ is contained in $J$ we say that $G$ acts on $T$ with \textit{non-overlapping translation intervals}.
 
 \begin{figure}[h]
\includegraphics[width=6in ]{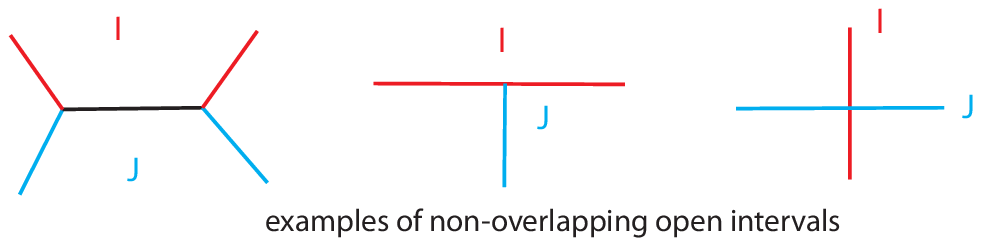}
\vspace{0.1in}
 \caption{}
\end{figure}

\begin{figure}[h]
\includegraphics[width=6in ]{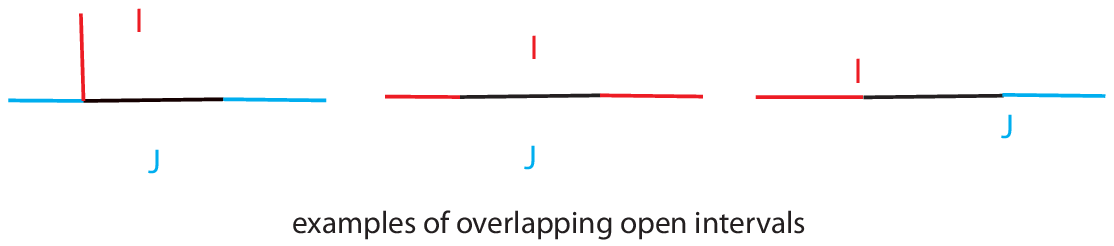}
\vspace{0.1in}
 \caption{}
\end{figure}

\end{Def}
\begin{Def} Let $G$ be a finitely generated group acting nestingly on an $\R$-tree $T$. A \textit{cross component} of  $T$ is a subtree $T'$ maximal with the property that for any two points $x,y \in T'$ there is a finite sequence of intervals of translation $I_1, \dots I_n$ of $g_1, \dots g_n$ respectively with $x \in I_1$,  $y \in I_n$ and $I_i \cap I_{i+1} \neq \emptyset$ for all $1\le i <n$. 

If $Q$ is a cross component
of $T$ then an element of $\overline  Q\setminus Q$ is called an \textit{endpoint} of $Q$.

\end{Def}

\begin{Lem} \label{cross-ends}
Suppose that a group $G$ acts minimally and nestingly on the $\mathbb R$-tree $T$ with non-overlapping translation intervals. If a cross-component $Q$ of $T$ is not an interval then the stabilizer of some endpoint
of $Q$ is a proper subgroup of the stabilizer of $Q$. If a cross-component $Q$ of $T$ is an interval then $G$ acts non-nestingly and non-trivially on an $\R$- tree  with arc stabilizers virtually conjugate into the stabilizer of $Q$. 
\end{Lem}
\begin{proof} Assume that $Q$ is not an interval and let $e$ be an endpoint of $Q$. If the stabilizer of $e$ is a proper
subgroup of the stabilizer of $Q$ the lemma is shown. Otherwise, since the action of $G$ is minimal, $Q$ has at least one more endpoint, say $f$.
Since $Q$ is not an interval there is an element $g$ translating along a nesting interval $I$ such that $f$ is not an endpoint of $I$.
It follows then that $g$ does not fix $f$.

If $Q$ is an interval then for any element $g$ in the stabilizer of $Q$ we alter its action as follows: if $g$ fixes  the ends of $Q$
in the new action $g$ fixes $Q$ pointwise while if $g$ reverses the ends of $Q$ in the new action $g$ acts as inversion (a flip).
Finally we collapse the part of the tree that does not lie in the union $\bigcup _{h\in G} hQ$ to obtain a new minimal $G$ $\R$-tree, say $T'$.   Any nesting interval of $T'$ gives rise to a nesting interval of $T$ which will violate non-overlapping translation intervals.  Thus $G$ acts non-nestingly on $T'$, and arc stabilizers in $T'$  will stabilize some translate of $Q$.  
\end{proof}

\subsection{Pretrees} 

By remark \ref {R:order} of section 2 we have the following definition of pretrees which is equivalent to the one given earlier:

\begin{Def}[see \cite{BOW5}]Let $\cP$ be a set.  A function $(.,.) : \cP^2 \to 2^\cP$ is called a pretree structure on $\cP$ (so $(a,b) \in \cP^2,\, \forall a,b \in \cP$)
if the following conditions are satisfied
\begin{enumerate}
\item $(x,y)=(y,x),\, \forall x,y \in \cP$
\item $[x,z] \subset [x,y]\cup [y,z],\, \forall x,y,z \in \cP$
\item If $z \in (x,y)$, then $x \not \in (z,y),\, \forall x,y,z \in \cP$
\end{enumerate} where we define
$[a,b]= (a,b) \cup \{a,b\},\, \forall a,b \in \cP$.
The set $\cP$ with this interval relation is called a pretree. 
\end{Def}
\begin{Def}
Let $\cP$ be a pretree.
We say distinct $a,b \in \cP$ are {\em adjacent} if $(a,b) = \emptyset$.  A subset $A \subset \cP$ is called {\em convex} if  $(a,b)\subset A,\, \forall a,b \in A$.   A collection of sets $\cA$ of $\cP$ is called colinear if  for any finite subset $S \in \cup \cA$, $S$ is contained in an interval of $\cP$.
\end{Def}

\begin{Thm}\label{T:pretreequotient} Let $\cP$ be a pretree and $\cR$ be a partition of $\cP$ into convex subsets.  If we define an interval structure on $\cR$ by
$A \in (B,C)$ if and only if $A,B,C \in \cR$ are distinct and there exist $a\in A$, $b\in B$ and $c \in C$ with $a \in (b,c)$, then $\cR$ is a pretree.
\end{Thm}
\begin{proof}  Clearly $\cR$ satisfies condition (1).  For condition (2), let $A,B,C \in \cR$, and let $D \in (A,B)$.  Thus $A,B,D$ are distinct and there are $a\in A$, $b \in B$, and $d\in D$ with $d \in (a,b)$ in $\cP$.  For any $c \in C$, $[a,b]\subset [a,c] \cup [c,b]$ so $d \in [a,c] \cup [c,b]$.  Thus $D \in [A,C]\cup[C,B]$.  

For condition (3), let  $A,B, D\in \cR$ with $D \in (A,B)$. Then $A,B,D$ are distinct and there exist $a\in A$, $b \in B$ and $d \in D$ with $d\in (a,b)$.  Suppose by way of contradiction that  $A \in (D,B)$,  then there exists $\hat a \in A$, $\hat b \in B$ and $\hat d \in D$ with $\hat a \in (\hat d, \hat b)$.    However by condition (2) in $\cP$,  
$$ \hat a \in [\hat d,d] \cup[d,\hat b] \subset D \cup [d,\hat b]$$ so $\hat a \in (d,\hat b)$.  Similarly $$d \in (a,b) \subset [a,\hat a] \cup [\hat a,b] \subset [a,\hat a]\cup[\hat a,\hat b] \cup[\hat b,b] \subset A \cup [\hat a, \hat b] \cup B $$ so $d \in (\hat a, \hat b)$  which contradicts condition (3) in $\cP$.  
\end{proof}
There is an obvious way to insert a linearly ordered set between adjacent points of a pretree.  The following theorem belabors this process.
\begin{Thm}[see \cite{SWE}]\label{T:preinsert}
Let $\cP$ be a pretree and  consider a collection $\{ \{-\iy_\omega, \iy_\omega\}: \omega \in \Omega \}$ of adjacent pairs of $\cP$.   For each $\omega \in \Omega$ choose a linearly ordered set $J_\Omega$.  Define  $\cR = \bigcup\limits_{\omega \in \Omega}J_\omega \cup \cP$.  Extent  the betweeness relation on $\cP$ to a relation on $\cR$ in the following  steps:
\begin{enumerate} [I]
\item For $a \in J_\omega,\, \, p,q \in \cP$, we define $q \in (a,p)\iff  q\in [-\iy_\omega, p) \cap [\iy_\omega, p)$.  
\item For $a \in J_\omega,\,\, b \in J_\psi, \, q \in \cP$ we define $$q\in (a,b) \iff q \in [\iy_\omega, \iy_\psi]\cap  [-\iy_\omega, \iy_\psi]\cap  [\iy_\omega, -\iy_\psi]\cap  [-\iy_\omega, -\iy_\psi]$$
\item For $a \in J_\omega,\,\, b,c \in \cR\setminus J_\omega$ we define $a \in (b,c) \iff  \pm \iy_\omega \in [b,c]$.
\item For $a,b \in J_\omega$, then $(a,b)$ is defined as a subset of the linearly ordered set $J_\omega$.
\item For $a,b \in J_\omega, \,\, c \in \cR\setminus J_\omega$ with $a<b$ in the linear order on $J_\omega$, then  $b\in (a,c) \iff \iy_\omega \in (a, c]$
\item For $a,b \in J_\omega, \,\, c \in \cR\setminus J_\omega$ with $a>b$ in the linear order on $J_\omega$, then  $b\in (a,c) \iff -\iy_\omega \in (a, c]$
\end{enumerate}
Then $\cR$ is a pretree.  
\end{Thm}
The proof with its unenumerable trivial cases is left to the obsessive compulsive reader.  

We assume now that a finitely generated group $G$ acts on an $\mathbb R$-tree $T$ with non-overlapping translation intervals.
We will define a new tree $S$ called the \textit{non-nesting quotient} of $T$.


An informal description of $S$ is as follows: we collapse each cross component of $T$ to a new vertex and for each endpoint $e$ of a cross component $Q$ we introduce a
vertex $v_{e,Q}$ and an edge joining the new vertex to $Q$ and $e$. It is clear that both these operations
produce a pretree, which we may complete to get an $\R $-tree.

Another way to think of this is as follows: Replace each endpoint $e$ of a cross-component $Q$ by an edge with
endpoints $v_{e,Q},e$ (so that $v_{e,Q}$ is between $Q,e$).
Then collapse all subtrees corresponding to cross components to points. 

We define now $S$ formally by introducing a pretree $\mathcal S$ which we complete to a tree.
First let $ \mathcal  S_0$ be the partition of $T$ into cross components and singletons (i.e. points that lie in no cross component).  Since $ \mathcal S_0$ is a partition of a pretree into convex subsets, it inherits a pretree structure via theorem \ref{T:pretreequotient}.  

For $e$, an endpoint of the cross component $Q$, and $e\in E \in \mathcal S_0$  then $Q$ and $E$ are adjacent in $\mathcal S_0$.  

We now define $\mathcal S = \mathcal S_0 \cup \{v_{e,Q}:\, e$ is an endpoint of the cross component $Q \}$.  The linearly ordered set $\{v_{e,Q}\}$ will be inserted between the adjacent pair $Q,E$ where $e \in E \in \mathcal S_0$.  The pretree structure on $\mathcal S_0$ extends to a structure on $\mathcal S$ via theorem \ref{T:preinsert}.
By theorem 5.4 of \cite{PS2} we may complete this pretree $\mathcal S$ to an $\R$-tree $S$.


Note that each vertex $v_{e,Q}$ of $T$ corresponding to some end-point $e$ is connected by
an edge to $Q$ and by another edge either to $e$ or to the cross-component containing $e$ if $e$ is contained in some cross-component.

Clearly each cross-component has at least one end-point and if a cross-component is not an interval then its stabilizer is not virtually $\mathbb Z$.
The action of $G$ on $T$ induces an action of $G$ on $S$.

\begin{Lem} \label{quotient}
The action of $G$ on $S$ is non-nesting. If $e$ is an endpoint of a cross-component $Q$ then the stabilizer
of the edge $[e,v_{e,Q}]$ is equal to the stabilizer of the end of $Q$ corresponding to $e$.
\end{Lem}
\begin{proof} Assume that $g$ acts nestingly on $S$ with $g([A,B])\subset [A,B]$, $g(B)=B$ and $g(A)\in (A,B]$.
We may assume that
 $B$ is the unique fixed point of $g$ in $[A,B]$. Since $v_{e,Q}$ was inserted between the adjacent points $Q$ and $E\ni e$ it follows that if $g$ fixes $v_{e,Q}$, then it either fixes or interchanges 
 $Q$ and $E$.  Thus $B \neq v_{e,Q}$. 
Since the group $G$ is finitely generated by lemma \ref{L:int} every interval of $T$ intersects at most countably
many cross-components. It follows that 
we may assume that $A$ is a point of $T$ ( and so $A$ is not contained in a cross component) . Note that if $B$ is also a point of $T$ then $g$ acts nestingly on the interval $[A,B]$ of $T$,
so $[A,B)$ is contained in a cross component of $T$, hence $g$ does not act nestingly on $[A,B]$ in $S$.

Assume now that $B$ is a cross-component and let $p$ be a geodesic in $T$ from $A$ to $B$. If $p$ goes through an endpoint $x$ of $B$
then $x$ is fixed by $g$ (otherwise $g(A)$ wouldn't lie in $[A,B]$) implying that  $g$ fixes $v_{x,B}\in (A, B)$ which is a contradiction.
Similarly if $p$ does not go through an endpoint of
$B$ and $x$ is the first point of $p$ lying on $B$ then $x$ is fixed by $g$. It follows that the interval $[A,x]$ is a nesting interval
of $g$ in $T$, and that $A$ is contained in a cross component of $T$ which is a contradiction. So the action
of $G$ on $S$ is non-nesting.

Assume that $g$ fixes the edge $[e,v_{e,Q}]$. Then $g(Q)=Q$ and $ge=e$. Therefore $g$ fixes the end of $Q$ corresponding to $e$.
Conversely if $g$ fixes the end $e$ of $Q$ it is either a nesting element with nesting axis contained in $Q$ or it is an elliptic element
fixing an interval of $Q$ which has $e$ as an endpoint. In the latter case $gQ\cap Q$ contains an interval which implies that $gQ=Q$.
It follows that $g$ fixes the edge $[e,v_{e,Q}]$. 
\end{proof}

\begin{Thm} \label{T:nesting}
Suppose that a finitely generated group $G$ acts minimally and nestingly on the $\mathbb R$-tree $T$ with non-overlapping translation intervals.
 Then $G$ acts non-nestingly on an $\R$-tree $S$ (without proper $G$-invariant subtree) with every arc stabilizer of $S$  stabilizing the  end of a cross-component of $T$.  
\end{Thm} 
 
\begin{proof} 
Consider the nesting quotient $\hat S$ of $T$. Then $G$ acts on $\hat S$ and the action of $G$ on $\hat S$ is non-nesting, and has no proper $G$-invariant subtree.

Let $e$ be an endpoint of a cross-component $Q$. Note that the stabilizer of the segment $E=[e,v_{e,Q}]$ is equal to the stabilizer of the end $e$ of $Q$ by lemma \ref{quotient}. 
By collapsing the connected components of $\hat S\setminus G\cdot E$ to points we obtain an $\R$-tree $S$. $G$ acts non-nestingly on $S$, $S$ has no proper $G$-invariant subtree, 
and arc stabilizers of $S$ are conjugate into an end stabilizer of $Q$.


%
 

\end{proof}

\section{Rank 1, Nesting version}

Let $Z$ be be the boundary of a CAT(0) space $X$ on which $G$ acts geometrically.
We assume that $Z$ is an $m$-thick continuum with cactus tree $T$.  $G$ acts  as a $\pi$-convergence group on $Z$.  Thus $G$ acts by homeomorphism on the $\R$-tree $T$ which has no terminal points (by construction). The aim of this section is to show that if the action is nesting and $G$ does not virtually split over a 2-ended group then $G$ is rank 1.

\begin{figure}[htbp]
\hspace*{-3.3cm}  
\begin{center} 
   \includegraphics[viewport=0 280 400 600,height=8cm,clip]
   {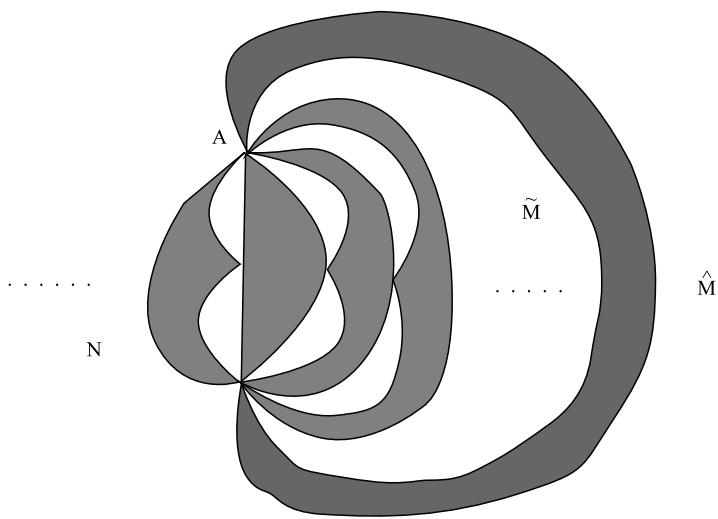}
   \end{center}
 \caption{The positive and the attracting continuum of a translation}
  \label{7.2Def.eps}
\end{figure}

\begin{Def} Let $A \in \cP$ the Cactus pretree for $Z=\bd X$.  We say that $g$ an isometry of $X$ \textit{translates along} $A$ if $A \in (g^{-1}(A), g(A))$ and $g$ doesn't invert the obvious interval of $T$.   \end{Def} 

Notice that when $A$ is a min cut, then there are continua $M,N$ of $\bd X$ with $$g^{-1}(A) \subset N , \,g(A) \subset M,\, M \cap N = A$$ and $M \cup N = \bd X$.  Since $g$ doesn't invert an interval, $g(M) \subsetneq M$ and $g^{-1}(N) \subsetneq N$.  When $A$ is a wheel we can replace $A$ with a minimal cut contained in the wheel which separates $g^{-1}(A)$ and $g(A)$ and then use this minimal cut instead.  In either case we have a minimal cut subset  $A$ decomposing $\bd X$ into $M$ and $N$ with $g(M) \subsetneq M$ and $g^{-1}(N) \subsetneq N$. 

We introduce now some notation that will be useful in the sequel:

\begin{Def} We define the {\em positive continuum} $\hat M$ of the $g$ translation along $A$ to be
$$\hat M= \bigcap\limits_{n>0} g^n(M) .$$

We define the {\em attracting continuum} $\tilde M$ of the $g$ translation along $A$ to be the intersection of the closure of $\bigcup\limits_{n>0} g^{n}(N)$ with the positive continuum of $g$. So
$$\tilde M =\overline {\bigcup\limits_{n>0} g^{n}(N)}\cap \hat M .$$

We define the {\em negative continuum} $\hat M$ of $g$ to be
$$\hat N= \bigcap\limits_{n<0} g^n(N) $$
and the {\em repelling continuum} $\tilde N$ of $g$  by
$$\tilde N =\overline {\bigcup\limits_{n<0} g^{n}(M)}\cap \hat N .$$
Let $\hat A=\tilde M\cap \tilde N= \hat M \cap \hat N \subset A$. 
If $I$ is the interval of $\cP$ spanned by $g^n(A),n\in \Z$ then these continua and $\hat A$ are clearly independent of any choice of min cut subset in $I$.   \end{Def} 



\begin{Thm}\label{T:miss}  Let $X$ be a proper CAT(0) space and $g$ an isometry of $X$ such that there is a compact $M \subset \bd X$ with $g(M)\subset \Int M$. Then $g$ is rank 1, and $\{g^+\} = \cap g^n(M) $
\end{Thm}
\begin{proof} Clearly $g$ has infinite order and so is a hyperbolic isometry.   Let $A = \bd M$, the topological boundary of $M$ in $\bd X$.  The points $g^\pm$ are fixed by $g$, and $g(A) \cap A = \emptyset$ so $g^\pm \not \in A$.
Recall Ballmann's dichotomy \cite{BAL} that if $g$ is rank one, then $g^+$ and $g^-$ are at infinite Tits distance from all other points in $\bd X$, and if $g$ is not rank 1 then  $d_T(g^+,g^-) =\pi$.

We first show that $d_T(A,g^+) = \iy$.  Suppose not.   By lower semi-continuity of the Tits metric, since $A$ is compact there is a Tits geodesic $\alpha$ from $A$ to $g^+$ of length $d_T(A,g^+)$.  Since $g$ acts on the Tits boundary by isometries, $$d_T(A, g^+) = d_T(g(A), g^+).$$ However if $g^+\in M$ (which implies that $ g^+ \in g(M)$) then  $\alpha$ passes through $g(A)$.  Thus $$d_T(g(A), g^+) < \ell(\alpha)= d_T(A,g^+)$$ which is a contradiction.   If on the other hand $g^+\not \in M$ then similarly  $g(\alpha)$ passes through $A$ and so $$d_T(A, g^+) < \ell(g(\alpha)) = \ell(\alpha) =d_T(A,g^+)$$ which is ridiculous.  Thus $d_T(A,g^+) = \iy$ and it follows from the above dichotomy that $d_T(A,g^-)= \iy$.  

Now by $\pi$-convergence $g^n(A) \to g^+$, and $g^{-n}(A) \to g^-$.  It follows that $g^+ \in \Int M$ and $g^- \not \in M$.  Thus any alleged Tits geodesic from $g^-$ to $g^+$ must pass through $A$, but this is ludicrous as the $A$ is at infinite Tits distance from $g^+$, so $d_T(g^-,g^+) = \iy$ and $g$ is rank 1.  Since $g^- \not \in M$, by $\pi$-convergence  $\{g^+\} = \cap g^n(M)$.  
\end{proof}
\begin{figure}[h]
\includegraphics[width=6.0in ]{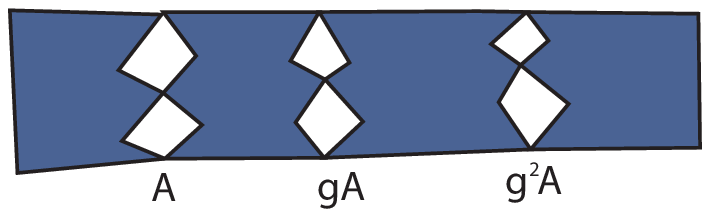}
\vspace{0.1in}
 \caption{}
\end{figure}
 
\begin{Cor} \label{C:disjoint} Let $X$ be a proper CAT(0) space and $g$ an isometry of $X$ such that $g$ acts on the cactus tree of $X$ as a translation along an interval $I$.
 If $I$ contains two disjoint min cuts  then $g$ is rank 1, $\{g^+\} = \tilde M$ and $\{g^-\}= \tilde N$.
\end{Cor}
\begin{proof}  
Assume $I$ contains two disjoint min cuts $A_1$ and $A_2$. Say $A_i$ decomposes $Z$ into $N_i, M_i$.  We may assume that $A_{3-i} \subset M_i$ and that $g$ translates in the direction from $A_1$ to $A_2$.  We may assume that $gA_1 \subset N_2$.  It follows that $gM_1 \subset N_2 \subsetneq M_1$ and the result follows from the theorem.  
\end{proof}


\subsection*{Setting}
Suppose now that $G$ acts nestingly on $T$, so there is $g\in G$ and $[A,B] \in T$ with 
$g([A,B]) \subsetneq [A,B]$.
So $g$ nests with nesting cut $A$. We will show that $G$ is rank 1. Note that $G$ doesn't virtually fix a point of $Z$. Indeed if $G$ virtually fixes a point of $Z$ then by \cite[Lemma 26]{PS1}
$G$ is virtually of the form $H\times \Z$ . If $H$ is one-ended $\partial X$ is not $m$-thick while if $H$ has more than 1 end then $G$ virtually splits
over a 2-ended group and $\partial X$ is 2-thick.

There are continua $M,N$ such that $$M \cup N =Z= \bd X,\, M\cap N = A,\, g(M) \subsetneq M \text{ and } g^{-1}(N) \subsetneq N.$$ Let $\tilde M $ be the attracting continuum and let $\tilde N$ be the repelling continuum of this translation.

\begin{Lem}\label{L:limit} $g^+\in \tilde M, \, g^-\in \tilde N $.
\end{Lem} 
\begin{proof}

By theorem \ref{bigtits} we may assume that there is some $p\in M\cap g^7N$ such that $d_T(p,g^-)\geq \pi$. So by $\pi $-convergence $g^np\to g^+$. However clearly $g^np$ converges to a point in $\tilde M$. So $g^+\in \tilde M$. Similarly $g^-\in \tilde N$. 

\end{proof}

\begin{Lem}\label{L:Int} Let $G$ act geometrically on $X$,  a proper CAT(0) space and let $\cP$ be the cactus pretree of $\bd X$ .  If $g\in G$ translates along $A\in \cP$ and $\rad \hat N\geq \pi $,  where $\hat N$ is the negative continuum of this translation,  then $g^+ \in \hat A$. \end{Lem}

\begin{proof}
Clearly there is $n\in \hat N$ with  $d_T(n,g^-) \geq \pi$.  By $\pi$-convergence $g^i(n) \to g^+$, but
$\hat N$ is closed and $g$-invariant so $(g^i(n)) \subset \hat N$ which implies $g^+ \in \hat N$. Thus  $g^+ \in \hat N \cap \tilde M = \hat A$.
\end{proof}

\begin{Cor}\label{C:pm} If $g$ translates along $A$ in a nesting fashion  then $$\hat A \cap \{g^\pm\} \neq \emptyset .$$
\end{Cor}
\begin{proof}  Since $g$ is nesting, one of $\hat N$, $\hat M$ contains an min cut.  
The result now follows from  lemma \ref{L: bigpi}.
\end{proof}

We will need some auxilliary lemmas:

\begin{Lem}\label{L:nestcap} Suppose $g_1,g_2$ are acting nestingly and translate along the disjoint min-cuts $A_1,A_2$ and let $I_1,I_2$ be their translation intervals
containing $A_1,A_2$ respectively. If the closure of $I_1\cap I_2$ contains an end-point of $I_2$ then $G$ is rank 1.\end{Lem}

\begin{proof} For $i=1,2$, let $$\hat A_i=\bigcap _{n\in \Z} g_i^nA_i\,  (i=1,2)$$ and notice that $g_i(\hat A_i) =\hat A_i$.  By replacing $g_2$ with a power, we may assume that  $$\hat A_2=A_2\cap g_2A_2=g_2^iA_2\cap g_2^{i+1}A_2$$ for any $i\in \Z$. We note also that if $B_j$, $j\in J$ is the set of min cuts lying in
$I_1$ then $$\hat A_1=\bigcap _{i\in \Z}g_1^iA_1=\bigcap _{i\in J}B_j.$$

 As $I_1\cap I_2$ contains an end-point of $I_2$ it follows that for some $i\in \Z$, $g_2^{i}A_2,g_2^{i+1}A_2$ lie in $I_1\cap I_2$.
Therefore $\hat A_1\subseteq \hat A_2$. But $A_1\cap A_2=\emptyset $ and $\hat A_i\subseteq A_i\,  (i=1,2)$. It follows that $\hat A_1=\emptyset $ so by Corollary \ref{C:disjoint}, $G$ is rank 1.
\end{proof}

\begin{Lem}\label{L:product} Suppose $g_1,g_2$ are acting nestingly and translate along the disjoint min-cuts $A_1,A_2$ and let $I_1,I_2$ be the translation intervals of $g_1,g_2$
containing $A_1,A_2$ respectively. Assume that 
the shortest path (possibly trivial) in $T$ joining $I_1,I_2$ does not contain the endpoint of either $I_1$ or $I_2$.
Then $G$ is rank 1.
\end{Lem}

\begin{proof} The proof is very similar to that of theorem  \ref{rank1}, we refer the reader to the picture there. By lemma \ref{L:nestcap}
$I_1\cap I_2$ is either empty or it does not contain the endpoints of either $I_1$ or $I_2$.
Let $[x,y]$ be a shortest path joining $I_1,I_2$ in $T$ (so $x\in I_1$, $y\in I_2$). Pick $n,k\in \mathbb Z$ so that
$A_1\in [x,g_1^nx], A_2\in [y,g_2^ky]$. We set $a=g_1^n,b=g_2^k$. Then $$ab^{-1}([bx,by])=[ax,ay]$$ and $[bx,by],[ax,ay]$ are contained in
an interval. It follows that $ab^{-1}$ acts along an interval without fixing a point and this interval
contains both $A_1,A_2$, so by Corollary \ref{C:disjoint}.
 $G$ is rank 1.
\end{proof}

\begin{Lem}\label{L:disjoint} Suppose $g_1,g_2$ are acting nestingly and translate along the disjoint min-cuts $A_1,A_2$ and let $I_1,I_2$ be the translation intervals of $g_1,g_2$
containing $A_1,A_2$ respectively. If there is an element $g$ nesting along an axis $I$ such that $I$ intersects both $I_1,I_2$ non trivially
and contains at least one endpoint of each, then $G$ is rank 1.\end{Lem}

\begin{proof} For $i=1,2$, we may assume that $I$ contains the endpoint of $I_i$ in the direction of $g_i$ translation.    Let $$\hat A_i=\bigcap _{n\in \Z} g_i^nA_i\,  (i=1,2)$$ and notice that $g_i(\hat A_i) =\hat A_i$.  By replacing $g_i$ with a power, we may assume that  $\hat A_i=A_i\cap g_iA_i$ for $i=1,2$. Since $A_1,A_2$ are disjoint $\hat A_1\cap \hat A_2=\emptyset $.
Since the sets $g_i^j(A_i\setminus \hat A_i)$ are pairwise disjoint ($j\in \Z$) it follows that
for any $m$ there are $k,n>m$ such that $g_1^k(A_1), g_2^n(A_2)$ are disjoint. As $I$ contains infinitely many such translates of $A_1,A_2$ the result follows by
Corollary \ref{C:disjoint}.
\end{proof}
%
%

\begin{Lem}\label{L:many} Let $X$ be a CAT(0) space and let $G$ be a one ended group acting geometrically on $X$ without virtually fixing a point of $\bd X$.
Assume $\bd X$ is $m$-thick and that $G$ acts nestingly on the cactus tree $T$ of $\bd X$. Then either $G$ is rank 1 or for any $n\in \mathbb N$ 
there are  $g_1,...,g_n\in G$ and pairwise disjoint min-cuts $A_1,...,A_n$ such that $g_i$ translates along $A_i$ nestingly $(i=1,...,n)$.
Moreover the intervals of translation of $g_i$ containing the $A_i$'s are mutually disjoint.
\end{Lem}
\begin{proof} Let $g_1\in G$ be an element which nests with nesting cut $A_1$. So there is an interval $[A_1,B_1] \subset T$ with 
$$g([A_1,B_1]) \subsetneq [A_1,B_1].$$
 We may assume that $A_1$ is an inseparable cut and that $I_1=(C_1,B_1)$ is an interval of translation of $g_1$ with $A_1\in I_1$. If for some power $g_1^n$ we have
$g_1^nA_1\cap A_1=\emptyset $ then by theorem \ref{T:miss} $g_1$ is rank 1 and the lemma is proven.

Otherwise $\bigcap _{n\in \mathbb Z}g^nA_1\ne \emptyset $. By lemma \ref{permutation} there is some $h\in G$ such that $hA_1\cap A_1=\emptyset $.
Note now that if an inseprable cut $D$ lies in $I_1$ then $D\cap A_1\ne \emptyset $. It follows that $hA_1$ does not lie in $I_1$. Clearly $I_2=h(I_1)$ is an interval of translation of
$g_2=hg_1h^{-1}$. Since we assume that $G$ is not rank 1 by lemmas \ref{L:nestcap} \ref{L:product} it follows that $I_1\cap h(I_1)=\emptyset $.

 Assume we have produced nesting elements $g_1,...,g_k$ as in the lemma translating along the pairwise disjoint
min-cuts $A_1,....,A_k$. By lemma \ref{permutation} there is some $u\in G$ such that $$u(\cup _{i=1}^k A_i)\cap (\cup _{i=1}^k A_i)=\emptyset .$$ So $g_{k+1}=ug_1u^{-1}$
nests along $u(A_1)$ which is disjoint from all $A_2,...,A_k$. Again by lemmas \ref{L:nestcap} \ref{L:product} it follows that $I_{i}\cap u(I_1)=\emptyset $ for all $i=1,...,k$.
\end{proof}

If $B,C$ are subsets of the tree $T$ we denote by $Hull(B,C)$ the smallest convex subset containing $B,C$. We say that two intervals $I,J$ are co-linear if
$Hull(I,J)$ is an interval.

\begin{Lem}\label{L:notline} Let $X$ be a CAT(0) space and let $G$ be a one ended group acting geometrically on $X$ without fixing a point of $\bd X$.
Assume $\bd X$ is $m$-thick and that $G$ acts nestingly on the cactus tree $T$ of $\bd X$. Then either $G$ is rank 1 or 
there are $g_1,g_2$ which  translate along  disjoint min-cuts $A_1,A_2$  nestingly so the that translation intervals $J_1,J_2$ are not colinear.
\end{Lem}
\begin{proof}
Let $N=m+1$. By lemma \ref{L:many} 
there are nesting elements $g_1,...,g_N$ and pairwise disjoint min-cuts $A_1,...,A_N$ such that $g_i$ translates along $A_i$ nestingly with translation interval $I_i$ and $I_i\cap I_j=\emptyset $, for $i\ne j,\, i,j =1,...,N$.
If there is an $i$ such that the shortest path joining $I_1$ to $I_i$ ($i>1$) 
does not intersect the end points of either $I_1$ or $I_i$ then by lemma \ref{L:product} $G$ is rank 1.  We assume now that $Hull(I_1,I_i)$ is an interval
for all $i>1$.

By theorem \ref{bigtits} there is a continuum $C_i$  between $A_i$ and $g_i^7A_i$ for $i=2,...,N$ such that $\rad C_i\geq \pi$. Let $x$ be a point
between $A_1, g_1A_1$ that does not lie on $A_1$ or $g_1A_1$. Let $(h_j) \subset G$ such that $h_j(p)\to x$, and $ h_j^{-1}(p)\to y$ for any $p \in X$. Since $\rad C_i\geq \pi$ there are points
$n_i\in C_i$ with $d_T(n_i,y)\geq \pi$.  Note that for any interval $J$ disjoint from $I_1$ if $Hull (I_1,J)$ is an interval then $A_1\cup  g_1A_1$ separates $x$ from $J$.

Therefore since $h_jn_i\to x$ there is a $j$ big enough such that for any $i$, $h_jn_i$ is between $A_1$ and $g_1A_1$, so either $$h_jI_i\cap I_1\ne \emptyset $$ or $Hull (I_1,h_jI_i)$ is not an interval. Since $A_2,...,A_N$ are all disjoint, all $h_jA_i$ are disjoint $(i=2,...,N)$.
Since $N=m+1$, for some
$k$, $h_jA_k$ will be disjoint from $A_1$. Clearly $h_jg_kh_j^{-1}$ acts nestingly along the interval $h_j(I_k)$. If
$$h_j(I_k)\cap I_1\ne \emptyset $$
then by lemmas \ref{L:nestcap}, \ref{L:product}, $G$ is rank 1.
We are left with the case where  $I_1,h_jI_k$  are disjoint and not colinear.   Clearly $h_jg_kh_j^{-1}$ acts nestingly  
along $h_jA_k$ which proves the lemma (setting $J_1=I_1$ and $J_2=h_jI_k$).

\end{proof}

\begin{Lem}\label{elliptic} Suppose $g_1,g_2$  translate along  disjoint min-cuts $A_1,A_2$  nestingly and let $I_1,I_2$ be the translation intervals of $g_1,g_2$
containing $A_1,A_2$ respectively.
Assume that the shortest path $[a,b]$ joining $\bar I_1$ to $\bar I_2$ contains an endpoint $a$ of $I_1$ and  that $I_1$ and $I_2$ are not colinear.

Then there is some $x\in [a,b]$ fixed by $g_1$ and so that if $I_2'=Hull(x,I_2)$ then
 $g_1^nI_2'\cap I_2'=\{x\}$ for all $n\ne 0$.
\end{Lem}

\begin{proof} By lemmas \ref{L:nestcap}, \ref{L:product}, since $A_1,A_2$ are disjoint,  either $G$ is rank 1 or
  $I_1,I_2$ are disjoint. Without loss of generality we may assume that $g_1$ translates towards $I_2$.  Let $I_2 =(c,d)$.
By lemma \ref{L:Int},  $g_1^-\in \hat A_1$. 
By corollary \ref{C:pm}, we may assume that $g_2^+\in \hat A_2$.
By theorem \ref{bigtits} for any $k$, between any two translates  of $A_2$ of the form  $g_2^kA_2,g_2^{k+7}A_2$ there is some $q$ with $d_T(q,g^+)\geq \pi $ so  if $n\to -\infty $, $g_1^nq\to g_1^-$. Since $g_1^-\in A_1$ 
and $A_1,A_2$ are disjoint it follows that if the lemma is not true then either $(c,b)$ or $(b,d)$ is contained in a nesting interval $I_3$ of $g_1$. 

Clearly we may assume that for some $k\in \Z$, $g_2^kA_2, g_2^{k+1}A_2$ lie in $I_3$ and $$g_2^kA_2\cap g_2^{k+1}A_2=\hat A_2 .$$
It follows that if $$S=\bigcap_{n\in \Z} g_1^n(g_2^kA_3) $$ then $S\subseteq \hat A_2.$ As before by corollary \ref{C:pm}, $\{g_1^{\pm} \}$ intersects $S$.

Since $\hat A_2$ is fixed by $g_2$, $g_2$ fixes one of $g_1^{\pm}$.

 By \cite[Theorem 8]{SWE} $g_1,g_2$ virtually commute. So by passing to a power if necessary we may assume that $g_1,g_2$ commute. However
this contradicts lemma \ref{L:commute}.

 Since $g_1$ can neither fix nor translate along $I_2$ it acts elliptically so
there is some $x\in [a,b]$ fixed by $g_1$ and so that
$g_1^nI_2'\cap I_2'=\{x\}$ for all $n\ne 0$.
\end{proof}

\begin{Def}
Suppose $g_1,...,g_n$  translate along  disjoint min-cuts $A_1,...,A_n$  nestingly so that  be the translation intervals $I_1,...,I_n$ of  $g_1,...,g_n$ 
containing $A_1,...,A_n$ are mutually disjoint.
  We say that $g_1,...,g_n$ form a \textit{nesting chain of length $n$} if the shortest path
joining $I_k$ to $I_{k-1}$ $(k=2,...,n)$ contains an endpoint of $I_k$ and does not contain an endpoint of $I_{k-1}$. \end{Def} 

\begin{figure}[htbp]
   \includegraphics[scale=1.00]{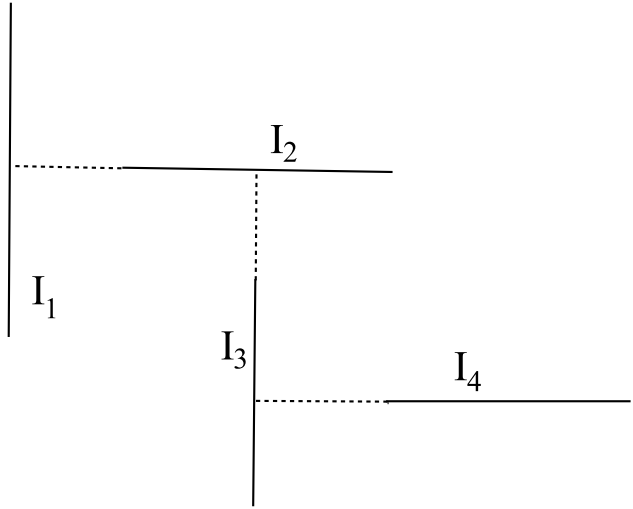}%

  \caption{A nesting chain of length 4}
  \label{nesting chain.eps}
\end{figure}

\begin{Lem}\label{L:chain} Let $X$ be a CAT(0) space and let $G$ be a one ended group acting geometrically on $X$.
Assume $\bd X$ is $m$-thick and that $G$ acts nestingly on the cactus tree $T$ of $\bd X$. Then either $G$ is rank 1 or for any $n\in \N$ there is a chain
of nesting elements of length $n$ on $T$. \end{Lem}
\begin{proof} The proof is almost the same as for lemma \ref{L:notline}. Pick $N\gg n$. 

By lemma \ref{L:many} 
there are elements $g_1,...,g_N\in G$ and pairwise disjoint min-cuts $A_1,...,A_N$ such that $g_i$ translates nestingly along $A_i$ $(i=1,...,N)$.
Let $I_i$ be an interval of translation of $g_i$ containing $A_i$. If for some $i\neq j$, $I_i\cap I_j\ne \emptyset $ then by lemmas
\ref{L:nestcap}, \ref{L:product} $G$ is rank 1. Similarly if the shortest path joining $I_i$ to $I_j$
does not intersect the end points of either $I_i$ or $I_j$ then by lemma \ref{L:product} $G$ is rank 1.

 Without loss of generality we may assume
that $g_1,...,g_k$ is a maximal nesting chain for some $n>k\geq 1$. We show that there is a nesting chain of length $k+1$. 

If there is an $i$ such that the shortest path joining $I_k$ to $I_i$ ($i>k$) which
does not intersect the end points of either $I_k$ or $I_i$ then by lemma \ref{L:product} $G$ is rank 1.  We assume now that the shortest path joining $I_k,I_i$ contains an endpoint of $I_k$
for all $i>k$.

By theorem \ref{bigtits} there is a continuum $C_i$  between $A_i,g_i^7A_i$ for $i=k+1,...,N$ such that $\rad C_i\geq \pi$. Let $x$ be a point
between $A_k, g_kA_k$ that does not lie on $A_k$ or $g_kA_k$. Let $h_j\in G$ such that $h_j(p)\to x$, $ h_j^{-1}(p)\to y $ for all $p \in X$. Since  $\rad C_i\geq \pi$ there are points
$n_i\in C_i$ with $d_T(n_i,y)\geq \pi$.  Note that for any interval $J$ disjoint from $I_k$, if the shortest path joining $I_k$ and $J$ contains an endpoint of $I_k$ then $A_k\cup  g_kA_k$ separate $x$ from $J$.

Since $h_jC_i\to x$, for $j \gg0$ either $h_jI_i$ intersects $I_k$  or  $I_k,h_jI_i$ are not colinear. Since $A_1,...,A_N$ are all disjoint all $$h_jA_i, \, i=1,...,N$$ are disjoint. Therefore for some
$r$, $h_jA_r$ will be disjoint from $A_1,A_2,...,A_k$. Since $G$ is not rank 1 It follows that $h_jI_r$ does not intersect $I_k$ . Since the arc from $h_j\bar I_r$  to $\bar I_k$ intersects $I_k$ in a non-endpoint and $h_jg_rh_j^{-1}$ acts 
on $h_jA_r$ nestingly with minimal overlap, this proves the lemma (setting $I_{k+1}$ to be $h_jI_r$).

%
\end{proof}

\begin{Thm} \label{rank2} Let $X$ be a CAT(0) space and let $G$ be a one ended group acting geometrically on $X$.
Assume $\bd X$ is $m$-thick and that $G$ acts nestingly on the cactus tree $T$ of $\bd X$. Then $G$ is rank 1.
\end{Thm}
\begin{proof} 

\begin{figure}[htbp]
   \includegraphics[scale=1.00]{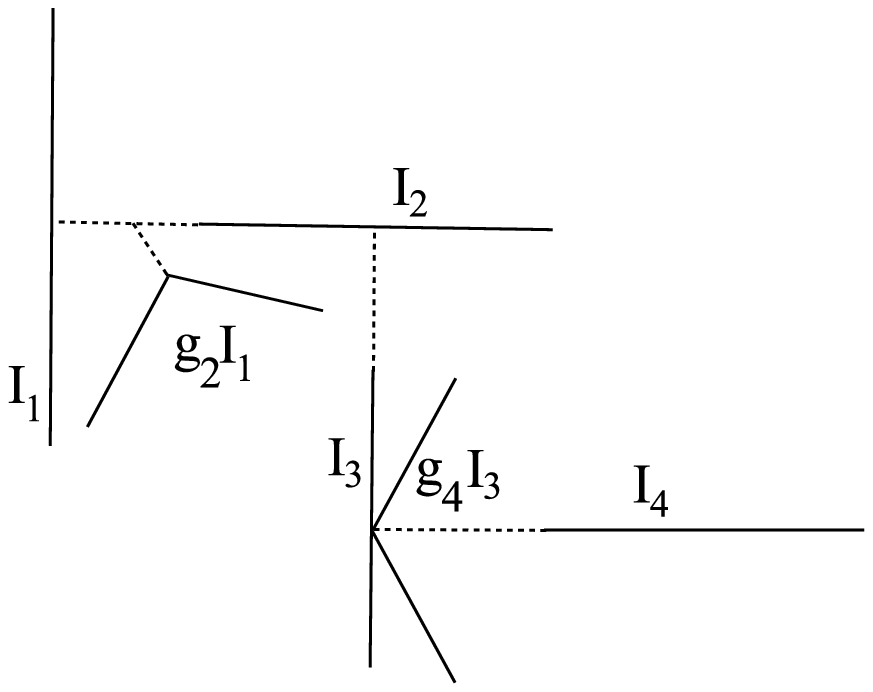}%

  \caption{}
  \label{chain-rank1.eps}
\end{figure}

If $G$ is not rank 1 then by lemma \ref{L:chain} there are nesting elements $g_1,g_2,g_3,g_4$ translating along pairwise disjoint min-cuts $A_1,...,A_4$ with intervals
of translation $I_1,...,I_{4}$. 
By lemma \ref{elliptic} $g_4(I_3)$ and $g_2(I_1)$ intersect the intervals $I_3,I_1$ at at most one point. If $a=g_2g_1g_2^{-1}$ and
$b=g_4g_3g_4^{-1}$ then $a,b$ act nestingly on the intervals $g_2(I_1), g_4(I_3)$. Moreover the shortest path $p$ joining $g_2(I_1), g_4(I_3)$
intersects non trivially $I_2,I_3$ and contains one endpoint of each one of them. Since the axis of the element $ab$ contains $p$
it follows by lemma \ref{L:disjoint} that $G$ is rank 1.

\end{proof}

We show now that intervals of translation do not overlap:

\begin{Lem} \label{intersection-cat}
Suppose that $g$ and $h$ are nesting elements with $I\subset T$ an interval of translation of $g$ and $J\subset T$ an interval of translation of $h$ with $\{g^\pm \}\neq \{h^\pm\}$. If $a \in J$ is an endpoint of $I$, then $J \cap I =\emptyset$.
\end{Lem}

\begin{proof}
Let's say that $I=(b,a)$.
Replacing $g$ with $g^2$ if need be  we may assume that $g(a) = a$. If $I\cap J\ne \emptyset$ then a min-cut $A$ lies in $I\cap J$. Notice that $$\hu\{g^n(A):\, n \in \Z\}$$ is the interval of translation $I$ of $g$ (ie $g$ translates along $A$). Since $a\in J$ by replacing $g$ by $g^{-1}$ if necessary, we may assume that $g^nA\in I$ for all $n\geq 0$.

Then $g^nA$ is also a nesting cut of $h\in G$ for all $n\geq 0$.
Since $h$ is nesting, by corollary \ref{C:pm}, we may assume that at least one of $h^+,h^-$ lies in $g^nA$ for all $n$. Say $h^+ \in g^n A$ for all $n$.
Then $h^+$ is fixed by $g$. It follows  by \cite[Theorem 8]{SWE} that $h^s,g$ commute for some $s\neq 0$. Replacing $h$ by $h^s$ we have that $g,h$ commute and an endpoint of $I$ is contained in $J$. By lemma \ref{L:commute} for all $m,k\neq 0$,  $g^mh^k$ are nesting homeomorphisms that translate along $A$. It follows that $A$ contains at least one of the two limit points of $g^mh^k$. However this is impossible as these  are distinct  for $m,k$ relatively prime, and $A$ is finite.

\end{proof}

Let $Z$ be an $m$-thick continuum. We explained in section 2 how to associate to $Z$ a `cactus tree' (see \cite{PS2}). To do this
one starts with a pretree $\mathcal R$ consisting of `isolated' min cuts and `wheels' consisting of crossing min cuts. It turns out that this pretree
is `preseparable' and has `few gaps' and hence embeds to an $\mathbb R$-tree $T$. We note that this $\mathbb R$-tree might have terminal points.
This is the case if and only if the pretree $\mathcal R$ has terminal points. Note that by definition a wheel can not be terminal so the
possible terminal points are isolated min cuts. Note that if a group $G$ acts on $T$ by isometries terminal points of $T$ are mapped to terminal
points of $T$, in particular the action of $G$ on $T$ is not minimal.

We assume now in the rest of the section that $Z=\partial X$ where $X$ is CAT(0) and $G$ acts on $X$ geometrically.
We showed in the previous section that if $G$ acts nestingly on $T$ then $G$ is rank 1. In particular by $\pi $-convergence it follows that
$\mathcal R$ has no terminal points.
%

\begin{Lem} \label{end-interval}
Suppose that $g$ is a nesting element and that $I,J$ are intervals of translation of $g$. Then $I,J$ lie in distinct cross components.
\end{Lem} 
\begin{proof} By lemma \ref{intersection-cat} if $I_1=I,I_2,...,I_n=J$ is a sequence of pair-wise intersecting intervals the end-points of $I$ do not lie in $I_2$.
It follows that there is a path joining $I$ to $J$ which does not contain any-endpoint of $I$ contradicting Lemma \ref{L:int}(4).
\end{proof}

\begin{Lem} \label{endstab} Let $S$ be a cross-component of $T$ and let $E$ be an end of $S$ which is not an end of $T$.  It follows that $E$ corresponds to a nonterminal point of $T$.   Then $\st(E)< \st(S)$ is virtually cyclic.
\end{Lem}

\begin{proof}We may assume that $\st(E)$ is infinite, and not virtually $\Z$.

\noindent
{\bf Case I:}  For each $h \in \st(E)$ there is a ray in $T$ representing $E$ which is fixed by $h$.
Let $H< \st(E)$ finitely generated and infinite (there is a uniform bound on the size of a finite subgroup).  Since $H$ is finitely generated, it follows that there is a ray $R \subset T$ with $R$ fixed by $H$. By the definition of cross-components $R$ contains infinitely many linearly ordered min-cuts.
It follows by theorem \ref{bigtits} that there are linearly ordered min-cuts $A_1,A_2,A_3$ in $R$ and continua $c_1$ between $A_1,A_2$ and $c_2$ between $A_2,A_3$ such that $\rad c_i\geq \pi,\, i=1,2$. Let $h_n\in H$ such that $h_n\to h^+,h_n^{-1}\to h^-$.  There are $y_1\in c_1,y_2\in c_2 $ such that
$d_T(y_i,h^-)\geq \pi $. Therefore by $\pi $ convergence $h_ny_i\to h^+$. However $H$ fixes $R$ so $h_ny_1$ converges to a point between $A_1,A_2$
and $h_ny_2$ converges to a point between $A_2,A_3$. It follows that $h^+\in A_2$.


Therefore  $\Lambda H \subseteq A_2 $ so $\Lambda H$ is finite.  By \cite[Lemma 16]{PS1} we have that $H$ is virtually cyclic. It follows by \cite[Theorem 7.5, sec. II]{BRI-HAE} that $ \st(E)$ is virtually cyclic.


\noindent
{\bf Case II:} There is some $g \in \st(E)$ and an $R\subset T$ a ray representing $E$ with 
$g(R) \subsetneq R$.  Clearly $g$ has infinite order.  By the definition of the Cactus tree $T$ there is a minimal cut $A$ and continua $M,N$ such that $$M \cup N = \bd X,\, M\cap N = A,\, g(M) \subsetneq M \text{ and } g^{-1}(N) \subsetneq N.$$ 

Let $\hat M = \bigcap\limits_{i\in \Z} g^i(M)$, $\hat N = \bigcap\limits_{i\in \Z}  g^i(N)$, and  $\hat A = \bigcap\limits_{i \in \Z} g^{i}(A) = \hat M \cap \hat N$. 
By lemma \ref{end-interval}
we may assume that $E\in T$ is a limit point of the interval $I$
of $T$ spanned by $g^i(A), \, i\in \Z$. It follows that $\st(E)\subseteq \st(\hat M)$ and that $\st(E)$ stabilizes $\hat A$.  By lemma \ref {L:limit} $g^+\in \hat M$ and $g^-\in \hat N$.
 
\noindent
 {\bf Subcase i :}   $g$ is rank 1. By $\pi$-convergence either  $\hat N = \{g^-\}$ and $\hat M =\{g^+\}$ or $\{g^\pm\} \cap  \hat A\neq \emptyset$.  Either way 
 $\st(E)$ virtually fixes $g^\pm$.  It follows that  $\st(E)<_{vi}\st(g^+)$. By \cite[Theorem 8]{SWE} $$\st(g^+)=\bigcup _{n>0}Z _{g^n}.$$ Since
 $g$ is rank 1 $Z _{g^n}$ is virtually $\langle g\rangle$. It follows by \cite[Theorem 7.5, sec. II]{BRI-HAE} that $ \st(E)$ is virtually cyclic.

 \noindent
 {\bf Subcase ii :}  $g$ is not rank 1.  Thus $d_T(g^+,g^-) = \pi$. It follows by theorem \ref{T:miss} that $\hat A $ is non-empty.  By replacing $g$ by a power if necessary we may assume that $\hat A $ is fixed by $g$, so by \cite[Theorems 3.2, 3.3]{RUA} $\hat A \subset \Lambda Z_g$. 
 
Since $E$ is not an end of $T$, $\hat M$ has non-empty interior, so by lemma \ref{L:Int},  $g^-\in  A$.

We note that $ \st(E)$ stabilizes $\hat A\ni g^-$  and $\hat A$ is finite. It follows that a finite index subgroup of $\st(E)$
fixes $g^-$. As we are trying to determine $\st (E)$ up to finite index to simplify notation we assume $\st (E)$ fixes $g^-$.
For any $h\in \st(E)$ since $h$ fixes $g^-$ by \cite[Theorem 8]{SWE} there is an $n>0$ such that $g^nh=hg^n$ .


We note that by \cite[Theorem 7.5, sec. II]{BRI-HAE} if every finitely generated subgroup of $\st (E)$ is virtually cyclic then $\st(E)$ is virtually cyclic. Let $H=\langle g_1,...,g_k,g\rangle $ be a finitely generated subgroup of $\st(E)$ that is not virtually cyclic. Clearly for some
$n>0$, $g^n$ lies in the center of $H$.

\textbf{Claim}. There is some $h\in H$ such that $\langle h,g^n\rangle $ is isomorphic to $\Z ^2$.

\begin{proof}

Let $N$ be the (normal) subgroup of $H$ fixing $I$ (the interval of translation of $g$). By the same argument that we used in case I, $N$ is virtually cyclic. If $N$ is infinite
then there is some $h\in N$ of infinite order so $\langle h,g^n\rangle $ is isomorphic to $\Z ^2$.
Otherwise since $H$ is
not virtually cyclic the quotient group $H/\langle\!N,g^n\!\rangle $ is infinite. We note that $H$ acts non-nestingly on $S$. This follows by lemma \ref{intersection-cat} and by
the definition of $S$. So by \cite{Le} we can promote this action to an action by isometries on an $\mathbb R$-tree $S'$ where $g^n$ acts by translations on $S'$. Since $g^n$ lies in the center of $H$ every element of $H$ acting by isometries acts along the axis of $g^n$. Using this we get a homomorphism from $H/N\rangle $ to the group of isometries of $\R$. Since $H/\langle\!N,g^n\!\rangle $ is infinite
it follows that $H/N $ is not virtually $\Z$, so
there is some $h\in H$ so that $\langle h,g^n\rangle $ is isomorphic to $\Z ^2$. \end{proof}

We note that $h$ acts on $S$ and fixes $E$. From the proof of the claim above either $h$ translates along an interval $J$
that has $E$ as an endpoint or $h$ fixes $I$. In the first case note that by lemma \ref{intersection-cat}  $I,J$ do not contain the endpoints of each other.
Therefore if $I\ne J$ we may assume that there is $x\in J$ such that $x\notin I$. 
Note now that $gh(x)\neq hg(x)$, contradiction. So $I=J$. It follows that both $h,g$ stabilize $\hat M, \hat N$ (this is clearly true if $h$ fixes $I$). It follows
that every element of $\langle h,g^n\rangle $ stabilizes $\hat M $ and $\hat N$. It follows by lemma \ref{L:Int} that for any $a\in \langle h,g^n\rangle $, $a^-\in A$. This is a contradiction
since $A$ is finite.

\end{proof}

\begin{Thm} \label{thick}Let $Z$ be be the boundary of a 1-ended CAT(0) space $X$ on which $G$ acts geometrically.
Suppose that $Z$ be an $m$-thick continuum with cactus tree $T$ and
 that $G$ acts nestingly on $T$. Then $G$ acts non-nestingly without any fixed point on an $\mathbb R$-tree $\bar T$ with virtually cyclic
edge stabilizers.
\end{Thm} 
 
\begin{proof} Let $T$ be the cactus tree of $Z$. We claim that $T$ has no endpoints.

Assume that $T$ has some endpoint $x$ which is necessarily an min cut of $Z$. Then there are disjoint open sets $U,V$ such that $Z\setminus x=U\cup V$.
Since $G$ is rank-1 the Tits-diameter of both $U,V$ is infinite. Therefore by $\pi$-convergence there is some $g\in G$
such that $gx\subset U, g^{-1}x\subset V$. It follows that $x$ is not an endpoint of $T$.

Note that the same argument shows  that is there is no proper $G$-invariant subtree of $T$.

By lemma \ref{intersection-cat} $G$ acts on $T$ with non overlapping translation intervals. By lemma \ref{endstab} the stabilizers of ends of cross components of $T$
are virtually cyclic, and so by  theorem \ref{T:nesting},   $G$
 acts non-nestingly without proper invariant subtree on an $\R$-tree $\bar T$ with virtually cyclic arc stabilizers.

%


%
 

\end{proof}

\section{Non-nesting actions on $\R$-trees}  In all cases we arrive at a non-nesting action of $G$ on an $\R$-tree with virtually cyclic arc stabilizers.

Even though Bestvina-Feighn-Rips theory \cite{B-F} does not imply immediately that $G$ splits over a 2-ended group we can reach this conclusion using
a further refinement of the theory due to Sela \cite{Se} and Guirardel \cite{Gui}. Guirardel gives a quite precise description of the cases when $G$ does not split over a 2-ended group.
Using CAT(0) geometry and our hypothesis on the boundary we can rule out these cases.

We recall here Guirardel's theorem \cite{Gui}:

\begin{Guirardel} Consider a non-trivial minimal action of a finitely generated
group $G$ on an $\mathbb R$-tree $T$ by isometries. Assume that:

1. $T$ satisfies the ascending chain condition;

2. for any unstable arc $J$,

(a) $G(J)$ is finitely generated;

(b) $G(J)$ is not a proper subgroup of any conjugate of itself.

Then either $G$ splits over the stabilizer of an unstable arc or over the stabilizer of an
infinite tripod, or $T$ has a decomposition into a graph of actions where each vertex action
is either

1. simplicial: $ G_v \curvearrowright Y_v$
 is a simplicial action on a simplicial tree;
 
2. of Seifert type: the vertex action $ G_v \curvearrowright Y_v$ has kernel
$N_v$, and the faithful action $G_v/N_v \curvearrowright Y_v$ is dual to an arational measured foliation on a closed 2-orbifold with
boundary;

3. axial: $Y_v$ is a line, and the image of $G_v$ in Isom($Y_v$) is a finitely generated group
acting with dense orbits on $Y_v$.
\end{Guirardel}

We give now an informal description of a graph of actions. Let $G$ be a group acting on a simplicial tree $T_s$ and let $\Gamma $ be the
corresponding graph of groups decomposition of $G$. We colour the edges of $T_s$ red.
Let $v$ be a vertex of $\Gamma$ and let $G_v$ be its vertex group. Assume each $G_v$ acts on an $\R $-tree $Y_v$. We assume further that if $G_e$ is an edge
group of an edge $e$ incident to $v$ then $G_e$ fixes a point $s_{ev}$ of $Y_v$.
We colour the $Y_v$'s black. Now we replace
each vertex in the orbit $G\cdot v$ of $T_s$ by a copy of $Y_v$. If $v,w$ are adjacent vertices of $T_s$ joined by an edge $e$ then we join $Y_v, Y_w$ by a red edge with endpoints $s_{ev}$, $s_{ew}$.
In this way we obtain an $\R$-tree $T'$. Finally we collapse all red edges of $T'$ and we obtain an $\R$-tree $T$. Clearly $G$ acts on $T$ by isometries. We say then
that the action of $G$ on $T$ decomposes as a graph of actions (with underlying graph $\Gamma $).

For a formal and more detailed description of graphs of actions see \cite{Gui}.


In our context we will call `exceptional' the groups that admit a decomposition
as in Guirardel's theorem with edge groups virtually $\mathbb Z^2$. With some more work we will be able to show that, under our assumptions, these exceptional groups
split over 2-ended groups. We define now this class of groups formally:

\begin{Def} Let $G$ be a one ended group acting geometrically on a CAT(0) space $X$. 
We say that $G$ is {\it exceptional} if $G$ admits a decomposition as a non-trivial graph of groups which has the following properties:

There are two types of vertex groups:

Black vertex groups $G_v$ where for some 2-ended $N_v$, $G_v/N_v$ is the fundamental group of a closed 2-orbifold with
boundary;

White vertex groups $G_v$ where $G_v$ could be any CAT(0) group.

Edge groups are virtually $\mathbb Z ^2$ and at least one endpoint of every edge is a black vertex.

Moreover if $G_e$ is an edge group incident to a black vertex group $G_v$ then $G_e/N_v$ is a subgroup of $G_v/N_v$ corresponding
to a boundary component of the underlying 2-orbifold.

We call this decomposition of $G$ a {\it Guirardel decomposition  }.

We call the black vertex groups of this decomposition {\it groups of Seifert type}.

If $G_v$ is a group of Seifert type, $q:G_v\to G_v/N_v$ is the quotient map to the orbifold group and $<b>$ is the fundamental group of
a boundary component of the underlying orbifold we say that $q^{-1}(<b>)$  is a {\it peripheral subgroup} of $G_v$.
\end{Def}

In our context we have the following:

\begin{Thm}\label{T:Rtree} Let $G$ be a one ended group acting geometrically on a CAT(0) space $X$. 
If $G$ acts non-nestingly on a non-trivial $\R$-tree $T$ without fixing a point and with virtually cyclic arc stabilizers, 
then either $G$ splits over a 2-ended group or $G$ has a 2-ended normal subgroup $N$ such that $G/N$ is the fundamental group
of a closed 2-orbifold, or $G$  is exceptional.
\end{Thm}

\begin{proof}
Applying a result of Levitt \cite{Le} to the non-nesting $G$ action on the tree $T$ we obtain an $\mathbb R$-tree $ T _0$ endowed with a $G$ action by isometries such that the stabilizer of an arc in $ T _0$ will also be the stabilizer of an arc of $T$.
Clearly the action satisfies the ascending chain condition
since 2-ended subgroups of CAT(0) groups are contained in maximal 2-ended subgroups. Also in a CAT(0) group a 2-ended group is not a proper subgroup
of a conjugate of itself. So  either $G$ splits over a 2-ended group or the action can be decomposed as a graph of actions as in the theorem.
A decomposition of the action as a graph of actions induces a decomposition of $G$ as a graph of groups. The vertex groups of this graph of groups
correcpond to the groups $G_v$ of Guirardel's theorem and we have actions $ G_v \curvearrowright Y_v$ where $Y_v$ is a subtree of $S$.
An edge group of en edge joining $v,w$ is the stabilizer of the point of intersection $Y_v\cap Y_w$ (if this intersection
is non-empty).  We note that in the case of vertex actions (case 1 in Guirardel's theorem) $Y_v$ might be reduced
to a point. 

Therefore either $G$ splits over a 2-ended group or this action can be obtained as a graph of actions by gluing together
simpler `building blocks', i.e. actions on $\mathbb R$-trees $Y_v$. The building blocks are glued together along a point and the combinatorics of the gluing are described by a simplicial tree.
If a building block $Y_v$ is of simplicial type and is not a point then the edge stabilizers of $ G_v \curvearrowright Y_v$ are 2-ended so $G$ splits over a 2-ended group.
It follows that
either $G$ splits over a 2-ended group
or $G$ is decomposed as a graph of groups where all vertex groups $G_v$ are of the following types: 
\begin{enumerate}

\item  $G_v$ is of simplicial type, $ G_v \curvearrowright Y_v$
is a trivial action and $Y_v$ is a single point.
 
\item $G_v$ is of Seifert type: $G_v$ acts on an $\mathbb R$-tree $Y_v$ with kernel of the action a 2-ended group $N_v$, and the
faithful action of $G_v/N_v$ on $Y_v$ is dual to an arrational measured foliation
on a closed 2-orbifold with boundary.

\item $G_v$ is of axial type: $G_v$ acts on a tree $Y_v$ which is a line, and the image of $G_v$ in $Isom(Y_v)$ is a finitely
generated group acting with dense orbits on $Y_v$.
 
\end{enumerate}

We denote by $T_B$ the Bass-Serre tree of this graph of groups decomposition of $G$.

We note now that if there are vertex groups of axial type then $G$ splits over a 2-ended group. Indeed note that the edge groups
incident to such vertex groups are stabilizers of the axial component. Since the edge stabilizers of the original action are 2-ended
the stabilizer of the axial component is 2-ended. Hence the edge stabilizer is 2-ended,
so if there are vertex groups of axial type $G$ splits over a 2-ended group. It follows that we may assume that if there is such a vertex group then the
graph of groups reduces to a single vertex so $G$ is virtually $\mathbb Z ^2$. Since we assume that it acts on a non-trivial tree it splits over a 2-ended group.

Therefore we may assume that all vertex groups are of Seifert type and the edge groups are extensions of $\Z$ (the boundary groups) by
a 2-ended group. Since $G$ is CAT(0), by Thm. 7.1, ch. II.7 of \cite{BRI-HAE}, the edge groups are virtually $\Z ^2$. If there is a single vertex group $G_v$
then $G=G_v$ is of Seifert type. But such a group either splits over a 2-ended group or $G_v/N_v$ is the fundamental group of a closed
2-orbifold (with no boundary).

Finally to deduce that vertex groups are CAT(0) we use a recent result of 
Hruska-Ruane \cite{HR} stating that if a CAT(0) group $G$ decomposes
as graph of groups with convex edge groups then the vertex groups are convex as well.
Since edge groups are virtually $\mathbb Z^2$ and $\Z^2$ is a convex subgroup by corollary \ref{C:convex} each edge group $G_e$ acts co-compactly on a convex subset of 
$X$, $X_e$. By Proposition 7.6 of \cite{HR},
for each vertex $v$ of $T$, $G_v$ acts co-compactly on a convex subset of $X$, $X_v$, so $G_v$ is CAT(0).

\end{proof}


\begin{Lem}\label{suspension} Let  $Z$ be a suspension of a Cantor set.
Then every finite cut of $Z$ contains the pair of its suspension points and the pair
of suspension points is the unique minimal cut of $Z$.
\end{Lem}

\begin{proof} Let $a,b$ be the suspension points of $Z$ and $M,N$ be continua with $M\cup N=Z$
and $M\cap N=F$ where $F$ is finite. We claim that if $\alpha $ is a suspension arc and there is some $x\in M\setminus F$
then $\alpha \subset M$.

Indeed note that if $x\in M\setminus F$ then an open neighborhood of $x$
lies in $A$. So $M$ intersects infinitely many
suspension arcs limiting to $\alpha $. Since $F$ is finite it follows that $M$ actually contains almost all of these arcs.
As these arcs limit to $\alpha $ it follows that $\alpha  $ is contained in $M$ as well.

The same holds for $N$. So $M,N$ are both a union of suspension arcs together with $F$. If $x\in  M$ lies in the interior of a suspension arc  $\alpha $
which is not contained in $M$ then $Z=(M\setminus \{x\})\cup N$ and $M\cap N=F\setminus \{x\}$. So we can eliminate
successively elements of $F$ that are not the suspension points. Since $F$ separates we conclude that $F$ necessarily contains the suspension points.

\end{proof}


We show now that exceptional groups do not arise from an action of $G$ on its cactus tree.

\begin{Thm} \label{endgame}Let $G$ be a one ended group acting geometrically on a CAT(0) space $X$. Assume that $\partial X$ has a finite cut.
Then either $G$ is virtually a surface group or $G$ splits over a 2-ended group.
\end{Thm} 

\begin{proof}  Let $T$ be the cactus tree of $\partial X$. We note that if $\partial X$ is a circle and $G$ is hyperbolic then by \cite {GAB}, \cite{CAS-JUN} $G$ is virtually a surface group. If $G$ is not hyperbolic then by \cite{BRI} $X$ contains a flat plane. If $\partial X$ is a circle then $X$ is contained in a finite neighborhood of this plane so $G$ is virtually $\mathbb Z^2$. So we may assume that $\partial X$ is not a circle. It follows then by theorem \ref{rank1} in the case that the action of $G$ on $T$ is non-nesting and by theorem \ref{rank2} in the
case that this action is nesting that $G$ acts as a rank 1 group on $X$. Further by theorems \ref{nonnest} and \ref{thick} we have that either $G$ splits over a 2-ended group or we may obtain from the action of $G$ on $T$
a non-trivial non-nesting action of $G$ on an $\R $-tree $\bar T$ with 2-ended arc stabilizers.  We note that $\bar T$ is obtained from $T$ by collapsing some subtrees of
$T$ to points. By a theorem of Levitt \cite{Le} we can replace this action by an action by isometries on an $\R$-tree $\bar T _0$ so that arc stabilizers of $\bar T_0$ are arc stabilizers of $\bar T$.

By theorem \ref{T:Rtree} either $G$ splits over a 2-ended group or $G$ is exceptional. In particular $G$ admits a graph of groups
decomposition $\Gamma $ with at least one vertex group $G_v$ which is of Seifert type, so it has a 2-ended normal subgroup $N_v$ and $G_v/N_v$ is the
fundamental group of a 2-orbifold . By Guirardel's theorem $G_v$ acts non trivially on a subtree $Y_v$ of $\bar T _0$. 
Also by Theorem 1.3 of \cite{HR},
$G_v$ acts geometrically on a convex subset of $X$, $X_v$, in particular it is a CAT(0)-group. Since $N_v$ is two ended it has a finite
index subgroup isomorphic to $\mathbb Z$ which is normalized by $G_v$. It follows by Thm. 7.1, ch. II.7 of \cite{BRI-HAE} that $G_v$ has a subgroup of finite index, $G_v'$,
that contains $\mathbb Z$ as a direct factor. Since $G_v/N_v$ is a 2-orbifold group we may assume that $G_v'=F_v\times \Z$ where $F_v$ is the fundamental group
of a compact suface with non-empty boundary, so it is a free group.

It is well known by work of Croke-Kleiner \cite{C-K} that a given CAT(0) group may act geometrically on CAT(0) spaces with different boundaries. However Bowers-Ruane \cite{BR} have
shown that this does not happen for CAT(0) groups of the form $F\times \Z ^n$ with $F$ hyperbolic. Since $F_v\times \Z$ acts geometrically on $\rm{Tree}\times \R$,
$\bd X_v$ is a suspension of a Cantor set. If $h\in G_v$ is a hyperbolic element corresponding to the $\Z $  factor of $F_v\times \Z$ then $h^{\pm}$ are the suspension points
of $\bd X_v$. By lemma \ref{suspension}
any min cut of $\bd X$ separating $\bd X_v$ contains $h^{\pm}$ (if there is such a min cut). 
In the next lemma we examine  closely the structure of $\bd X_v$ and $\bd X$.

\begin{Lem}\label{cutcircle} Let $G$ be exceptional acting geometrically on the CAT(0) space $X$ and let $\Gamma $ be a Guirardel decomposition of $G$. 
Let $G_v$ be a Seifert type vertex group of $\Gamma $ acting geometrically on a convex subset $X_v$ of $X$. Then there is a collection of
geodesic circles $C_i\subseteq \bd X_v$ and continua $Z_i$ $(i\in I)$ such that

$$\bd X=\bd X_v \bigcup _{i\in I} Z_i \text { with }Z_i\cap \bd X_v= C_i.$$
Moreover if a min cut $A$ separates $\bd X_v$ then it separates exactly two of the continua $Z_i$.
\end{Lem}

\begin{proof} 
Since $G_v$ is of Seifert type it has a normal subgroup $N_v$ so that $G_v/N_v$ is the
fundamental group of a 2-orbifold. If $G_e$ is an edge group adjacent to $G_v$ 
then $G_e/N_v$ is a subgroup of $G_v/N_v$ corresponding
to a boundary component of the underlying 2-orbifold. 

As noted above $G_v$ has a subgroup of finite index of the form $F_v\times \Z$
where $F_v$ is the fundamental group
of a compact suface with boundary. $F_v\times \Z$ acts discretely on $H^2\times \R$ and $F_v$ acts co-compactly on a convex subset $Q$ of $H^2$ quasi-isometric to a tree.
So $F_v\times \Z$ acts geometrically on $Q\times \R$. If $h$ is the generator of the $\Z$-factor $\bd (Q\times \R)$ is the suspension of the Cantor set $C=\bd Q$
with suspension points $h^{\pm}$. $C$ is equiped by a natural cyclic order as $Q$ is
a convex subset of $H^2$ and $\bd Q\subseteq \bd H^2=S^1$. If $x,y\in C$ are the endpoints of a connected component of $S^1\setminus C$ we say that $\{x,y\}$ is a gap of $C$.

By \cite{BR} $\bd X_v$ is equivariantly homeomorphic to $\bd (Q\times \R)$ so we identify these two sets, and we
consider $C=\bd Q$ to be a subset of $\bd X_v$. If $G_e$ is an edge group adjacent to $G_v$ then $G_e$ is virtually $\mathbb Z^2$ so it acts geometrically on a flat plane $P_e\subseteq X_v$
(see Proposition 7.6 and Theorem 1.3 of \cite{HR}).

If $\{x,y\}$ is a gap of the Cantor set $C$ then its suspension by $\{h^{\pm}\}$ is a circle $a\subseteq \bd X_v$. We call such a circle a peripheral circle of $\bd X_v$. Any peripheral circle is equal to
$\bd P_e$ where $P_e$ is a plane stabilized by some edge group $G_e$ contained in $G_v$. 

We pick now a base point $x_0\in X_v$ and we consider infinite rays from $x_0$. Any ray either stays inside $X_v$ or crosses a unique plane $P_i$. We define the set
$Z_i$ to be the endpoints of rays crossing $P_i$ together with $C_i=\bd P_i$. If $x_n\in Z_i$ with $x_n\to x$ then the rays $r_n$ defining $r_n$ either
intersect $P_i$ in a bounded set so they converge to a ray $r$ that crosses $P_i$ or for any compact set $K\subset P_i$ almost all $r_n$ do not intersect $K$.
In this case $x\in C_i$. This shows that $Z_i$ is closed. It is also connected since $\bd X$ is connected and $C_i\subset Z_i$.

Clearly $$\bd X=\bd X_v \bigcup _{i\in I} Z_i \text { with }Z_i\cap \bd X_v= C_i.$$

Let's say that a min cut $A$ separates $\bd X_v$. By lemma \ref{suspension} $A$ contains the suspension points $h^{\pm}$ of $\bd X_v$. We note that each arc joining
$h^+,h^-$ in $\bd X_v$ is a limit of other such arcs. So $A$ necessarily separates two such arcs. We note that the set of peripheral circles $C_i$ is dense in $\bd X_v$,
so if $A$ separates $\bd X_v$ it necessarily separates the two arcs of some peripheral circle $C_i$. It follows that $A\cap Z_i$ separates $Z_i$. We note now that
$$(\bigcup _{j\ne i}Z_j)\cup \bd X_v$$ is connected and contains $C_i$ so as before $A$ since it separates $C_i$ it separates some peripheral circle $C_j$ for $j\ne i$.

The set of arcs of $\bd X_v$ joining $h^+,h^-$ inherits a circular order from $C$. As any two points seprate the circle in two arcs, any two arcs of $\bd X_v$  joining $h^+,h^-$ 
define two `intervals' in this circular order.

If $A$ separates $Z_i$ and $Z_j$ then there are arcs $\zeta _1,\zeta _2$ of $C_i$
and $\eta _1, \eta _2$ of $C_j$ separated by $A$ and we may assume that $\zeta _2,\eta _2$ are not separated by $\zeta _1,\eta _1 $ in this circular order.
Let's say that $Z_i=Z_i'\cup Z_i''$ with $Z_i'\cap Z_i''\subset A$, $Z_j=Z_j'\cup Z_j''$ with $Z_j'\cap Z_j''\subset A$ and $\zeta _1\subset Z_i', \eta _1\subset Z_j'$.

If we denote by $S_1,S_2$ the two `intervals' of the circular order defined by $\zeta _1, \eta _1$ and 
$$Z'=\overline{ (\bigcup_{s\in S_1} Z_s)\cup Z_i'\cup Z_j'},\,Z''=\overline{ (\bigcup_{s\in S_2} Z_s)\cup Z_i''\cup Z_j''}$$
then $Z=Z'\cup Z''$ and $Z'\cap Z''\subseteq A$.

\end{proof}

We have the following lemma:

\begin{Lem}\label{act} There is a Guirardel decomposition $\Gamma $ of $G$ such that either $\Gamma $ has a 2-ended edge group or 
there is some Seifert type vertex group $G_v$ of $\Gamma $ and some $g\in G_v$
acting hyperbolically on the cactus tree $T$ of $\bd X$.

\end{Lem}

\begin{proof} We argue by contradiction. We give a partial order on the set of Guirardel decompositions of $G$: $\Gamma _1\leq \Gamma _2$ if every
Seifert type vertex group of $\Gamma _1$ is contained in a conjugate of a Seifert type vertex group of $\Gamma _2$. By Bestvina-Feighn accessibility \cite{B-F} there
is a maximal element $\Gamma $ for this partial order. Let $H_1,...,H_k$ be the Seifert type vertex groups of $\Gamma $. If none of them contains
an element acting hyperbolically on $T$ each of them fixes a point of $T$. It follows by constrcution of $\bar T$ that each of the $H_i$'s fixes a point
of the tree $\bar T$ (obtained in theorems \ref{nonnest} and \ref{thick}).

By corollary 6 of \cite{Le} there is an action by isometries of $G$ on an $\mathbb R$-tree $\bar T _0$ such that each of the $H_i$'s acts elliptically
on $\bar T _0$ and arc stabilizers of $\bar T _0$ are arc stabilizers of $\bar T$.
By theorem \ref{T:Rtree} either $G$ splits over a 2-ended group or $G$ is exceptional. In particular $G$ admits a graph of groups
decomposition $\Gamma '$ with at least one vertex group $G_v$ which is of Seifert type, so it has a 2-ended normal subgroup $N_v$ and $G_v/N_v$ is the
fundamental group of a 2-orbifold. By Guirardel's theorem $G_v$ acts non trivially on a subtree $Y_v$ of $\bar T _0$. However this contradicts the maximality
of the Guirardel decomposition $\Gamma $ as $G_v$ is not contained in a conjugate of any of the $H_i$'s. This proves the lemma.
\end{proof}

Let $\Gamma $ and $G_v$ be as in the lemma \ref{act} above and let $g\in G_v$ be an element acting by translations along an interval $I$ of $ T$. We will show that this leads to a contradiction.

We distinguish two cases:
 
 \textit{Case 1.} Some min cut $A$ that separates $\bd X_v$ lies in $I$. Then $I$ is spanned by $g^nA$, $n\in \Z$. Since $g$ fixes the suspension points of $\bd X_v$ all min cuts $g^nA$ contain the
suspension points $h^+,h^-$ of $\bd X_v$. 

If a min cut $A$ separates $\bd X_v$ then by lemma \ref{cutcircle} it separates the two semicircles of some peripheral circle. 
In fact by minimality of $A$, $A$ necessarily separates exactly two peripheral circles
$a,a'$ and if $$\bd X=X_1\cup X_2,
X_1\cap X_2=a, \bd X_v\subseteq X_1$$ and
 $$\bd X=X_1'\cup X_2',
X_1'\cap X_2'=a', \bd X_v\subseteq X_1'$$
then $$A=A_1\cup A_2, A_1\cap A_2=\{h^+,h^-\}$$ and $A_1$ is a cut of $X_2$ separating $a$ while $A_2$ is a cut of $X_2'$ separating $a'$. We refer to the sets $A_1,A_2$ 
obtained by a min cut as above as half cuts. We note now that if $A,B$ are any two min cuts separating $\bd X_v$ with corresponding half cuts $A_1,A_2$ and $B_1,B_2$
then $A_i\cup B_j$ is a finite cut of $\bd X$ for any $i,j$. It follows that $|A_1|=|A_2|$ and $A,B$ lie in the same wheel of $\bd X$.

It follows that all min cuts $g^nA$ lie in the same maximal wheel $W$ of $\bd X$. Moreover $gW=W$. Therefore the interval $I$
does not contain any min cut separating $\bd X_v$. \smallskip

  \textit{Case 2.} No min cut $A$ that separates $\bd X_v$ lies in $I$.
 
Since $g\in G_v$, if we denote by $g^{\pm}$ the attractive-repelling points of $g$ we have $g^{\pm}\in \bd X_v$, so for any min cut $A$ lying
on $I$ $g^+,g^-$ are not separated by $A$. As $G$ is rank 1 there is a point $x\in \bd X$ separated from  $g^{\pm}$ by $A$. It follows by $\pi$-convergence that
$g^nx$ and $g^{-n}x$ converge to $g^+,g^-$ respectively as $g\to \infty $. However this is not possible as $g$ acts by translations on $I$ so at least one of $g^nx,g^{-n}x$
is separated from both $g^+,g^-$ by $A$ for every $n$. We conclude that $G$ is not exceptional, so it splits over a 2-ended group.

\end{proof}

\end{document}
\]
\]
\end{align*}
\]
\]

\end{document}